\documentclass{article}
\usepackage{
amsmath,
amsfonts,
latexsym, 
amssymb
}

\usepackage[non-sorted-cites]{amsrefs}
\usepackage{hyperref}
\usepackage{bm}
\usepackage{cleveref}

\usepackage{tocloft}
\usepackage{pstricks}
\usepackage{subcaption}

\usepackage{tikz}
\usepackage{tikz,fullpage}
\usetikzlibrary{arrows,%
                petri,%
                topaths, automata}%
\usepackage{floatrow}
\usepackage{enumerate}
\usepackage{mathrsfs}
\usepackage{wrapfig}
\usepackage{color}

\usepackage{graphicx}
\def\cB{{\mathcal B}}

\newcommand{\dif}{{\mathrm{d}}}

\numberwithin{equation}{section}
\newcommand{\bla}{\bm{\lambda}}
\newcommand{\bL}{\bm{L}}
\newcommand{\bRR}{\bm{R}}
\newcommand{\btheta}{\bm{\theta}}

\newcommand{\bmu}{\bm{u}}
\newcommand{\bmw}{\bm{w}}

\newcommand{\bmr}{\bm{r}}
\newcommand{\bms}{\bm{s}}

\newcommand{\rd}{{\rm d}}

\newcommand{\bR}{{\mathbb R}}
\newcommand{\D}{\mathbb{D}}
\newcommand{\bZ}{{\mathbb Z}}

\newcommand{\bx}{{\bf{x}}}
\newcommand{\by}{{\bf{y}}}

\newcommand{\bv}{{\bf{v}}}
\newcommand{\bw}{{\bf{w}}}

\newcommand{\al}{\alpha}

\newcommand{\bH}{{\mathbb H}}

\newcommand{\be}{\begin{equation}}
\newcommand{\ee}{\end{equation}}

\newcommand{\e}{{\varepsilon}}

\newcommand{\td}{\tilde}

\newcommand{\cL}{{\cal L}}

\newcommand{\cS}{{\mathcal S}}
\newcommand{\br}{{\bf{r}}}

\newcommand{\cU}{{\mathcal U}}

\newcommand{\cD}{{\mathcal D}}

\newcommand{\ri}{{\rm i}}
\renewcommand{\Im}{{\rm Im}}
\renewcommand{\Re}{{\rm Re}}
\newcommand{\sfT}{{\mathsf T}}

\newcommand{\ft}{{\mathfrak t}}
\newcommand{\dist}{{\rm dist}}

\newcommand{\del}{\partial}

\usepackage{amsmath} 
\usepackage{amssymb}
\usepackage{amsthm}

\renewcommand{\cal}{\mathcal}

\newcommand{\wh}{\widehat}
\newcommand{\wt}{\widetilde}

\newcommand{\ii}{\mathrm{i}} 

\newcommand{\deq}{\mathrel{\mathop:}=}

\renewcommand{\epsilon}{\varepsilon}
\renewcommand{\leq}{\leqslant}
\renewcommand{\geq}{\geqslant}

\renewcommand{\le}{\leq}
\renewcommand{\ge}{\geq}

\renewcommand{\P}{\mathbb{P}}
\newcommand{\E}{\mathbb{E}}
\newcommand{\R}{\mathbb{R}}
\newcommand{\C}{\mathbb{C}}
\newcommand{\N}{\mathbb{N}}
\newcommand{\Z}{\mathbb{Z}}

\newcommand{\abs}[1]{\lvert #1 \rvert}

\def\XXint#1#2#3{{\setbox0=\hbox{$#1{#2#3}{\int}$ }
\vcenter{\hbox{$#2#3$ }}\kern-.6\wd0}}

\DeclareMathOperator{\supp}{supp}

\DeclareMathOperator{\im}{Im}
\DeclareMathOperator{\dom}{\mathcal{D}}

\DeclareMathOperator{\sgn}{sgn}

\DeclareMathOperator{\OO}{O}
\DeclareMathOperator{\oo}{o}

\theoremstyle{plain} 
\newtheorem{theorem}{Theorem}[section]
\newtheorem*{theorem*}{Theorem}
\newtheorem{lemma}[theorem]{Lemma}
\newtheorem*{lemma*}{Lemma}
\newtheorem{corollary}[theorem]{Corollary}
\newtheorem*{corollary*}{Corollary}
\newtheorem{proposition}[theorem]{Proposition}
\newtheorem*{proposition*}{Proposition}
\newtheorem{assumption}[theorem]{Assumption}

\newtheorem*{definition*}{Definition}

\newtheorem*{example*}{Example}
\newtheorem{remark}[theorem]{Remark}

\newtheorem*{remark*}{Remark}
\newtheorem*{remarks*}{Remarks}

\makeatletter
\renewcommand{\subsection}{\@startsection
{subsection}
{2}
{0mm}
{-\baselineskip}
{0 \baselineskip}
{\normalfont\bf\itshape}} 
\makeatother

\makeatletter
\renewcommand{\subsubsection}{\@startsection{subsubsection}{3}{\z@}%
  {3.25ex \@plus 1ex \@minus .2ex}{-1em}{\normalfont\normalsize\itshape}}
\makeatother

\usepackage{dsfont}
\usepackage{stmaryrd}

\def\bR{{\mathbb R}}
\def\bZ{{\mathbb Z}}

\def\@empty{}

\def\author#1{\par
    {\centering{\authorfont#1}\par\vspace*{0.05in}}
}

\def\titlefont{\fontsize{13}{15}\bfseries\boldmath\selectfont\centering{}}
\def\authorfont{\fontsize{13}{15}}
\def\abstractfont{\fontsize{8}{10}}

\let\affiliationfont\rhfont

\def\address#1{\par
    {\centering{\affiliationfont#1\par}}\par\vspace*{11pt}
}

\def\body{
\setcounter{footnote}{0}
\def\thefootnote{\alph{footnote}}
\def\@makefnmark{{$^{\rm \@thefnmark}$}}
}

\def\title#1{
    \thispagestyle{plain}
    \vspace*{-14pt}
    \vskip 79pt
    {\centering{\titlefont #1\par}}%
    \vskip 1em
}

\renewenvironment{abstract}{\par%
    \vspace*{6pt}\noindent 
    \abstractfont
    \noindent\leftskip10pt\rightskip10pt
}{%
  \par}

\usepackage{tocloft}

\makeatletter
\renewcommand{\section}{\@startsection
{section}
{1}
{0mm}
{-2\baselineskip}
{1\baselineskip}
{\normalfont\large\scshape\centering}} 
\makeatother

\begin{document}

~\vspace{-1cm}

~\hspace{-0.3cm}

\title{
Fluctuations for non-Hermitian dynamics}

\vspace{1cm}

\noindent\begin{minipage}[b]{0.33\textwidth}

 \author{Paul Bourgade}

\address{Courant Institute,\\ New York University\\
   bourgade@cims.nyu.edu}

 \end{minipage}
\noindent\begin{minipage}[b]{0.33\textwidth}
 \author{Giorgio Cipolloni}

\address{Mathematics department,\\  University of Rome Tor Vergata \\
  cipolloni@axp.mat.uniroma2.it}

 \end{minipage}
\noindent\begin{minipage}[b]{0.33\textwidth}

 \author{Jiaoyang Huang}

\address{Statistics department, Wharton,\\ University of Pennsylvania\\
   huangjy@wharton.upenn.edu}

 \end{minipage}

\begin{abstract}
We prove that under the Brownian evolution on large non-Hermitian matrices
the log-determinant converges in distribution to a 2+1 dimensional Gaussian field in the Edwards-Wilkinson regularity class,
namely it is logarithmically correlated  for the parabolic distance.
This dynamically extends a seminal result by Rider and Vir{\' a}g 
about convergence to the Gaussian free field.  The convergence holds out of equilibrium for centered, i.i.d.  matrix entries as an initial condition.

A remarkable aspect of the limiting field is its non-Markovianity,  due to  long range correlations of the eigenvector overlaps, for which we identify 
the exact space-time polynomial decay.

In the proof, we obtain a quantitative  relaxation at the hard edge,  for an extension of the Dyson Brownian motion, with a driving noise arbitrarily correlated in space.
\noindent 
\end{abstract}

\vspace{-0.5cm}

\tableofcontents

\vspace{0.2cm}
\noindent

\section{Introduction}
 
\subsection{Random matrices and logarithmically correlated fields, dimensions one and two.}\ The eigenvalues  of large random matrices exhibit universal,
anomalous small fluctuations. These are log-correlated and originate from 
equal superposition of randomness on all length scales.  
This phenomenon first appeared in dimension one, the prominent example being 
Haar-distributed unitary random matrices of size $N$, with eigenangles $\btheta$: If $f,g$ are smooth with mean zero on the torus,  then $\sum f(\theta_k)$ and $\sum g(\theta_k)$ converge 
with no normalization to Gaussian random variables $\ell_f,\ell_g$ as $N\to\infty$,  and
\begin{equation}\label{eqn:1d}
{\rm Cov}(\ell_f,\ell_g)=-\frac{2}{\pi^2}\iint f'(\theta)g'(\varphi)\log|e^{\ii\theta}-e^{\ii\varphi}|\rd \theta\rd \varphi.
\end{equation}
This limiting covariance was derived first by Dyson and Mehta \cite{DysMeh1963}.
It has then appeared for multiple matrix models in dimension one,  L-functions and high genus hyperbolic surfaces, 
often with techniques from integrable systems,  representation theory,  
mathematical physics (loop equations) and probability theory (moments method). We refer to
the end of this introduction for a partial review of this immense literature.

In dimension two,  logarithmic correlations appeared much later,  first in the context of time-dependent 1d spectra in Spohn \cite{Spo1986}.   Consider, for example, $B$ a Brownian motion on $N\times N$ Hermitian matrices,  or more precisely its 
Ornstein-Uhlenbeck version $H$ which has the Gaussian unitary ensemble as equilibrium,  and the Dyson Brownian motion dynamics  induced on its  spectrum $\bla$:
\begin{equation}\label{eqn:HerBM}
\rd H_t=\frac{\rd B_t}{\sqrt{N}}-\frac{1}{2}H_t\rd t,\ \ \ \ \ \ \ \ \ \  \rd \lambda_k(t)=\frac{\rd b_{k}(t)}{\sqrt{N}}+\Big(\frac{1}{N}\sum_{\ell\neq k}\frac{1}{\lambda_k(t)-\lambda_\ell(t)}-\frac{1}{2}\lambda_k(t)\Big)\rd t,
\end{equation}
where $b_1,\dots,b_n$ are independent standard Brownian motions.  Log-correlations for these dynamics hold in the following sense: For any centered,  smooth $f$
supported in the bulk of the spectrum,  the linear statistics ${\rm Tr}f(H_s)$  converge in distribution in the limit of large dimension,  jointly in $f$ and $s$, to 
Gaussian random variables $\ell_{f,s}$ with covariance

\begin{equation}\label{eqn:dyn2d}
{\rm Cov}(\ell_{f,s},\ell_{g,t})= \frac{2}{\pi^2}\iint f'(x)g'(y)\, k((x,s),(y,t))\rd x\rd y\ \ \ {\rm where}\ \ k(u,v)  \underset{u\to v}{\sim}-\log |u-v|.
\end{equation}
Here $k$ is an explicit kernel which is smooth off the diagonal,   and $|\cdot|$ is the Euclidean distance in $\mathbb{R}^2$.

The derivation of this limiting 2d log-correlated field relies on the self-adjointness of $H$: It makes the eigenvalues dynamics in (\ref{eqn:HerBM}) autonomous, in the sense that it does not depend on the eigenvectors.  Actually the Dyson Brownian motion also coincides with Brownian motions conditioned not to intersect,  so from the Karlin-McGregor formula the spectrum $(\bla_t)_{t\geq 0}$ forms a determinantal point process,  which is fully integrable.

Another natural  context for bidimensional random matrix theory is non-Hermitian models.  The paradigmatic example is the Ginibre ensemble,  an $N\times N$  random matrix $G$ with independent complex Gaussian entries with covariance ${{\bf 1}_2}/(2N)$\footnote{Here ${{\bf 1}_2}$ denotes the $2\times 2$ identity matrix.}.  The eigenvalues for this model also form a determinantal point process,  
which implies the following analogue of  (\ref{eqn:dyn2d}),  by Rider and Vir{\'a}g \cite{RidVir2007}: 
For smooth enough $f,g$ supported  in the unit disk, 
${\rm Tr} f(G)$ and ${\rm Tr} g(G)$ converges to Gaussian random variables $L_f,L_g$ with covariance
\begin{equation}\label{eqn:stat2d}
{\rm Cov}(L_f,L_g)=\int \nabla f(z)\cdot\nabla g(z)\frac{\rd z}{4\pi}
=-\frac{1}{8\pi^2}\iint \Delta f(z) \Delta g(w)\log|z-w|\, \dif z\dif w,
\end{equation}
where $\rd z$ is the Lebesgue measure on $\mathbb{C}$.
For $f=\log|v-\cdot|$ we have $\Delta f=2\pi\delta_{v}$ in distribution, so the above central limit theorem means distributional convergence of $(\log|\det(G-z)|)_{|z|<1}$ to a $2d$ log-correlated field.

\subsection{Result.}\ The Ginibre ensemble is the stationary distribution for the non-Hermitian Ornstein-Uhlenbeck process.  At the matrix and eigenvalues level these dynamics are
\begin{equation}\label{eq:OU}
\rd X_t=\frac{\rd B_t}{\sqrt{N}}-\frac{1}{2}X_t\rd t,\ \ \ \ \ \ \ \ \ \ \rd\sigma_i(t)=\frac{\rd M_i(t)}{\sqrt{N}}-\frac{1}{2}\sigma_i(t)\rd t,
\end{equation}
where the entries of $B$ are independent complex standard Brownian motions, i.e. of type $\frac{1}{\sqrt{2}}(B_1+\ii B_2)$ with $B_1,B_2$ independent standard real Brownian motions.  Here the martingales $M_i$ are correlated, with joint bracket
$
\rd\langle M_i,M_j\rangle_t=\mathscr{O}_{ij}(t)\rd t$,    $\mathscr{O}_{ij}= (R_j^* R_i)(L_j^* L_i),
$
where the left and right eigenvectors $(L_i,R_i)_{i=1}^N$ of $X_t$ are a biorthogonal basis sets, normalized by
$
L_i^tR_j=\delta_{ij}.
$
Therefore,  contrary to the Hermitian setting (\ref{eqn:HerBM}),  
the dynamics for non-Hermitian spectra fundamentally differs from the Langevin dynamics for the log-gas.  In particular the evolution is now non-autonomous,  and the eigenvalues distribution is determinantal only at equilibrium, for one fixed time.  Nevertheless,  our result below gives $3d$-logarithmic correlations for these non-Hermitian dynamics,  now with respect to the parabolic distance
\begin{equation}
\label{eq:defpardist}
d((z,s),(w,t))=(|z-w|^2+|t-s|)^{1/2}.
\end{equation}

\begin{theorem*} Consider (\ref{eq:OU}) at equilibrium and let $f,g$ be smooth functions supported in the open unit disk $\mathbb{D}$. Then 
for any fixed $s,t$,  as $N\to\infty$  the linear statistics
${\rm Tr} f(G_s)-N\int \frac{f}{\pi}$ and ${\rm Tr} g(G_t)-N\int \frac{g}{\pi}$
converge jointly  to Gaussian random variables $L_{f,s},L_{g,t}$,  which are centered  with covariance
\begin{align}\label{eqn:cov3}
\mathbb{E}[L_{f,s}L_{g,t}]&=\frac{1}{16\pi^2}\iint \Delta f(z) \Delta g(w)K(z,w,|t-s|)
\dif z\dif w,\\
K(z,w,\tau)&:=-\log\big((1-e^{-\tau})(1-|z|^2)+ |z-e^{-\tau/2}w|^2\big).\label{eqn:kernel}
\end{align}
\end{theorem*}
\noindent One can easily check that the above kernel $K$ is symmetric in $z,w$.  Moreover it satisfies
$$
K(z,w,t-s)\asymp-\log d((z,s),(w,t))
$$
as $(z,s)\to(w,t)$, uniformly in compact subsets of $\mathbb{D}\times \mathbb{R}$.

The above theorem follows from a general decomposition theorem which allows to treat out of equilibrium dynamics with i.i.d. initial condition,  test functions overlapping the edge of the spectrum  and mesoscopic scales.  We refer to Section \ref{sec:Main} for these results. There,  we will also give consequences on the space-time correlations between overlaps  of non-Hermitian random matrices, and the main ideas of the proof,   circumventing the lack of autonomous spectral evolution and determinantal representations.\\

\subsection{Related literature.}\ We conclude this introduction with more context on logarithmically correlated fields, in random matrix theory and other settings.\\

\noindent In dimension one,  log-correlations 
are now considered a core result in random matrix theory,  as manifested by the books \cite[Chapter 16]{Mehta} and  \cite[Chapter 14]{For2010}.
Formulas similar to (\ref{eqn:1d}) hold for a multitude of  matrix models,   including Wigner matrices \cite{LytPas2009} and their deformations \cite{LiSchXu2021},  covariance matrices \cite{BaiSil2004},  products of random matrices \cite{GorSun2022},  band matrices \cite{AndZei2006},  random graphs \cite{DumJohPalPaq2013}, free sums of random matrices \cite{BaoSchXu2022}.  
This covariance also appears in the limit of large dimension for general log-gases \cite{Joh1988},  determinantal point processes \cite{BreDui2017},  and particle systems emerging from special functions \cite{BorGorGui2017}.    The most common techniques towards such results are the method of moments \cite[Chapter 2]{AndGuiZei2010},  Schur-Weyl duality \cite{DiaSha},   loop equations \cite{Joh1998}, 
transport maps \cite{Shc2014,BekFigGui2015,BekLebSer2018} and the analysis of Toeplitz determinants
\cite{HugKeaOCo}.
Moreover,  in a few cases,  the distributional convergence to a logarithmically correlated field has been upgraded to pointwise convergence \cite{Bou2010,BouModPai2022}.
In completely different contexts,  the covariance (\ref{eqn:1d}) was also exhibited for linear statistics related to zeros of L-functions \cite{BouKua2014,Rod2014} and eigenvalues of the Laplacian on random hyperbolic surfaces \cite{Rud2023}.  It is impossible to mention all such contributions and we refer to \cite{For2023} for a recent review.\\

\noindent In dimension two,  Spohn's original space-time fluctuations (\ref{eqn:dyn2d}) were first proved for non-intersecting Brownian motions at equilibrium \cite{Spo1986},  and then extended to arbitrary temperature on the circle,  still at equilibrium \cite{Spo1998}. 
The space-time fluctuations of Dyson's Brownian motion are now understood out of equilibrium,  on any mesoscopic scale \cite{DuiJoh2018,HuaLan2019,AdhHua2018}.
The $2d$ Gaussian free field also arises in discrete setting, often in connection with non-intersecting paths; important examples include 
domino tilings \cite{Ken2001},  
dimer models \cite{Ken2008,Dub2015},
and many more discrete particle systems, see e.g.  \cite{GorHua2022}.

Still in dimension two, but in the non-Hermitian setting, the Rider-Vir{\'a}g central limit theorem (\ref{eqn:stat2d}) is actually more general as it allows functions overlapping the edge of the spectrum. 
It fully verified a prediction from \cite{For1999}, after partial results in \cite{Rid2004,Rid2004bis} for radial or angular functions.
It has been generalized in three major directions,  
first to arbitrary normal random matrices, i.e.   to non-quadratic confining external potential \cite{AmeHedMak2011,AmeHedMak2015}, then to $2d$ Coulomb gases at arbitrary temperature
\cite{LebSer2018,BauBouNikYau2019}, and finally to
non-Hermitian random matrices with arbitrary independent entries \cite{CipErdSch2023,CipErdSch2021,CipErdSch2022bis}.\\

\noindent In dimension three,  before (\ref{eqn:cov3}) the only example connected to random matrices came from
minors of Hermitian matrices
along the matrix dynamics (\ref{eqn:HerBM}),  thanks to the method of moments   \cite{Bor2014}.  The covariance kernel of the limiting $3d$ Gaussian field depends on the spectral variable, the size of the submatrix, and time.  It is log-correlated for the Euclidean distance in $\mathbb{R}^3$.  A similar $3d$ limiting field was then proved in connection to representations of $U(\infty)$ \cite{BorBuf2014}, which is also  log-correlated for the ${\rm L^2}$ norm.

By contrast, non-Hermitian dynamics present the same singular behavior as the so-called Edwards-Wilkinson fluctuations in dimension 2+1, i.e.  logarithmic correlations for the parabolic distance.  The canonical model for these fluctuations is the $2d$
additive stochastic heat equation \cite{Hai2009}.  Despite the conjectured breadth of the  Edwards-Wilkinson universality class,  only a few models have been proved to be part of it.

In the continuous setting,  the following stochastic partial differential equations have been shown to exhibit Edwards-Wilkinson fluctuations in the weak coupling regime:  the multiplicative stochastic heat equation (or the directed polymer model in random environment) \cite{CarSunZyg2017, DunGra2024},
and the KPZ equation   driven by space-time white noise,  may it be isotropic subcritical \cite{CarSunZyg2020} or anisotropic \cite{CanErhTon2023}.
Langevin dynamics for Liouville quantum gravity measures also present log-correlated fluctuations for the parabolic distance \cite{Gar2020}.

In the discrete setting,   fluctuations for the Ginzburg-Landau interface model were shown to converge to a $2d$
additive stochastic heat equation,  at equilibrium \cite{GiaOllSpo2001}. 
Other stochastic growth models in 2 + 1 dimensions, that belong to the anisotropic KPZ class,  fluctuate like the additive stochastic heat equation,  for specific initial conditions.
This was first proved through convergence to the free field at fixed time \cite{BorFer2014},  and then extended to space-time correlations \cite{BorCorFer2018,BorCorTon2017}. \\

\noindent As described above,  Edwards-Wilkinson fluctuations have been proved mostly at equilibrium or for a restricted range of initial conditions.  Theorem \ref{theo:mainres} below works out of equilibrium,  for arbitrary initial random matrix with centered i.i.d.  entries.  In fact,
the methods we rely on should be robust enough to cover broad classes of deterministic initial data, including the possibility of time-dependent  hydrodynamic profiles. 
The assumptions of centered independent entries is for convenience and it gives a particularly simple formula for correlations of the limiting field (\ref{eqn:kernel}).  In the bulk of the spectrum,  the local singularity of the limiting kernel should be independent of the  initial condition, only its long range behavior may change.

The obtained fluctuations for non-Hermitian dynamics raise other questions, for which some techniques developed in this paper may apply.  
In view of the singularity of the covariance  (\ref{eqn:kernel}),  it would be interesting to develop branching methods for logarithmically-correlated fields for the parabolic distance,
to understand extreme statistics.  In particular the Fyodorov-Hiary-Keating analogies \cite{FyoHiaKea12,FyoKea2014} ($d=1$) and the universality proved in \cite{DinRoyZei} (any dimension) 
give broad classes of models with the same extreme values,  in the case of log-correlations for the Euclidean distance. 
Similarly,  the  recent convergences of random matrix statistics to Gaussian multiplicative chaos hold with respect to the ${\rm L}^2$ distance,  in dimension 1 \cite{Web2015, NikSakWeb20,ChaFahWebWon} and 2 \cite{BouFal2022}.

Finally,  we note that there is no first principle explanation for logarithmic correlations for so many random matrix models,  despite the enormous literature on this topic in dimensions 1 and 2, and the new occurrence in dimension 3 proved in this paper.

\subsection{Acknowledgments.}\ P.\ B.\ was supported by the NSF standard grant DMS-2054851.  J.\ H.\ was supported by  NSF standard grant DMS-2331096 and career grant DMS-2337795, and the Sloan research award. G. C. is partially supported by the MUR Excellence Department Project MatMod@TOV awarded to the Department of Mathematics, University of Rome Tor Vergata, CUP E83C18000100006. We also thank the referees for their excellent review work.

\section{Main Results.}\label{sec:Main}

\subsection{Multi-time central limit theorems for eigenvalues.}\  Our initial condition consists in  $N\times N$ non-Hermitian matrices $X$ with i.i.d. complex entries, i.e. $x_{ab}\stackrel{d}{=}N^{-1/2}\chi$, with $\chi$ satisfying the following assumptions.
\begin{assumption}
\label{ass:1}
The random variable $\chi$ is centered and has unit variance, i.e. $\E\chi=0$, $\E |\chi|^2=1$, and we also assume that $\E \chi^2=0$. Additionally,  high moments of $\chi$ are finite,  i.e.  there exist a constants $C_p>0$, for any $p\in\N$, such that
\begin{equation*}
\E |\chi|^p\le C_p.
\end{equation*}
\end{assumption}

We now consider the matrix dynamics (\ref{eq:OU}), where we remind that
$B_t$ is a matrix whose entries are i.i.d.  standard complex Brownian motions. Note that
\[
X_t\stackrel{(\dif)}{=} e^{-t/2}X+\sqrt{1-e^{-t}}\widetilde{X},
\]
with $\widetilde{X}$ being a complex Ginibre matrix independent of $X$. The first two moments of $X_t$ are preserved along the flow,  and the Ginibre ensemble is the unique equilibrium for these dynamics.  Denoting the eigenvalues of $X_t$ by $\{\sigma_i(t)\}_{1\leq i\leq N}$,  our main interest is to study the space--time correlation of the linear statistics
\begin{equation}
\label{eq:linstattime}
L_N(f,t)=\sum_{i=1}^N f(\sigma_i(t))-\E\sum_{i=1}^N f(\sigma_i(t)).
\end{equation}
Here $f$ is a test function supported on $\Omega$, a fixed disk with center $0$ and arbitrary radius greater than $1$.  Without loss of generality, by polarization, it is enough to consider real test functions $f$, which will be assumed to be in the Sobolev space ${\rm H}_0^{2+\varepsilon}(\Omega)$, for a fixed small $\varepsilon>0$, which is defined as the completion of the smooth compactly supported functions $C_c^\infty(\Omega)$ under the norm
\[
\lVert f\rVert_{{\rm H}^{2+\varepsilon}(\Omega)}=\lVert (1+|\xi|)^{2+\varepsilon}\widehat{f}(\xi)\rVert_{{\rm L}^2(\Omega)},
\]
where $\widehat{f}$ denotes the Fourier transform of $f$. Furthermore, for $h$ defined on the boundary of the unit disk $\partial\mathbb{D}$,  our convention for its Fourier transform is
\[
\widehat{h}_k=\frac{1}{2\pi}\int_0^{2\pi} h(e^{\ii\theta}) e^{-\ii\theta k}\, \dif \theta, \qquad k\in\mathbb{Z},
\]
and we denote
\begin{equation*}
\langle g,f\rangle_{{\rm H}^{1/2}(\partial\mathbb{D})}=\sum_{k\in\Z}|k| \widehat{f}_k\overline{\widehat{g}_k}, \qquad \|f\|_{{\rm H}^{1/2}(\partial\mathbb{D})}^2=\langle f,f\rangle_{{\rm H}^{1/2}(\partial\Omega)}.
\end{equation*}

We now first state the CLT for linear statistics \eqref{eq:linstattime} with macroscopic test functions in Theorem~\ref{theo:mainres}, and then we present the statement of the mesoscopic case separately in Theorem~\ref{theo:mainresmeso}. Before stating these results we introduce some useful short--hand notations. We define the averages
\begin{equation*}
\langle f\rangle_\D:=\frac{1}{\pi}\int_\D f(z)\, \dif z,\qquad\quad \langle f\rangle_{\partial\D}:=\frac{1}{2\pi}\int_0^{2\pi} f(e^{\ii\theta})\, \dif\theta.
\end{equation*}
With the notation (\ref{eqn:kernel}) for the limiting covariance kernel $K$ at equilibrium in the bulk of the spectrum,  and defining $P_t f(z)=\sum_{\mathbb{Z}} e^{-|k|\frac{t}{2}}\,\widehat{f}_k z^k$ for  the Poisson kernel, we denote
\begin{align}
\Gamma(f,g,\tau,\kappa)=\frac{1}{\pi^2}&\int_{\D^2} \partial_{\overline z} f(z)\partial_{w} g(w)\partial_{z}\partial_{\overline w}K(z,w,\tau)\, \dif z\dif w
+\frac{1}{2}\langle f,P_\tau g \rangle_{H^{1/2}(\partial\mathbb{D})}\notag \\
&+ \kappa e^{-\tau}\big(\langle f\rangle_\D-\langle f\rangle_{\partial\D}\big)\big(\langle g\rangle_\D-\langle g\rangle_{\partial\D}\big).\label{eq:defcov}
\end{align}
In the following main result,  $\kappa_{4,t}$ is the fourth cumulant of the matrix entries after time $t$:
\[
\kappa_{4,t}=\E|\chi_t|^4-2\ \ \ {\rm where}\ \ \ \chi_t\stackrel{\mathrm{(d)}}{=}e^{-t/2}\chi+\sqrt{1-e^{-t}}g,
\]
with $g$ a standard complex Gaussian random variable independent of $\chi$, and  $\chi$ satisfying Assumption~\ref{ass:1}.

\begin{theorem}[Macroscopic CLT]
\label{theo:mainres}
Let $X_t$ be the solution of \eqref{eq:OU}, with $X$ satisfying Assumption~\ref{ass:1}.   Consider real valued $f,g\in H_0^{2+\varepsilon}(\Omega)$.  Then for any fixed $s\leq t$, 
$L_N(f,t),L_N(g,s)$ converge jointly in distribution to centered Gaussian random variables $(L(f,t)$,  $L(g,s))$ with covariance
\begin{equation}
\label{eq:secmom}
\E |L(f,t)|^2=\Gamma(f,f,0,\kappa_{4,t}), \qquad \E L(f,t) L(g,s)=\Gamma(f,g,t-s,\kappa_{4,s}), \qquad \E |L(g,s)|^2=\Gamma(g,g,0,\kappa_{4,s}).
\end{equation}
Additionally,  we have the decomposition
\begin{equation}
\label{eq:dec1}
L_N(f,t)=L_{N,1}(f,s,t)+L_{N,2}(f,s,t),
\end{equation}
with $L_{N,1}(f,s,t)$ depending only on $X_s$, and $L_{N,1}(f,s,t)$, $L_{N,2}(f,s,t)$ converging to independent Gaussian  random variables $\mathcal{L}_1(f,s,t), \mathcal{L}_2(f,s,t)$ with
\begin{equation}
\label{eq:dec2}
\E|\mathcal{L}_1(f,s,t)|^2= \Gamma(f,f,2(t-s),\kappa_{4,s}), \qquad \E |\mathcal{L}_2(f,s,t)|^2=\Gamma(f,f,0,0)- \Gamma(f,f,2(t-s),0).
\end{equation}
\end{theorem}

\noindent The above theorem and Equation (\ref{eq:secmom}) naturally extend to an arbitrary fixed number of test functions and times.  In particular (\ref{eq:dec2}) 
identifies the limiting distribution of the field through its increments,  and indeed the proof relies on the decomposition (\ref{eq:dec1}).

Moreover,   for $\kappa_{4,t}=0$ and functions supported in the bulk of the spectrum one recovers the special case  (\ref{eqn:cov3}) stated in the introduction.  Theorem \ref{theo:mainres} also
generalizes the static result from \cite{CipErdSch2023},  as \eqref{eq:defcov} gives
\[
\Gamma(f,f,0,\kappa_4)=\frac{1}{4\pi}\int_\D \big|\nabla f\big|^2\, \dif z+\frac{1}{2}\lVert f\rVert_{{\rm H}^{1/2}(\partial\D)}^2+\kappa_4\left|\frac{1}{\pi}\int_\D f(z)\, \dif z-\frac{1}{2\pi}\int_0^{2\pi} f(e^{\ii\theta})\rd\theta\right|^2
\]
which agrees with \cite[Equation (2.6)]{CipErdSch2023}.  Note also that precise asymptotics are known for the centering term $\E\sum_{i=1}^N f(\sigma_i(t))$ in (\ref{eq:linstattime}), see \cite[Equation (2.8)]{CipErdSch2023}.

We now state the mesoscopic version of Thereom~\ref{theo:mainres}. For this purpose given a function $f$ we define its rescaled version around a point $v\in \C$ as
\begin{equation}
\label{eq:resfunc}
f_{v,a}(z)=f(N^a(z-v)), \qquad a\in (0,1/2).
\end{equation}
The following result shows that after a proper time rescaling,  a universal Gaussian limiting field emerges,  with a simpler covariance structure denoted by
\begin{equation}
\label{eq:defcovmeso}
\Gamma_v(f,g,\tau)=-\frac{1}{\pi^2}\int_{\C^2}\partial_{\overline z} f(z)\partial_{w} g(w)\partial_{z}\partial_{\overline w}\log\big(\tau(1-|v|^2)+ |z-w|^2\big)\, \dif z\dif w.
\end{equation}

\begin{theorem}[Mesoscopic CLT]
\label{theo:mainresmeso}
Fix $T\ge 0$, a small $c>0$, $|v|\le 1-c$, $a\in (0,1/2)$, $\Omega\subset\C$ an open set, and let $X_t$ be as above. Consider real valued $f,g\in H_0^{2+\varepsilon}(\Omega)$ and let  $s_a=T+ n^{-2a}s$, $t_a=T+ n^{-2a}t$ with $0\le s\le t \lesssim 1$  fixed.  Then $(L_N(f_{v,a},t_a), L_N(g_{v,a},s_a))$ converge jointly in distribution to centered Gaussian random variables  $L(f,t),L(g,s)$ with covariance 
\[ \E |L(f,t)|^2=\Gamma_{v}(f,f,0), \ \ \E L(f,t) L(g,s)=\Gamma_{v}(f,g,t-s), \ \ \ \E |L(g,s)|^2=\Gamma_{v}(g,g,0).\]
Additionally,  for $L_N(f_{v,a},t_a)$ we have decomposition and limiting increments similar to \eqref{eq:dec1},  \eqref{eq:dec2}.
\end{theorem}

Note that the covariance of the limiting field on mesoscopic scales in \eqref{eq:defcovmeso} does not depend on the fourth cumulant of the entries of $X_t$,  contrary to the macroscopic case \eqref{eq:defcov}.

\begin{remark}[Mixing and distance to the edge] From the prefactor $1-|v|^2$ in (\ref{eq:defcovmeso}),  relaxation of the dynamics takes longer close to the edge of the spectrum. This is
reminiscent of \cite[(1.16)]{BouDub2020} which states that at equilibrium individual eigenvalues exhibit diffusive scaling, with quadratic variation proportional to the distance to the edge.
\end{remark}

\begin{remark}[Averaging by stereographic projection]\label{rem:stereo} The covariance (\ref{eq:defcovmeso}) can be made explicit:
\begin{equation*}
\Gamma_v(f,g,\tau)=\frac{1}{\pi}\int_{\C^2}\partial_{\overline z} f(z)\partial_{w} g(w) q_r(z-w) \dif z\dif w, \ \ \ q_r(z)=\frac{r}{\pi(r+|z|^2)^2},\ \ \ r=\tau(1-|v|^2).
\end{equation*}
We have  $q_r(z)\rd z\to\delta_0$ as $r\to 0$, so one recovers (\ref{eqn:stat2d}).  Moreover $q_r(z)\rd z$ is the pushforward,  by stereographic projection,  of the uniform probability measure on $\mathscr{S}_{r}$  (the 2-sphere with center $0$ and radius $r$).  Denoting $(Q_r f)(z)=\int f(z-w)q_r(w)\rd w$ the averaging by this stereographic kernel,  
 Theorem \ref{theo:mainresmeso} means
\[
{\rm Cov}(L(f,s), L(g,t))=\frac{1}{\pi^2}\int_{\C}\partial_{\overline z} f(z)\partial_{z} (Q_r g)(z) \dif z={\rm Cov}(L_f,L_{Q_{r} g}),\ \ \ r=|t-s|(1-|v|^2),
\]
relating the covariance of  the time-dependent limiting field in  Theorem \ref{theo:mainresmeso}, to the static field $(L_f)$ from (\ref{eqn:stat2d}).

However,  $(Q_r)_{r\geq 0}$ is not a semigroup ($Q_{r_1+r_2}\neq Q_{r_1}Q_{r_2}$), so that the limiting field $(L(f,s))_{f,s}$ is not Markovian with respect to time,  see Section \ref{eqn:overlaps}.
\end{remark}

\begin{remark}[Existence of Gaussian fields]\label{rem:existence}
The mesoscopic covariance $\Gamma_v(f,g,\tau)$ from Theorem \ref{theo:mainresmeso} can be directly proved to be positive definite, i.e.  
$
\sum_{1\leq i,\,j\leq m}\Gamma_v (f_i,f_j,|t_i-t_j|)\geq 0
$
for any $m\geq 1$,  functions $f_i$ and times $t_i$.  For $m=2$ this follows easily from
Schwarz's and then Young's inequality:
\[
\|f\|^2+\|g\|^2+2{\Re}\langle f,Q_r g\rangle\geq \|f\|^2+\|g\|^2-2\|f\|\cdot \|Q_r g\|\geq \|f\|^2+\|g\|^2-2\|f\|\cdot \|g\|\geq 0
\]
for any $f,g$.  For general $m$ a Fourier-based proof will be given in Section \ref{sec:Fourier}.  For (\ref{eq:defcov}),  positive definiteness follows from Theorem \ref{theo:mainres} but a direct proof is unclear.
\end{remark}

\begin{remark}[Restriction to the cylinder]  In Theorem \ref{theo:mainres},  dynamical fluctuations at the edge of the spectrum are characterized  by a Gaussian field with covariance $\langle f,P_{|t-s|} g \rangle_{H^{1/2}(\partial\mathbb{D})}$.  Remarkably,  this exact limiting field on the cylinder $\partial\mathbb{D}\times \mathbb{R}$ also describes fluctuations of the unitary Brownian motion \cite{Spo1986,Spo1998,BouFal2022}: at equilibrium when $N\to\infty$,
$$
{\rm Cov}({\rm Tr}f(U_N(s)),{\rm Tr}g(U_N(t))\to \langle f,P_{|t-s|} g \rangle_{H^{1/2}(\partial\mathbb{D})}.
$$
This dynamically extends the  equality between macroscopic fluctuations of eigenvalues for unitary matrices (the right-hand side of (\ref{eqn:1d}) is also $\langle f,g\rangle_{{\rm H}^{1/2}(\partial\mathbb{D})}$) and for  the edge of the Ginibre ensemble  (when the supports of $f,g$ overlap the boundary,  
the additional term 
$\frac{1}{2}\langle f,g \rangle_{H^{1/2}(\partial\mathbb{D})}$ needs to be added to the right-hand side of  (\ref{eqn:stat2d}), see \cite{RidVir2007}).
\end{remark}

\begin{remark}[Uniformity in the parameters]
We stated Theorem~\ref{theo:mainresmeso} only in the bulk regime $|v|<1$, however a similar statement holds when $|v|=1$ if $a\le \epsilon$, for some very small $N$--independent $\epsilon>0$. 
Additionally, for the sake of clarity we stated Theorems~\ref{theo:mainres}--\ref{theo:mainresmeso} only for two $N$--independent times $s,t$, but the proof also gives an analogous result for multiple,  $N$-dependent, times.
\end{remark}

\begin{remark}[Coupling at equilibrium]
By inspecting the proof of the main result,  in the special equilibrium case,  it is possible to upgrade the weak convergence in the following sense.
There exists a coupling between the fields $(L_N)_{f,t}$ and $(L)_{f,t}$,  with $L_N(f,t)$ and $L(f,t)$ being both measurable function of $(B_s)_{-\infty< s\leq t}$, 
such that $(L)_{f,t}$ is a Gaussian field with the covariance (\ref{eqn:cov3}) and $L_N(f,t)-L(f,t)$ converging to 0 in probability as $N\to\infty$, for any fixed $f,t$.
\end{remark}

\subsection{Eigenvector overlaps.}\label{eqn:overlaps}\ 
The limiting field from Theorem \ref{theo:mainresmeso}  is non--Markovian,  as stated below,  a fact due to the influence of the eigenbasis on the dynamics (\ref{eq:OU}),  as
the martingales $M_i$ have joint bracket
$
\rd\langle M_i,M_j\rangle_t=\mathscr{O}_{ij}(t)\rd t$,  where we remind that  $\mathscr{O}_{ij}= (R_j^* R_i)(L_j^* L_i),
$
with the left and right eigenvectors $\bL,\bRR$ normalized with
$
L_i^tR_j=\delta_{ij}.
$

\begin{proposition}\label{prop:Markov} The random field $(L(f,s))$ from Theorem  \ref{theo:mainresmeso} is not Markovian with respect to the time variable.  More precisely,  
denoting $\Sigma_{s-}=\sigma(\{L(g,u),u\leq s,g\in\mathscr{C}^\infty_c\})$ and $\Sigma_{s}=\sigma(\{L(g,s),g\in\mathscr{C}^\infty_c\})$,  for any $s<t$ there exists
$f\in \mathscr{C}^\infty_c$ such that
$$
\mathbb{P}\left(\E\left[L(f,t)\mid \Sigma_{s-}\right]=\E\left[L(f,t)\mid \Sigma_{s}\right] \right)<1.
$$
\end{proposition}

This non-Markovianity means that the limiting field keeps memory of some statistics on eigenvectors. This is for example manifested through the following consequence of Theorem  \ref{theo:mainresmeso}, on space-time correlations between overlaps,  at equilibrium:  
While eigenvalues of Ginibre matrices show exponential decay of correlations in space,  the overlaps exhibit polynomial (quartic) decay,  even in space-time.  

\begin{corollary}\label{cor:overlaps} Consider the dynamics (\ref{eq:OU}) at equilibrium, $|v|<1$ and some $a\in(0,1/2)$.
Then for  $|t-s|^{1/2},|\sigma_i(s)-v|$ and $|\sigma_j(t)-v|$ of order $N^{-a}$,  we have (denoting $c_v=1-|v|^2$)
\begin{equation}\label{eqn:conclusionCorr1}
{\rm Cov}\Big(\mathscr{O}_{ii}(s),\mathscr{O}_{jj}(t)\Big)\sim\frac{c_v^2}{(c_v|t-s|+|\sigma_i(s)-\sigma_j(t)|^2)^2}
\end{equation}
in the following distributional sense:
For any $\tilde s_1<\tilde s_2<\tilde t_1<\tilde t_2$ fixed,  $s_i=N^{-2a}\tilde s_i$,   $t_i=N^{-2a}\tilde t_i$  and $f_{v,a},g_{v,a}$ as in (\ref{eq:resfunc}),  for some fixed $\e>0$ we have
\begin{align}\begin{split}
&\phantom{{}={}}\int_{[s_1,s_2]\times[t_1,t_2]}\E\Big[\sum_{i,j} f_{v,a}(\sigma_i(s))g_{v,a}(\sigma_j(t))\Big(\frac{\mathscr{O}_{ii}(s)}{N}-c_v\Big)\Big(\frac{\mathscr{O}_{jj}(t)}{N}-c_v\Big)\Big]\rd s\rd t\\
&=
\int_{[s_1,s_2]\times[t_1,t_2]}
\int_{\mathbb{C}^2}  f_{v,a}(z)g_{v,a}(w) \frac{ c_v^2}{(c_v|t-s|+|z-w|^2)^2}\frac{\rd z\rd w}{\pi^2}
\rd s\rd t+\OO(N^{-4a-\e}).\label{eqn:DynOve}
\end{split}\end{align}
\end{corollary}

For the relevant application of the above result,   $f,g$ are nonnegative with disjoint supports, so that the  integral on the right-hand side is $\gtrsim N^{-4a}$: The above corollary identifies the  asymptotics of correlations in distributional sense on any mesoscopic scale.

\begin{remark}[Coherence with static correlations]
For the Ginibre ensemble,  pointwise correlations of eigenvector overlaps are known: From  \cite[Equation (1.12)]{BouDub2020}, when $z_1,z_2$ are in the bulk of the spectrum at distance of order $N^{-a}$ ($a\in(0,1/2)$), we have
$$
\E\left(\mathscr{O}_{11}\mathscr{O}_{22}\mid \sigma_1=z_1,\sigma_2=z_2\right)= N^2c_{z_1}c_{z_2}\left(1+\frac{1}{N^2|z_1-z_2|^4}\right)\left(1+\OO(N^{-2a+\e})\right),
$$
and by mimicking the proof of \cite[Equation (1.9)]{BouDub2020} we also have
$
\E\left(\mathscr{O}_{11}\mid \sigma_1=z_1,\sigma_2=z_2\right)= Nc_{z_1}\left(1+\OO(N^{-1+\e})\right)
$.
Both equations together give
$$
\E\left((\mathscr{O}_{11}-Nc_{z_1})(\mathscr{O}_{22}-Nc_{z_1})\mid \sigma_1=z_1,\sigma_2=z_2\right)= \frac{c_{z_1}c_{z_2}}{|z_1-z_2|^4}+\OO(N^{2-2a}),
$$
which agrees with (\ref{eqn:conclusionCorr1}) when $s=t$ (at least for $a>1/3$, which implies $|z_1-z_2|^{-4}\gg N^{2-2a}$).
\end{remark}

\begin{remark}[Microscopic separation of eigenvalues]
In the Ginibre case,  \cite{BouDub2020} also provides correlation asymptotics when $\sigma_1-\sigma_2\asymp N^{-1/2}$.  It remains an interesting problem to understand the joint distribution of overlaps
in the more general dynamical situation, when $d((\sigma_1,s),(\sigma_2,t))\asymp N^{-1/2}$.
\end{remark}

\subsection{Proof ideas.}\  
The starting point of the analysis of our non-Hermitian ensembles is Girko's Hermitization method,  in a time-dependent setting,  which allows to decompose the linear statistics $L_N(f,t)$ as a sum
from submicroscopic, microscopic, and mesoscopic scales (see Equation (\ref{eq:girkosplit})) for a family of Hermitian spectra.  In the static case,  the submicroscopic contributions have been known to be negligible since seminal lower bounds on smallest singular values (see e.g. \cite{TaoVu2010}),   and these estimates apply equally to our dynamical setting.
In the static case,  only recently the fluctuations from the  microscopic and mesoscopic scales have been evaluated in \cite{CipErdSch2023}, which is an important inspiration for our work.

Compared to \cite{CipErdSch2023},  to treat dynamics the novelties of our proofs are first technical on the microscopic scale,  with a very general statement on independence of small singular values (Theorem \ref{theo:mainthmdbm}), and then conceptual on the mesoscopic scale,  with direct  emergence of the limiting Gaussian field $L(f,t)$ from the noise in the Dyson Brownian motion (Section \ref{sec:resolvents}).\\

\noindent To handle microscopic scales, \cite{CipErdSch2023} proved that the smallest  singular values of Hermitized matrices corresponding to distant spectral points are independent,  by technically difficult variants of the dynamical method \cite{LanSosYau2019}.  
This independence result from \cite{CipErdSch2023} would be sufficient for that part of the proof of our main theorems,  but we present a self-contained general Theorem~\ref{theo:mainthmdbm}, which  gives this independence  in much greater generality,  under the assumptions of some weak local law
instead of rigidity of the particles,  and natural bounds on eigenvector overlaps, which are model-dependent and follow in our case from important estimates in \cite{CipErdSch2023}.   This theorem could be applied to a wide range of non-Hermitian models with non-trivial mean and a variance profile,  or directed graphs.   Furthermore, this method has potential to apply to many problems involving eigenvalues statistics close to a hard edge.

Like \cite{CipErdSch2023}, our proof of this independence of singular values proceeds by dynamics, but it is considerably
shorter and more direct. 
In particular,   \cite{CipErdSch2023} built on relaxation of the Dyson Brownian motion at the hard edge proved in \cite{chelopatto2019}. 
With a different method,  our main relaxation result (Proposition \ref{p:universality})
improves on this local ergodicity from \cite{chelopatto2019} in two directions: 
\begin{enumerate}[(1)]
\item It initiates the study of a natural extension of the Dyson Brownian motion, as it covers 
relaxation  for arbitrary,  correlated martingales driving individual particles,  see (\ref{e:sprocess}) and the minimal assumption (\ref{eqn:bracketBound}). Although,  currently,  this new correlated setting 
does not  clearly connect to random matrices,
dynamics with general colored noise is of general interest in stochastic (partial) differential equations, see for example Section 7.2.4 in the classical book \cite{DaPZab} for general theory, and 
\cite{DaPGat} for the stochastic Burgers equation.
Proposition \ref{p:universality}  provides an example of universal local convergence (i.e.  independent on the initial condition),  for singular SDE coefficients and essentially arbitrary noise structure in space (but white in time).   

\item We obtain the following submicroscopic error estimates (up to subpolynomial terms in $N$) for the relaxation under the sole assumptions of an initial local law  and arbitrary noise. In particular
($s_1$ and $s_1'$ are the particles closest to the hard edge)
\[
|s_1(t)-s'_1(t)|\leq \frac{1}{N\cdot\sqrt{Nt}},
\]
for any $\bms(0), \bms'(0)$  satisfying a weak local law with respect to the same deterministic density, and $\bms(t),\bms'(t)$ following the extended (correlated) Dyson Brownian Motion dynamics.
We refer to (\ref{eqn:Opt}) for an explanation of this error term in this very general setting.
When the underlying Brownian motions are independent and the initial conditions satisfy a stronger initial rigidity assumption,   relaxation holds with a smaller error term as explained in Remark \ref{rem:stronger}.
\end{enumerate} 

A key tool for the above relaxation estimate is an observable introduced in \cite{Bou2022}, defined in (\ref{eqn:ft}).  
However the proof of (2) is considerably simpler than the optimal relaxation of the eigenvalue gaps in the bulk of the spectrum  \cite{Bou2022},  as it bypasses a decomposition into short and long range interactions that has been customary in the dynamics approach \cite{ErdYau2015,BouYau2017,LanSosYau2019}, and does not require any maximum principle. Our proof of the hard edge relaxation is therefore closer in spirit to the proof of soft edge relaxation given in \cite{Bou2022}: It demonstrates that the sole analysis of a key averaged observable gives microscopic statistics at the hard edge,  by fully exploiting the symmetry of the particles.  After \cite{Bou2022} the possibility of such a proof at the hard edge was unclear,  because the interparticle distance $N^{-1}$ at the hard edge and in the bulk coincide. 

Moreover, in the considered setting of arbitrary correlations for the noise, the observable from \cite{Bou2022} is still the key to quantitative relaxation, although meaningful changes are needed compared to the case of independent noises: in \cite{Bou2022}, rigidity (a local law with error $(N\eta)^{-1}$) is preserved through the dynamics, which eventually yields the small error term $\frac{1}{N^2 t}$ for the relaxation of gaps in the bulk of the spectrum, while in our setting only a weaker local law (with error $(N\eta)^{-1/2}$, see Proposition~\ref{prop:LocalLaw}) is preserved, which is responsible for the weaker bound $\frac{1}{N\cdot\sqrt{N t}}$ at the hard edge. This deteriorated bound would also be present if arbitrary martingales were considered in \cite{Bou2022}, due, for example, to a larger bracket term in the proof of \cite[Lemma 3.10]{Bou2022}. \\

\noindent Regarding mesoscopic scales,  our method is also fully dynamical by generating space-time correlations of linear statistics from a family of coupled dynamics on resolvents, while \cite{CipErdSch2023} proceeded by cumulant expansions.  In our proof, Gaussianity of the limiting fluctuations easily follows from an explicit writing of the fluctuations as stochastic integrals,  with highly concentrated integrands, see Proposition~\ref{pro:decomp}.  In particular,  we unveil the (space--time) logarithmic correlation of the linear statistics as the cumulative effect of the Brownian dynamics \eqref{eq:OU}. For any $s<t$, this is achieved by splitting the randomness of $X_t$ into two parts: i) the randomness of $X_s$, and ii) the fresh randomness introduced by $\dif B$ to go from $X_s$ to $X_t$. More precisely,  for any $\eta_r>0$ such that $\eta_t=N^{-1}$ we decompose
\begin{equation*}
\begin{split}
\log\big|\mathrm{det}(X_t-w)\big|&\approx \frac{1}{2}  \log \det \left[ |X_t -w|^2 + N^{-1}\right] \\
&=\frac{1}{2}  \log \det \left[ |X_s -w|^2 + \eta_s^2\right] + \frac{1}{2} \int_s^t \dif \left(  \log \det \left[ |X_r -w|^2 + \eta_r^2\right] \right).
\end{split}
\end{equation*}
For a specific choice  $\eta_s\approx N^{-1}+(t-s)$, the above integral becomes purely stochastic (see  Proposition~\ref{pro:decomp} for more details), so that we reduced the correlations at different times to correlations at the same time but at different spectral parameters:
\begin{align*}
\mathrm{Cov}\big[\log\big|\mathrm{det}(X_t-w)\big|, \, \log\big|\mathrm{det}(X_s-z)|\big]
&\approx\mathrm{Cov}\left[\frac{1}{2}  \log \det \left[ |X_s -w|^2 + \eta_s^2\right], \frac{1}{2}  \log \det \left[ |X_s -z|^2 + N^{-1}\right]  \right] \\
&\approx - \log d\big((z,s), (w,t)\big),
\end{align*}
with $d\big((z,s), (w,t)\big)$ being defined in \eqref{eq:defpardist}; see \eqref{eq:usefform}--\eqref{eq:finanswer} for detailed calculations giving rise to this parabolic distance.

The method of characteristics is a powerful technique to characterize fluctuations of the spectrum of random matrices.    
We refer for example to \cite[Section 4]{HuaLan2019} for its application to $\beta$-ensembles and to \cite[Section 7]{LanSos2022} for  Wigner matrices.
In this work,  we extend its scope to fluctuations in the context of  time-dependent and non-Hermitian ensembles,  identifying correlation at different times as correlation for the same matrix but at different spectral parameters.\\

\noindent A surprising aspect of Theorems~\ref{theo:mainres}--\ref{theo:mainresmeso} is the remarkably simple limiting space-time logarithmic correlation,  despite its emergence from the superposition of the
complicated covariances from a family of resolvents (from the Hermitization).  
This simplicity can be interpreted, in Corollary
\ref{cor:overlaps},  as a statement on space-time correlations between eigenvector overlaps.  Indeed, Equation (\ref{eqn:DynOve}) follows from the combination of Theorem \ref{theo:mainresmeso} and the non-Hermitian dynamics (\ref{eq:OU}), at equilibrium.  Out of equilibrium,  such a quartic decay of correlations seems out of reach: To exhibit  overlaps at two distinct times, we need to consider not only forward but also backward dynamics, which are tractable only at equilibrium (see Subsection \ref{subsec:revers}). \\

\noindent Finally,  we note that in a recent breakthrough \cite{MalOsm2023},  universality for non-Hermitian matrices on the finer local scale was obtained through partial Schur transform and supersymmetry. It is unclear whether these methods may apply to multi-time fluctuations as treated in our work.

\subsection{Outline.}

\, The rest of the paper is organized as follows: In Section~\ref{sec:mainres} we prove our main results Theorems~\ref{theo:mainres}--\ref{theo:mainresmeso}. In Section~\ref{sec:Indep} we present a simple, self--contained, decorrelation argument for the small singular values of $X-z_1$, $X-z_2$ when $z_1,z_2$ are sufficiently far away from each other (i.e. $|z_1-z_2|\gg N^{-1/2}$). Finally, in Section~\ref{sec:resolvents} we analyze the resolvent in Girko's formula  along the stochastic advection equation, which enables us to identify the Gaussian randomness of the linear statistics as the result of time increments.\\

\subsection{Notations.}

\, For two quantities $X$ and $Y$ depending on $N$, 
we write that $X = \OO(Y )$ or $X\lesssim Y$ if there exists some universal constant $C>0$ such
that $|X| \leq C Y$ . We write $X = \oo(Y )$, or $X \ll Y$ if the ratio $|X|/Y\rightarrow \infty$ as $N$ goes to infinity. We write
$X\asymp Y$ if there exists a universal constant $C>0$ such that $ Y/C \leq |X| \leq  C Y$. We denote $\llbracket a,b\rrbracket = [a,b]\cap\bZ$ and $\llbracket n\rrbracket = \llbracket 1, n\rrbracket$. Furthermore, for a matrix $\C^{d\times d}$ we denote its normalized trace by $\langle A\rangle:= d^{-1}\mathrm{Tr}[A]$.  
Because of the Hermitization,  in this paper $d=2N$.

Finally,  we will write that a sequence of events $(A_N)_{N\geq 1}$ holds with very high probability if for any fixed $D>0$ there is a $C>0$ such that $\mathbb{P}(A_N)\geq 1-CN^{-D}$ for all $N\geq 1$.

\section{Time correlations: Proof of Theorems~\ref{theo:mainres}--\ref{theo:mainresmeso}}
\label{sec:mainres}

To make our presentation simpler, from now on we only consider random matrices $X$ whose entries $\chi$ have a probability density $g$ satisfying
\begin{equation}
\label{eq:entryass}
g\in L^{1+\alpha}(\C), \qquad\quad \lVert g\rVert_{L^{1+\alpha}(\C)}\le N^\beta,
\end{equation}
for some $\alpha,\beta>0$. If $X$ does not satisfy this assumption, then, relying on \cite[Theorem 23]{TaoVu2015}, we show that the distribution of the linear statistics of $X$ is close to the one of $X+N^{-\gamma}X_{\mathrm{Gin}}$, for any large $\gamma>0$ and $X_{\mathrm{Gin}}$ being a complex Ginibre matrix independent of $X$. Then it is easy to see that the entries of this new matrices satisfy \eqref{eq:entryass}.

To analyze the linear statistics \eqref{eq:linstattime} we rely on Girko's formula (cf. \cite{Gir1985, TaoVu2015}):
\begin{equation}
\label{eq:girko1}
\sum_{\sigma\in\mathrm{Spec}(X)}f(\sigma)=\frac{\ii}{4\pi}\int_\C \Delta f(z)\int_0^\infty  \mathrm{Tr} \big[G^z(\ii\eta)\big]\, \dif\eta\dif z.
\end{equation}
Here $G^z$ denotes the resolvent $G^z(\ii\eta):=(W-Z-\ii\eta)^{-1}$, with $W$ being the \emph{Hermitization} of $X$:
\begin{equation}
\label{eq:herm}
W:=\left(\begin{matrix}
0 & X \\
X^* & 0
\end{matrix}\right), \qquad\quad Z:=\left(\begin{matrix}
0 & z \\
\overline{z} & 0
\end{matrix}\right).
\end{equation}
We will use Girko's formula for the eigenvalues of $X_t$, with $X_t$ being the solution of the flow \eqref{eq:OU} with initial condition $X$. By $W_t$ we denote the Hermitization of $X_t$ defined as in \eqref{eq:herm} with $X$ replaced by $X_t$, and by $G_t^z(\ii\eta)$ we denote its resolvent. Note that the spectrum of $W_t-Z$ is symmetric with respect to zero as a consequence of its $2\times 2$ block structure (\emph{chiral symmetry}). We denote the eigenvalues of $W_t-Z$ by $\{\lambda_{\pm i}^z(t)\}_{i\in [n]}$, with $\lambda_{-i}^z(t)=-\lambda_i^z(t)$, and denote the corresponding eigenvectors by $\bm{w}_{\pm i}^z(t)$. As a consequence of the chiral symmetry, the eigenvectors of $W_t-Z$ are of the form
\begin{equation}
\label{eq:defevectors}
\bm{w}_{\pm i}^z(t)=\left(\begin{matrix}
\bm{u}_i^z(t) \\
\pm\bm{v}_i^z(t)
\end{matrix}\right),
\end{equation}
with $\bm{u}_i^z(t)$, $\bm{v}_i^z(t)$ being the left and right singular vectors of $X_t-z$, respectively.

The remainder of this section is divided into several subsections: in Section~\ref{sec:girkosplit} we divide the $\eta$--integral in \eqref{eq:girko1} into three main regimes, since each one of these regimes will be analyzed using very different techniques, then in Sections~\ref{sec:submic}--\ref{sec:meso} we deal with these regimes one by one. Finally, in Section~\ref{sec:end} we put all these together and conclude the proof of Theorems~\ref{theo:mainres}--\ref{theo:mainresmeso}.

\subsection{Decomposition in three terms.}
\label{sec:girkosplit}

\, We split the analysis of Girko's formula into several regimes (cf. \cite[Equation (3.10)]{CipErdSch2023})
\begin{align}
\sum_{i=1}^N f_{v,a}(\sigma_i(t))-\E \sum_{i=1}^N f_{v,a}(\sigma_i(t))&=\frac{1}{4\pi}\int_\C \Delta f(z)\big[|\log\mathrm{det}(H_t^z-\ii T)|-\E|\log\mathrm{det}(H_t^z-\ii T)|\big]\,\dif z \notag\\
&\quad -\frac{1}{4\pi \ii}\int_\C \Delta f_{v,a}(z)\left(\int_0^{\eta_0}+\int_{\eta_0}^{\eta_c}+\int_{\eta_c}^T \right) \mathrm{Tr} \big[G_t^z(\ii\eta)-\E G_t^z(\ii\eta)\big]\, \dif\eta\dif z\notag\\
&=: J_T(f_{v,a},t)+I_0^{\eta_0}(f_{v,a},t)+I_{\eta_0}^{\eta_c}(f_{v,a},t)+I_{\eta_c}^T(f_{v,a},t),\label{eq:girkosplit}
\end{align}
with
$
\eta_0=N^{-1-\delta_0}$, $\eta_c=N^{-1+\delta_1},
$
for some $\delta_0, \delta_1>0$, and $T=N^{100}$. In order to keep the notation for the proof of the macroscopic and mesoscopic CLT unified, with a slight abuse of notation, here we use the convention that $f_{v,a}=f$ when $a=0$.

The integral $J_T$ can be easily seen to be negligible (see e.g. the proof of \cite[Lemma 3]{CipErdSch2021bis} and \cite[Theorem 2.3]{AEK21}). Next, we show that the very small $\eta$--regime $I_0^{\eta_0}$ is negligible using smoothing inequalities for the smallest singular value of $X-z$ (see Section~\ref{sec:submic}). The regime $I_{\eta_0}^{\eta_c}$ will also be negligible (see Section~\ref{sec:micro}), for this term we will rely on the asymptotical independence of the small singular values of $X_t-z_1$, $X_t-z_2$ for $z_1,z_2$ sufficiently away from each other.  Finally, we will compute high moments of $I_{\eta_c}^T$ showing that this is the regime where the order one contribution to the lhs. of \eqref{eq:girkosplit} comes from (see Section~\ref{sec:meso}).

\subsection{Submicroscopic scale.}
\label{sec:submic}

\, By the lower tail estimate for the smallest singular value of $X-z$ from \cite[Theorem 3.2]{TaoVu2010} (see also \cite[Equation (4a)]{CipErdSch2020}), we readily conclude that the contribution of this regime is negligible. In particular, we have (cf. \cite[Lemma 4.4]{CipErdSch2023})
\begin{equation}
\label{eq:smallbneed}
\E |I_0^{\eta_0}(f_{v,a},t)|\le N^{-c},
\end{equation}
uniformly in $t\lesssim 1$, for some small fixed constant $c>0$.

\subsection{Microscopic scale.}
\label{sec:micro}

\, In this section we show that the contribution of the regime $I_{\eta_0}^{\eta_c}$ to the lhs. of \eqref{eq:girkosplit} is also negligible in a second moment sense:

\begin{proposition}
\label{prop:indep}
There exists a small constant $\e>0$ such that for $\delta_0$ and $\delta_1$ small enough,  and  any $t\lesssim 1$,
we have
\begin{equation}
\label{eq:indepresres}
\E |I_{\eta_0}^{\eta_c}(f_{v,a},t)|^2\le N^{-\e}.
\end{equation}
\end{proposition}
Note that \eqref{eq:indepresres} also implies $\E I_{\eta_0}^{\eta_c}(f_{v,a},t) I_{\eta_0}^{\eta_c}(f_{v,a},s)\le N^{-\e}$, for any $s\le t$, by a simple Schwarz inequality. The main input to prove Proposition~\ref{prop:indep} is the following independence of resolvents.
A similar statement was obtained in \cite[Proposition 3.5]{CipErdSch2023}.  In  Section~\ref{sec:Indep} we will give an independent proof,
which covers broad classes of initial conditions,  and proceeds through a new, simple and quantitative proof of hard edge universality in random matrix theory.

\begin{proposition}
\label{prop:resindep}
Fix any small $c,\omega_p>0$,
let $\delta_0,\delta_1$ be small enough constants.
Then uniformly in 
 $\abs{z_l}\le 1-N^{-c}$, $\abs{z_1-z_2}\ge N^{-1/2+\omega_p}$, and  $\eta_1,\eta_2\in [N^{-1-\delta_0}, N^{-1+\delta_1}]$, we have\footnote{We state this result for only two different $z$'s for simplicity, but a similar result holds for any finite product of resolvents such that $\min_{i\ne j}|z_i-z_j|\gg N^{-1/2}$.}
\begin{equation}
\label{eq:indepres}
        \E \langle G_t^{z_1}(\ii\eta_1)\rangle \langle G_t^{z_2}(\ii\eta_2)\rangle=   \E \langle G_t^{z_1}(\ii\eta_1)\rangle \E \langle G_t^{z_2}(\ii\eta_2)\rangle+\OO\left(N^{-\e}\right),
\end{equation}
for some $\e>0$.
\end{proposition}
 
 \begin{proof}[Proof of Proposition~\ref{prop:indep}]
Using that the regime $|z_1-z_2|< N^{-1/2+\omega_p}$ is negligible (see e.g. \cite[Eq. (4.12)]{CipErdSch2023} and \cite[Eq. (4.8)]{CipErdSch2022bis}), this is straightforward by expanding the square in $\E |I_{\eta_0}^{\eta_c}(f_{v,a},t)|^2$, using  \eqref{eq:indepres} and $\int |\Delta f|\lesssim 1$ (see the proof of \cite[Lemma 4.4]{CipErdSch2023}). We point out that in the proof of \cite[Lemma 4.4]{CipErdSch2023} it is shown that the regime $|z_l|\ge 1-N^{-c}$ is negligible (see the paragraph below \cite[Eq. (4.11)]{CipErdSch2023}) and so that to conclude \eqref{eq:indepresres} it is indeed enough to use \eqref{eq:indepres} for $|z_l|\le 1-N^{-c}$.
 \end{proof}

\subsection{Mesoscopic scales.}
\label{sec:meso}

\, In this section we show that the regime $I_{\eta_c}^T$ is the one that gives the order one leading contribution to \eqref{eq:girkosplit}. 

\begin{proposition}
\label{prop:multitCLTres}
Fix $T\ge 0$, and denote by $\Pi_p$ the set of pairings\footnote{Note that $\Pi_p=\emptyset$ if $p$ is odd.}  on \([p]\). Then, there exists $c=c(p)>0$ such that, for $t_{i,a}:=T+N^{-2a}t_i$, we have
\begin{equation}
\label{eq:CLTtimecor}
\E \prod_{i\in [p]} I_{\eta_c}^T(f_{v,a}^{(i)},t_{i,a})=  \sum_{P\in \Pi_p}\prod_{\{i,j\}\in P} \E I_{\eta_c}^T(f_{v,a}^{(i)},t_{i,a})I_{\eta_c}^T(f_{v,a}^{(j)},t_{j,a})+\OO(N^{-c}).
\end{equation}
Furthermore, we have
\begin{equation}
\label{eq:correxplis}
\E I_{\eta_c}^T(f_{v,a}^{(i)},t_{i,a})I_{\eta_c}^T(f_{v,a}^{(j)},t_{j,a})= \Gamma(f^{(i)},f^{(j)}, |t_i-t_j|,\kappa_{4,t_i\wedge t_j})+\OO(N^{-c}),
\end{equation}
with $\Gamma(f,g,\tau,\kappa_{4,t})$ from \eqref{eq:defcov} for $a=0$; if $a\in (0,1/2)$ we have the same result with $\Gamma(f,g,\tau,\kappa_{4,t})$ replaced by $\Gamma_v(f,g,\tau)$ from \eqref{eq:defcovmeso}. Additionally, for any $t_a>0$ and any $0\le s_a<t_a$, we have the decomposition
\begin{equation}
\label{eq:defint12}
I_{\eta_c}^T(f_{v,a},t_a)=:I_1(f_{v,a},s_a,t_a)+I_2(f_{v,a},s_a,t_a),
\end{equation}
with $I_1$ depending only on $X_s$, $ \E I_1(f,s,t)I_2(f,s,t)=\OO(N^{-c})$,  and
\begin{equation}
\label{eq:covsplit}
\E|I_1|^2= \Gamma(f,f,2(t-s),\kappa_{4,t})+\OO(N^{-c}), \qquad\quad \E |I_2|^2=\Gamma(f,f,0,0)-\Gamma(f,f,2(t-s),\kappa_{4,s})+\OO(N^{-c}), 
\end{equation}
for some small fixed $c>0$. If $a\in (0,1/2)$ we have the same result with $\Gamma(f,g,\tau,\kappa_{4,t})$ replaced by $\Gamma(f,g,\tau)$ from \eqref{eq:defcovmeso}.
\end{proposition}

Finally, in the next section we combine \eqref{eq:smallbneed} with Propositions~\ref{prop:indep} and \ref{prop:multitCLTres}, to conclude Theorems~\ref{theo:mainres}--\ref{theo:mainresmeso}.

\subsection{Proof of Theorems~\ref{theo:mainres}--\ref{theo:mainresmeso}.}
\label{sec:end}

\, Using that $J_T$ from \eqref{eq:girkosplit} is negligible (see e.g. the proof of \cite[Lemma 3]{CipErdSch2021bis} and \cite[Theorem 2.3]{AEK21}) and the a priori bounds
\[
\big|I_0^{\eta_0}(f_{v,a}^{(i)},t_{i,a})\big|+\big|I_{\eta_0}^{\eta_c}(f_{v,a}^{(i)},t_{i,a})\big|+\big|I_{\eta_c}^T(f_{v,a}^{(i)},t_{i,a})\big|\le N^\xi
\]
with very high probability for any small $\xi>0$ (see \cite[Equation (4.4)]{CipErdSch2023} and \cite[Equation (4.5)]{CipErdSch2022bis}, for the macroscopic and mesoscopic case, respectively), we find that
\begin{equation}
\label{eq:step11}
\E \prod_{i\in [p]} L_N(f_{v,a}^{(i)},t_{i,a})=\E \prod_{i\in [p]} \bigg[I_0^{\eta_0}(f_{v,a}^{(i)},t_{i,a})+I_{\eta_0}^{\eta_c}(f_{v,a}^{(i)},t_{i,a})+I_{\eta_c}^T(f_{v,a}^{(i)},t_{i,a})\bigg]+\OO(N^{-c}).
\end{equation}
Then by \eqref{eq:smallbneed} and \eqref{eq:indepresres} we find out that the only order one contribution to the lhs. of \eqref{eq:step11} comes from the large $\eta$ regime, i.e. we have
\begin{equation}
\label{eq:step22}
\E \prod_{i\in [p]}  L_N(f_{v,a}^{(i)},t_{i,a})= \E \prod_{i\in [p]}I_{\eta_c}^T(f_{v,a}^{(i)},t_{i,a})+\OO(N^{-c}).
\end{equation}
Finally, Theorem~\ref{theo:mainres} readily follows by Proposition~\ref{prop:multitCLTres} together with \eqref{eq:step22}. 
\qed

\section{Independence of singular values}\label{sec:Indep}

In this section we will make use of the notation
$\omega_*\ll \omega^*$ to denote that $\omega_*\le \omega^*/10$, for any 
two constants $\omega_*,\omega^*$.
Our main error parameter will be polynomial in $N$ and denoted by
\begin{equation}\label{eqn:phi}
\varphi=N^\nu,
\end{equation}
for some fixed (small) constant $\nu>0$.  All other small parameters appearing in this section will satisfy
\[
\nu, \omega\ll \omega_t\ll \omega_K \ll\wt\omega.
\]
Below is their meaning and where they are introduced:
\begin{itemize}

\vspace{-0.15cm}
\item[--]$\nu$. For $\varphi$ from \eqref{eqn:phi}, the local law holds on scale $\varphi/N$ (see Equation (\ref{eq:weakllaw})).

\vspace{-0.2cm}
\item[--]$\omega$.  The submicroscopic error in independence of singular values  is $N^{-1-\omega}$ (Equation (\ref{eq:evaluesclosind})).

\vspace{-0.2cm}
\item[--]$\omega_t$. This approximate independence occurs after time at least $N^{-1+\omega_t}$  (Equation (\ref{eq:evaluesclosind})).

\vspace{-0.2cm}
\item[--]$\omega_K$. The eigenvector products are small for spectral indices $k\leq N^{\omega_K}$ (Equation (\ref{eq:evass})).

\vspace{-0.2cm}
\item[--]$\wt\omega$. These eigenvector products are of size at most $N^{-\wt\omega}$ (Equation (\ref{eq:evass})).
\end{itemize}

\subsection{Main statement.} \, The main result of this section is the following general statement about asymptotic independence of small singular values.
It will imply Proposition \ref{prop:resindep}.

\begin{theorem}
\label{theo:mainthmdbm}
Let $X$ be a $N\times N$ matrix with complex entries, and consider the singular values $\lambda_i^z$ of $X-z$, for $z\in I:=\{z_1,\dots,z_q\}$,  $q$ fixed,   $z_1,\dots,z_q\in\mathbb{C}$ possibly $N$-dependent.

Assume the following local law on scale $\varphi/N$ holds.  There is a $c>0$ and $N_0$ such that for any  $N\geq N_0$ and
 $z\in I$ there exists $m=m^{z}$, the Stieltjes transform of a deterministic probability measure $\zeta=\zeta^z=\rho^z_0(x)\rd x$, such that $c\le |\Im\, m(w)|\le c^{-1}$, $|m(w)|+|\del_w m(w)|\leq c^{-1}$ for all $|w|\le c$,
and 
such that for any $\Im w\geq \varphi/N$,  $|w|\leq c$ we have
\begin{equation}
\label{eq:weakllaw}
\left|\frac{1}{N}\sum_{i=1}^N\frac{1}{\lambda_i^z-w}-m^{z}(w)\right|\le \frac{1}{\sqrt{N\Im w}} .
\end{equation}

Let $X_t:=X+N^{-1/2} B_t$, with $(B_t)_{ij}$ being i.i.d. standard complex Brownian motions,
and for $z\in I$ denote the singular values of $X_t-z$ by $\lambda_i^{z}(t)$. Let ${\bm u}_i^{z}(t), {\bm v}_i^{z}(t)$ be the left and right singular vectors of $X_t-z$, respectively, and assume that there exists $\omega_K\ll\wt\omega$ such that for any $z,w$ distinct in $I$ and $D>0$ we have 
\begin{equation}
\label{eq:evass}
\P\Big(\big|\langle {\bm u}_i^{z}(t),  {\bm u}_j^{w}(t)\rangle\big|+\big|\langle {\bm v}_i^{z}(t),  {\bm v}_j^{w}(t)\rangle\big|\le N^{-\wt\omega}, \quad \forall\,1\le i,j\le N^{\omega_K}\,\,\mathrm{and}\,\, t\in [0,\mathcal{T}]\Big)\ge 1-N^{-D},
\end{equation}
for a fixed $\mathcal{T}>0$. Let $(X^{(p)})_{1\leq p\leq q}$ be independent Ginibre matrices, and denote the singular values of $X^{(p)}$ by $\mu_i^{(p)}$. Then for any time $t\in[N^{-1+\omega_t},\mathcal{T}]$ there exists a coupling and constants $\omega\ll \omega_t\ll \omega_K$ so that for any $z_p\in I$ we have 
\begin{equation}
\label{eq:evaluesclosind}
\mathbb{P}\left(\big|\rho_t(0)\lambda_i^{z_p}(t)-\rho_{\mathrm{sc}}(0)\mu_i^{(p)}\big|\le N^{-1-\omega},\ \forall1\le i\le N^\omega\right)\ge 1-N^{-D}.
\end{equation}
Here $\rho_t$ is the density of  $ \zeta^{z_p} \boxplus \mu_{\rm sc}^{(t)}$ (see the discussion in \Cref{s:locallaw} for these notations for measures and their free convolutions), and $\rho_{\rm sc}$ is the density of the semi-circle law.

\end{theorem}

As a consequence of our general Theorem~\ref{theo:mainthmdbm} we readily conclude the proof of Proposition~\ref{prop:resindep}.

\begin{proof}[Proof of Proposition~\ref{prop:resindep}]

By a standard application of a Green's function comparison (GFT) argument, it is enough to prove \eqref{eq:indepres} for matrices with an additional Gaussian component of size $\sqrt{t}$, with $t\le N^{-1/2-\epsilon}$, for a small fixed $\epsilon>0$ (see e.g. \cite[Lemma 7.5]{CipErdSch2023}). We will thus prove \eqref{eq:indepres} for matrices having a Gaussian component of size $\sqrt{\mathcal{T}}$, with $\mathcal{T}=N^{-1+\omega_t}$, for some fixed small $\omega_t>0$. We now show that for matrices with a Gaussian component of size $\sqrt{\mathcal{T}}$ the desired result follows from Theorem~\ref{theo:mainthmdbm}. We point out that Theorem~\ref{theo:mainthmdbm} is stated under the assumption $c\le |\Im\, m(w)|\le c^{-1}$, $|m(w)|+|\del_w m(w)|\leq c^{-1}$, for an $N$-independent $c>0$. When $1-|z|\sim N^{-o(1)}$, this assumption is not satisfied as $|\Im\, m(w)|\sim N^{-o(1)}$ and $|\del_w m(w)|\leq N^{o(1)}$ (see e.g. \cite[Eq. (3.10)]{CipErdXu2025}). However, the conclusion in \eqref{eq:evaluesclosind} holds also under this slightly weaker assumptions as stated in Remark~\ref{rem:closeedge} below.

For $|z_1-z_2|\ge N^{-1/2+\omega_p}$,   by \cite[Theorem 3.1]{CipErdSch2022bis} and \cite[Lemma 7.9]{CipErdSch2023} (see also \cite[Corollary 3.6]{CipErdXu2025}), there exist $\wt\omega,\omega_K$ such that \eqref{eq:evass} is satisfied for a fixed time. To show the bound \eqref{eq:evass} uniformly in time it is enough to use a standard grid argument as a consequence of the fact that the bound on eigenvectors is inherited by a bound on product of resolvents (see e.g. the proof of \cite[Lemma 7.9]{CipErdSch2023}) which is H\"older in time. Additionally, the single resolvent local law \cite[Theorem 3.1]{CipErdSch2021} implies \eqref{eq:weakllaw}. This shows that the assumptions of Theorem~\ref{theo:mainthmdbm} are satisfied, and  \eqref{eq:evaluesclosind} holds. Given \eqref{eq:evaluesclosind} as an input (cf. \cite[Lemma 7.6--7.7]{CipErdSch2023}), using rigidity (cf. \cite[Proposition 7.3]{CipErdSch2023}), the proof of \eqref{eq:indepres} is completely analogous to the proof of \cite[Proposition 7.2]{CipErdSch2023} and so omitted.
\end{proof}

\subsection{Reduction to invariance and relaxation.}\ 
We consider coupled processes $s_i(t), r_i(t)$ defined by
\begin{align}\label{e:sprocess}
\rd s_i(t)&=\frac{1}{\sqrt{2N}}\rd b_i^s (t)+\frac{1}{2N}\sum_{j\neq i}\frac{1}{s_i(t)-s_j(t)}\rd t, \\
\rd r_i(t)&=\frac{1}{\sqrt{2N}}\rd b_i^r (t)+\frac{1}{2N}\sum_{j\neq i}\frac{1}{r_i(t)-r_j(t)}\rd t,\label{e:rprocess}
\end{align}
where $1\leq |i|\leq N$ and $\bms(0)=(s_i(0))_{1\leq |i|\leq N},\br(0)=(r_i(0))_{1\leq |i|\leq N}$  are such that $s_{-i}(0)=-s_i(0)$,  $r_{-i}(0)=r_i(0)$ for $1\leq i\leq N$. 

In this section, the  continuous martingales $b_i^s(t)$ and $b_i^r(t)$
are realized on a common probability space with a common filtration $\mathcal{F}_t$.
They do not need to be Brownian motions. They always satisfy the symmetry
$b^s_{-i}(t)=-b^s_i(t)$,\ $b^r_{-i}(t)=-b^r_i(t)$, and the quadratic variation bound
\begin{equation}\label{eqn:bracketBound}
\rd \langle b_i\rangle_t/\rd t\leq 1\ {\rm a.e.}
\end{equation}
(as a non-decreasing function $\langle b_i\rangle_t$ is differentiable almost everywhere)
where $b_i=b_i^s$ or $b_i^r$.
Under these sole assumptions,  existence and strong uniqueness hold for (\ref{e:sprocess}) and (\ref{e:rprocess}), see Proposition \ref{prop:WellPosed}.

We will consider the following additional hypothesis for these martingales.

\begin{assumption}[Vanishing correlations close to the origin]\label{e:assump}
For fixed parameters $0< \omega_K\ll\wt\omega$,  $K=N^{\omega_K}$ we have
\begin{align*}
\left|\frac{\rd}{\rd t} \langle b_i^s-b_i^r, b_j^s-b_j^r\rangle_t\right|\leq N^{-\wt\omega}\mathds{1}_{1\leq |i|,|j|\leq K}+4(1-\mathds{1}_{1\leq |i|,|j|\leq K}).
\end{align*}
\end{assumption}

We denote the empirical particle density of the initial data $\bms$ as
$
\mu_0=\frac{1}{2N}\sum_{1\leq |i|\leq N}\delta_{s_i(0)},
$
and it Stieltjes transform by $m_0(z)$\footnote{Within this section we use $m(z)$ to denote the Stieljes transforms of a measure $\mu$. Instead, we point out that in Theorem~\ref{theo:mainthmdbm} we used $m^{z}$ to denote the Stieltjes transform of limiting distribution of the singular values of $X-z$.}. We assume that $\mu_0$ is regular in the following sense.

\begin{assumption}[$(\nu,R)$-regularity]\label{d:regular} Fix $\nu\in(0,1],R>0$, we we recall $\varphi=N^\nu$ from \eqref{eqn:phi}. Let $\widetilde \mu$ be a probability measure with Stieltjes transform $\widetilde m$ satisfying,
for some fixed $C>0$,
\begin{align}\label{eqn:tildeC}
|\widetilde m(z)|, |\del_z \widetilde m(z)|\leq C, \quad C^{-1}\leq \Im[\widetilde m(z)]\leq C\ \text{ for all }\ |z|\leq R.
\end{align}
We say that $\mu$ is $(\nu,R)$-regular if,  for such a $\tilde\mu$,  its Stieltjes transform $m$ satisfies 
\begin{align}\label{e:m0tm0diff}
|m(z)-\widetilde m(z)|\leq  \frac{\varphi}{\sqrt{N\Im[z]}}\ \ \mbox{for any $|z|\leq R$ with $\Im[z]\geq \frac{\varphi}{N}$}.
\end{align}
\end{assumption}

\begin{remark}
\label{rem:closeedge}
For our later use,  we may deal with the case that $|\widetilde m_0(z)|, |\del_z\widetilde m_0(z)|$ may grow with $N$ or $\Im[\widetilde m_0(z)]$ may vanish slowly with $N$, namely
\begin{align*}
|\widetilde m_0(z)|, |\del_z \widetilde m_0(z)|\leq N^{\oo(1)}, \Im[\widetilde m_0(z)]\geq N^{-\oo(1)}.
\end{align*}
It is not too hard to see that our argument also extends to this case, as long as we take time $t$ longer enough, to compensate the effect of the error $N^{\oo(1)}$.  We omit the details to keep the presentation simpler.
\end{remark}

 Theorem~\ref{theo:mainthmdbm} will readily follow from the following invariance and relaxation statements.

\begin{proposition}[Invariance]\label{p:independence} 
Let the processes $\bms(t), \bmr(t)$ be solutions of the stochastic differential equations \eqref{e:sprocess} and \eqref{e:rprocess}.  Assume that the driving continuous martingales satisfy Assumption \ref{e:assump}. Additionally assume that the initial data $\bms(0)=\bmr(0)$ satisfies $(\nu, R)$-regularity in the sense of  Assumption \ref{d:regular}. 

Then, for any choice of small constants $\nu\ll \omega_t\ll \omega_K\ll \wt\omega$ there exists $\omega>0$ such that for any $0\le t\le N^{-1+\omega_t}$, with very high probability
\begin{align*}
|s_i(t)-r_i(t)|\leq N^{-1-\omega}, \ \mbox{for all $1\leq |i|\leq N^{ \omega}$}.
\end{align*}
\end{proposition}

\begin{proposition}[Relaxation]\label{p:universality}
Let the processes $\bms(t), \bms'(t)$ both satisfy the stochastic differential equation \eqref{e:sprocess}, with different initial data and the same driving processes $b_i$, which are 
arbitrary continuous martingales (satisfying (\ref{eqn:bracketBound})) and  $\bms(0), \bms'(0)$ are $(\nu, R)$-regular around $\rd \widetilde \mu_0(x)=\wt \rho_0(x)\rd x, \rd \widetilde \mu'_0(x)=\wt \rho_0'(x)\rd x$ in the sense of Assumption \ref{d:regular}, and the two densities are the same at $0$: $\widetilde \rho_0(0)= \widetilde \rho'_0(0)=A$. 
Then for any $\varepsilon>0$ and $t\geq N^{-1+\omega_t}$ with $\omega_t\gg \nu$,  for any $1\leq |i|\leq N$,
with very high probability one has
\begin{equation}\label{eqn:relax}
|s_i(t)- s'_i(t)|\leq N^{\varepsilon}\cdot\frac{|i|}{N}\cdot\left(\frac{1}{\sqrt{Nt}}+\max(\frac{|i|}{N},t)\right).
\end{equation}
\end{proposition}

%
%
%
%

\begin{remark} When the $b_i$'s are independent, 
Proposition \ref{p:universality} follows by \cite{chelopatto2019}. 
We will give a different, shorter proof,  which also covers correlated,  general continuous martingales with $\rd \langle b_i\rangle_t/\rd t\leq 1$. 
\end{remark}

\begin{remark}
In Proposition~\ref{p:universality}, we assumed that the two densities are the same at $0$: $\widetilde \rho_0(0)= \widetilde \rho'_0(0)=A$. If $\widetilde \rho_0(0)\neq  \widetilde \rho'_0(0)=A$, let $a=\wt \rho_0'(0)/\wt \rho_0(0)$. We can first do a rescaling for the process $\widehat s_i(t)= s_i(a^2 t)/a$, the empirical density of $\{\wh s_i(0)\}_{1\leq |i|\leq N}$ is regular around the rescaled density $a\wt \rho_0(ax)$, and $a\wt \rho_0(0)=A= \wt \rho'_0(0)$. 
Moreover, the process $\{\wh s_i(t)\}_{1\leq |i|\leq N}$ satisfies the following system of SDEs:
\begin{align}
\rd \wh s_i(t)&=\frac{1}{\sqrt{2N}}\rd \wh b_i^s (t)+\frac{1}{2N}\sum_{j\neq i}\frac{1}{\wh s_i(t)-\wh s_j(t)}\rd t,\qquad  \wh b_i^s (t)=\frac{b_i^s(a^2t)}{a}.
\end{align}
From the construction,  $\langle \wh b^s_i\rangle_t=\langle b^s_i\rangle_{a^2t}/a^2$. Thus  if 
$\rd \langle b^s_i\rangle_t/\rd t\leq 1$ a.e. as in \eqref{eqn:bracketBound}, then we also have $\rd \langle \wh b^s_i\rangle_t/\rd t\leq 1$ a.e. If Assumption \ref{e:assump} holds for the pair $(\{\wh b_i^s\}_{1\leq |i|\leq N}, \{ b_i^r\}_{1\leq |i|\leq N})$, then Proposition \ref{p:universality} holds for the pair $(\{\wh s_i(t)\}_{1\leq |i|\leq N}, \{s'_i(t)\}_{1\leq |i|\leq N})$.
\end{remark}

\begin{remark}
The error term $\max(\frac{|i|}{N},t)$ is due to the possibility of $\tilde\mu_0\neq\tilde\mu_0'$, i.e. it is only due to the propagation of error due to the initial  large gaps  $s_j(0)-s_j'(0)$, for large $j$.
For application such as the Hermitianized matrices considered in this article,  $\tilde\mu_0=\tilde\mu_0'$ and the relaxation statement is simply
\[
|s_i(t)-s'_i(t)|\leq N^{\varepsilon}\cdot \frac{|i|}{N\sqrt{Nt}}.
\]
\end{remark}

\begin{proof}[Proof of Theorem~\ref{theo:mainthmdbm}]
We first prove that (\ref{eq:evaluesclosind}) holds for any $t\in[N^{-1+\omega_t},2N^{-1+\omega_t}]$ for some $\omega,\omega_t>0$, and then at the end of the proof we will consider the remaining regime $t\in (2N^{-1+\omega_t},\mathcal{T}]$.

Consider the matrix flow
\begin{equation}
\label{eq:matDBMmata}
\dif X_s=\frac{\dif B_s}{\sqrt{N}}, \qquad X_0=X.
\end{equation}
Here $B_s$ is a matrix valued Brownian motion with entries being i.i.d. standard complex Brownian motions. Fix $z_l$, for $l=1,2$, and denote the Hermitization of $X_s-z_l$ by $H_s^{z_l}$ (see the analogous definition \eqref{eq:herm}). Then \eqref{eq:matDBMmata} induces the following flow on the eigenvalues $x_i^{z_l}(s)$ of  $H_s^{z_l}$:
\begin{equation}
\label{eq:singvaldbma}
\dif x_i^{z_l}(s)=\frac{\dif b_i^{z_l}(t)}{\sqrt{2N}}+\frac{1}{2N}\sum_{j\ne i}\frac{1}{x_i^{z_l}(s)-x_j^{z_l}(s)}\dif s,
\end{equation}
with $b_{-i}^{z_l}(s)=-b_i^{z_l}(s)$. Let ${\bm w}_i^{z_l}(s)$ be the eigenvectors of $H_s^{z_l}$ (they have the form \eqref{eq:defevectors}), then, for $1\le i,j\le N$, we have
\begin{equation*}
\frac{\rd}{\rd s}\langle b_i^{z_l},b_j^{z_l}\rangle_s=\delta_{ij}, \qquad\quad  \frac{\rd}{\rd s}\langle b_i^{z_1},b_j^{z_2}\rangle_s=4\Re\big[\langle {\bm u}_i^{z_1}(s),{\bm u}_j^{z_2}(s)\rangle\langle {\bm v}_j^{z_2}(s),{\bm v}_i^{z_1}(s)\rangle\big].
\end{equation*}
Additionally, by \eqref{eq:evass}, it follows that
\begin{equation*}
\frac{\rd}{\rd s}\langle b_i^{z_1},b_j^{z_2}\rangle_s\le N^{-2\wt\omega}, \qquad\quad 1\le i,j\le N^{\omega_K},
\end{equation*}
with very high probability. 

By our assumption \eqref{eq:weakllaw} we readily see that the initial condition of \eqref{eq:matDBMmata} satisfies Assumption~\ref{d:regular}. It follows that the measure $\zeta$ has a density $\wt\rho_0$. We denote its value at origin by $A=\wt \rho_0(0)$ and we denote the semi-circle distribution by $\rho_{\rm sc}$. Let 
\begin{align}
\label{e:defa}
a=\rho_{\rm sc}(0)/A,
\end{align}
 then $\wt \rho_0(x)$ and $\rho_{\rm sc}(x/a)/a$ matches at origin, both equal to $A$.

Denote $\mathbf{b}:=(b_1^{z_1},\dots, b_{N^{\omega_K}}^{z_1}, b_1^{z_2},\dots, b_{N^{\omega_K}}^{z_2})\in \R^{2N^{\omega_K}}$, for simplicity of notation we assume that $N^{\omega_K}$ is an integer. Next, by the martingale representation theorem (see e.g. \cite[Theorem 18.12]{Kal02}) we have
\[
\dif \mathbf{b}(t)=\sqrt{C(t)}\dif {\bm \beta}(t),
\]
where $C(t)$ is the $(2N^{\omega_K})\times (2N^{\omega_K})$ covariance matrix of the vector $\mathbf{b}(t)$, and ${\bm \beta}(t)\in\R^{2N^{\omega_K}}$ is a vector of standard real i.i.d. Brownian motions.  We now consider two fully independent processes $y_i^{(l)}(t)$, for $l=1,2$, with $y_i^{(l)}(t)$ the eigenvalues of the Hermitization of a Ginibre matrix rescaled by $a$ (recall from \eqref{e:defa}),  satisfying \eqref{eq:singvaldbma} with $b_i^{z_l}(t)$ being replaced by $\beta_i^{(l)}$, where $\beta_i^{(1)}$, $\beta_i^{(2)}$ are a family of $2N$ i.i.d. standard real Brownian motions such that $\beta_i^{(1)}=({\bm \beta})_i$, for $1\leq i\leq N^{\omega_K}$, and $\beta_{i-N^{\omega_K}}^{(2)}=({\bm \beta})_i$, for  $N^{\omega_K}+1\leq i\leq 2N^{\omega_K}$. For these martingales, with very high probability, we thus have
\[
\left|\frac{\dif}{\dif t}\langle \mathbf{b}_i-{\bm \beta}_i, \mathbf{b}_j-{\bm \beta}_j\rangle_t\right|=\left[\left(\sqrt{C(t)}-I\right)^2\right]_{ij}\le \frac{N^{2\omega_K}}{N^{4\tilde{\omega}}}\le N^{-\tilde{\omega}},
\]
where in the last inequality we used that $\omega_K\le \tilde{\omega}/10$ (see the display below \eqref{eqn:phi}). This shows that $\{b_i^{z_l}\}_i$, $\{\beta_i^{(l)}\}_i$ satisfy Assumption~\ref{e:assump}. Additionally,
by our assumption \eqref{eq:weakllaw} we readily see that the initial condition of \eqref{eq:matDBMmata} satisfies Assumption~\ref{d:regular}. Moreover, the initial condition of $y_i^{(l)}(t)$ also satisfies Assumption~\ref{d:regular} with respect to the scaled semi-circle distribution $\rho_{\rm sc}(x/a)/a$. And by \eqref{e:defa} $\wt \rho(0)=A=\rho_{\rm sc}(0)/a$.

We can thus apply Propositions~\ref{p:independence}--\ref{p:universality} to see that
\begin{equation}
\label{eq:finest}
\big|x_i^{z_l}(t)-y_i^{(l)}(t)\big|\le N^{-1-\omega}, \qquad\quad 1\leq |i|\le N^{\omega},
\end{equation}
with very high probability. More precisely, to obtain \eqref{eq:finest}, we first used Proposition~\ref{p:independence} to show that $|x_i^{z_l}(t)-\widetilde{y}_i^{(l)}(t)|\le N^{-1-\omega}$, where $\widetilde{y}_i^{(l)}(t)$ is the solution of
\[
\dif \widetilde{y}_i^{(l)}(s)=\frac{\dif \beta_i^{(l)}(t)}{\sqrt{2N}}+\frac{1}{2N}\sum_{j\ne i}\frac{1}{\widetilde{y}_i^{(l)}(s)-\widetilde{y}_j^{(l)}(s)}\dif s
\]
with initial data $\widetilde{y}_i^{(l)}(0)=x_i^{z_l}(0)$, and for $\beta_i^{(l)}$ as above. This first step does not change the initial data but replaces the driving martingales in \eqref{eq:singvaldbma} with i.i.d. standard real Brownian motions. Then, in the second step, we used Proposition~\ref{p:universality} to show $|\widetilde{y}_i^{(l)}(t)-y_i^{(l)}(t)|\le N^{-1-\omega}$. These two steps together give~\eqref{eq:finest}.  We notice that $\{y_i^{(l)}(t)\}_{1\leq |i|\leq N}$ have the same law as the eigenvalues of the Hermitization of a Ginibre matrix rescaled by $\sqrt{a^2+t^2}=(1+\oo(1))\rho_{\rm sc}(0)/\wt \rho_0(0)$.  The transfer from \eqref{eq:finest} on the (signed) eigenvalues $(x_i,y_i)$ to the singular values $(\lambda_i,\mu_i)$ is elementary. This leads to the claim \eqref{eq:evaluesclosind} for $t\in[N^{-1+\omega_t},2N^{-1+\omega_t}]$.

Finally,  for $t$ up to $\mathcal{T}$, the same reasoning applies after changing the initial condition and time to 
$t_0=t-N^{-1+\omega_t}$ and running the dynamics on $[t_0,t]$  (the  local law (\ref{eq:weakllaw}) is satisfied at time $t_0$ with very high probability, see Proposition \ref{prop:LocalLaw}).
\end{proof}

\subsection{Proof of Proposition \ref{p:independence}.} \,
We interpolate the two processes in \eqref{e:sprocess} and \eqref{e:rprocess}, by letting $0\leq \al\leq 1$,
\begin{align}\label{e:coupleB}
\rd x_i(t,\alpha)=\frac{\al \rd b_i^s (t)+(1-\al) \rd b_i^r (t)}{\sqrt{2N}}+\frac{1}{2N}\sum_{j\neq i}\frac{1}{x_i(t,\alpha)-x_j(t,\alpha)}\rd t, \quad 1\leq |i|\leq N,
\end{align}
where 
$
x_i(0,\al)=s_i.
$
We remark that for $\al=1$ we have $x_i(t,\al)=s_i(t)$, and for $\al=0$ we have $x_i(t,\al)=r_i(t)$. 

The following lemma gives some estimates of $\{x_i(t,\al)\}$, which will be used repeatedly in the rest of this section. We defer its proof to Section \ref{s:locallaw}, after the statement of the local law,  Proposition \ref{prop:LocalLaw}.

\begin{lemma}\label{l:xixj}
Adopt the assumptions in Proposition \ref{p:independence}, and fix a time $\sfT=\varphi^{-2}$. With very high probability there exists a small constant\footnote{In this section the small constant $c_R$ may vary from line to line but only depends on $R$ from the $(\nu, R)$-regularity assumption.} $c_R>0$ such that  uniformly for $0\leq t\leq \sfT$,   we have
\begin{align}\label{e:gapb}
&c_R\,\frac{|i-j| }{N}\leq |x_i(t,\al)-x_j(t,\al)|\leq \frac{|i-j|}{c_R\,N},\quad |i-j|\geq \varphi^{10}, \quad |i|,|j|\leq c_R N,\\
&\label{e:moveest}
|x_i(t,\al)-x_i(0,\al)|\leq \varphi^{3}\left( t+ \frac{1}{N}\right),\quad |i|\leq c_R\, N.
\end{align}
\end{lemma}

If we take derivative with respect to $\alpha$ on both sides of \eqref{e:coupleB}, we obtain the tangential equation for $u_i(t,\al)=\partial_\al x_i(t,\al)$ (see \cite{LanSosYau2019}), which is given by 
\begin{align}\begin{split}\label{e:tangential}
\rd u_i(t,\alpha)&=\rd \xi_i(t)-\sum_{j\neq i}c_{i,j}(t)\,(u_i(t,\alpha)-u_j(t,\alpha))\rd t, \\
\xi_i(t):&=\frac{ \rd b_i^s (t)-\rd b_i^r (t)}{\sqrt{2N}},\\
c_{ij}(t):&=\frac{1}{2N(x_i(t,\al)-x_j(t,\al))^{2}},
\end{split}\end{align}
where $u_i(0,\al)=0$ for all $1\leq |i|\leq N$. The key for the proof of Proposition \ref{p:independence} is the following lemma.

\begin{lemma}\label{l:intpdiff}
Under the  assumptions from Proposition \ref{p:independence},  for any $|i|\leq N^\omega$ and $\alpha\in[0,1]$, with  very high probability we have
$
|u_i(t_f, \al)|\leq N^{-1-\omega}.
$
\end{lemma}

To prove Lemma \ref{l:intpdiff}, we define the operator $\cB=\cB(t)$ through
\begin{align*}
(\cB v)_{i}=-\sum_{j\neq i} c_{ij}(t)(v_i-v_j)
\end{align*}
and we first study the linear differential equation
\begin{align}\label{e:tangential_eq}
\frac{\rd}{\rd t} v_i(t)=(\cB v(t))_i, \quad 1\leq |i|\leq N.
\end{align}

We need to further introduce the  long range and short range operators. Take $c<c_R/2$. Denoting  $I_{\cL}=\{(i,j): |i-j|\geq \ell, \min\{|i|,|j|\}\leq cN\}$, and $I_{\cS}=\{(i,j): i\neq j\}\setminus I_{\cL}$, let 
\begin{equation}
(\cS v)_{i}=-\sum_{(i,j)\in I_{\cS}}c_{ij}(t)(v_i-v_j),\label{eqn:short}\ \ \ \ \ \ \ \ 
(\cL v)_{i}=-\sum_{(i,j)\in I_{\cL}}c_{ij}(t)(v_i-v_j).
\end{equation}
In this definition,  we choose the dynamics cutoff parameter
\begin{equation*}
\ell=N^{\omega_\ell},\ \ \ \omega_t\ll \omega_\ell<\omega_K.
\end{equation*}
In particular $\omega_\ell>10\nu$, so  thanks to Lemma \ref{l:xixj},  for any $(i,j)\in I_{\cL}$ we have
\begin{align}\label{e:Lijbound}
|c_{ij}(t)|\leq \frac{N}{c_R|i-j|^2}. 
\end{align}
With $\bm\xi$ defined in (\ref{e:tangential}),  we consider the following evolution using the short range operator $\cS$,
\begin{align}\label{e:tvequation}
\rd \bmw=\cS\bmw +\rd \bm\xi,\quad  \bmw(0)=0.
\end{align}

\begin{lemma}\label{l:wbound}
There exists a constant $c>0$ such that for any $0\leq t\leq 1$,  with probability  $1-e^{-c(\log N)^2}$ we have
\begin{align*}
\|\bmw(t)\|_\infty\leq (\log N)^2\sqrt{\frac{t}{N}}.
\end{align*}
Moreover, this estimate holds for any choice of $\ell$, in particular when $\cS$ is replaced by $\cB$ in (\ref{e:tvequation}).
\end{lemma}

\begin{proof} 
Let $q=2p$ with $p\in\mathbb{N}_*$ and $f(t)=\E[\sum_{|k|\leq N}w_k(t)^q]$.
From It{\^o}'s lemma and (\ref{e:tvequation}) we have
\[
\del_t f(t)=-\frac{q}{2}\sum_{(i,j)\in I_\cS}\E[c_{i,j}(w_i-w_j)(w_i^{q-1}-w_j^{q-1})]+\sum_{1\leq |i|\leq N}\frac{q(q-1)}{2}\E\big[w_i(t)^{q-2}
\frac{\rd\langle\xi_i\rangle_t}{\rd t}\big].
\]
The first term on the right-hand side is non-positive, and the second one is bounded thanks to 
$\rd\langle\xi_i\rangle_t/\rd t\lesssim N^{-1}$ and repeated H{\"o}lder's inequalities:
\[
\E\left[\sum_{1\leq |i|\leq N}w_i(t)^{q-2}\right]\leq 2N \E\left[\Big(\frac{1}{2N}\sum_{1\leq |i|\leq N}w_i(t)^{q}\Big)^{\frac{q-2}{q}}\right]
\leq 
2N\E\left[\frac{1}{2N}\sum_{1\leq |i|\leq N}w_i(t)^{q}\right]^{\frac{q-2}{q}}=(2N)^{\frac{2}{q}}f(t)^{\frac{q-2}{q}}.
\]
We have obtained 
$
\del_t f(t)\leq \frac{p(2p-1)(2N)^{1/p}}{N}f(t)^{\frac{p-1}{p}}
$
so with Gr{\" o}nwall's inequality
\[
\E[\|\bw(t)\|_\infty^{2p}]\leq f(t)\leq2N\cdot (2p-1)^p\cdot \left(\frac{t}{N}\right)^p.
\]
The conclusion follows easily by Markov's inequality.
\end{proof}

Denote the propagator  for (\ref{e:tangential_eq}) $\cU=\cU(s,t,\al)$ ($s\leq t$), i.e.  
$
v_i(t)=\sum_{1\leq|j|\leq N}\cU_{ij}(s,t, \al)v_j(s).
$
The propagator preserves the sign ($\cU_{ij}\geq 0$) and 
the mass ($\sum_{1\leq |j|\leq N} \cU_{ij}=1$).
The difference of \eqref{e:tangential} and \eqref{e:tvequation} gives
$
\rd (\bmu- \bmw)=\cB(\bmu- \bmw) +\cL \bmw
$
so that,  with Duhamel's formula, 
\begin{align}\label{e:uit}
\bmu(t)=
 \bmw(t)
+\int_0^t \cU(s,t, \alpha)  \cL (s)  \bmw(s) \rd s.
\end{align}
For any $1\leq |i|\leq N$, thanks to Lemma \ref{l:wbound}, we conclude
\begin{align}\begin{split}\label{e:uij}
\sum_{1\leq |j|\leq N} \cU_{ij} (s,t)(\cL (s)  \bmw(s))_j
&\leq \sum_{1\leq |j|\leq N} \cU_{ij} (s,t)\sum_{|j-k|\geq \ell}\frac{CN}{|j-k|^2}|w_k(s)|
\leq \sum_{1\leq |j|\leq N} \cU_{ij} (s,t) \sum_{|j-k|\geq \ell}\frac{CN}{|j-k|^2}\|\bmw(s)\|_\infty\\
&\leq \sum_{1\leq |j|\leq N} \cU_{ij} (s,t)\frac{CN}{\ell}\|\bmw(s)\|_\infty
\leq \frac{CN}{\ell} \|\bmw(s)\|_\infty\leq \frac{CN(\log N)^2\sqrt{Ns}}{\ell},
\end{split}\end{align}
with probability $1-\OO(N^{-D})$ for any fixed $s$, where $C$ does not depend on $s$.
In the first inequality we used \eqref{e:Lijbound}; in the second line we bound $|w_k(s)|$ by $\|\bmw(s)\|_\infty$; in the last line we first sum over $k$, then sum over $j$, and also used Lemma \ref{l:wbound}.

Plugging \eqref{e:uij} into \eqref{e:uit} we obtain that with probability $1-\OO(N^{-D})$ we have
\begin{align}\label{e:uit2}
\bmu(t)=
 \bmw(t)
+\OO\left(\int_0^t \frac{C(\log N)^2\sqrt{Ns}}{\ell} \rd s\right)
= \bmw(t)
+\OO\left(\frac{(\log N)^2\sqrt{Nt^3}}{\ell}\right)
= \bmw(t)
+\OO\left(N^{-1-\omega}\right),
\end{align}
provided that $\ell\gg (\log N)^2 N^{3\omega_t/2+\omega}$. 
Note that the first inequality in (\ref{e:uit2}) holding with very high probability requires an intermediate step based on the Markov and H{\"o}lder inequalities,  because the good sets where $(\ref{e:uij})$ holds may not have an intersection with large probability as $s$ varies.  We refer to \cite[proof of Theorem 2.8]{Bou2022} for this elementary step.

Lemma~\ref{l:intpdiff} now follows from combining \eqref{e:uit2} and the following estimates on $\bmw(t)$.

\begin{lemma}
Let ${\bm w}$ be the solution of \eqref{e:tvequation}.  Let
$\ell=N^{\omega_\ell}$ and
$\alpha\in[0,1]$.  With  very high probability,  for any  $1\leq |i|\leq \ell$, 
$\nu, \omega\ll \omega_t\ll \omega_\ell<\omega_K \ll\wt\omega$ and $t=N^{-1+\omega_t}$,
we have
\[
|w_i(t)|\leq \frac{1}{N^{1+\omega}}.
\]
\end{lemma}

\begin{proof}
Denote  $\zeta=N/\ell$ and $\chi(x)$ be nonnegative,  smooth, compactly supported, and such that $ \int\chi=1$. 
In the rest of this proof, we denote, $b_k(t)=\al b_k^s(t) +(1-\al) b_k^r(t)$ (from \eqref{e:coupleB}).
Define $\psi(x)=\int \min\{|x-y|, \frac{c}{2c_R}\}\zeta\chi(\zeta y)\rd y$ (where $c_R$ is introduced in \eqref{e:gapb} and $c$ appears in the definition of $I_{\cL}$ before (\ref{eqn:short}))  and
\[
\psi_k(s)=\psi(x_k(s, \al)),\ \phi_k(s)=e^{-\zeta\psi_k(s)},\ v_k(s)=\phi_k(s)w_k(s).
\]
From the construction $\psi(x)=\frac{c}{2c_R}$ for $|x|\geq \frac{c}{c_R}$. 
The It\^o formula gives 
\begin{align*}
\rd v_k
=& \sum_{ (j,k)\in I_{\cS}} c_{jk} \left((v_j- v_k)+ \left( \frac {  \phi_k}  {  \phi_j} - 1 \right)  v_j \right) \rd s  +\left(\rd  \phi_k \right)  w_k
-\zeta \psi'(x_k)\phi_k(s)\frac{\rd\langle b_k, \xi_k \rangle_s}{\sqrt{2N}}+\phi_k\rd\xi_k,\\
\frac{\rd \phi_k}{\phi_k} 
=&- \zeta \psi'(x_{k})\frac{\rd b_{k}(s)}{\sqrt{2N}}-\zeta \frac{\psi'(x_{k})}{2N}\sum_{j\neq k}\frac{\rd s}{x_{k}-x_j}
+\frac{1}{2}\left(\frac{\zeta}{2N}\psi''(x_{k})+\frac{\zeta^2}{2N}\psi'(x_{k})^2\right)\rd\langle b_k\rangle_{s}.
\end{align*}
Thus if we define $X_s=\sum_{1\leq |k|\leq N}v_k(s)^2$, we obtain
\begin{align}
\rd X_s=&\sum_{1\leq |k|\leq N}-2\zeta\psi'(x_k) v_k^2 \frac{\rd b_k(s)}{\sqrt{2N}}
+2\sum_{1\leq |k|\leq N} v_k \phi_k \frac{\rd \xi_k(s)}{\sqrt{2N}}\label{stoc}\\
&-\sum_{(j,k)\in I_\cS}c_{jk}(v_j-v_k)^2\rd s\label{new1}\\
&+\sum_{(j,k)\in I_\cS}c_{jk}\left(\frac{\phi_k}{\phi_j}+\frac{\phi_j}{\phi_k}-2\right)v_jv_k\rd s
\label{new2}\\
&+\frac{\zeta}{2N}\sum_{1\leq |k|\leq N}  \psi''(x_k)v_k^2\rd\langle b_k\rangle_{s}\label{new3}\\
&+\frac{\zeta^2}{2N}\sum_{1\leq |k|\leq N}  \psi'(x_k)^2v_k^2\rd\langle b_k\rangle_{s}\label{new4}\\
&-\frac{\zeta}{N}\sum_{j<k}\frac{\psi'(x_j)v_j^2-\psi'(x_k)v_k^2}{x_j-x_k}\rd s\label{new5}\\
&-\sum_{1\leq |k|\leq N}   2 \zeta \psi'(x_k)\phi_kv_k\frac{\rd\langle b_k, \xi_k \rangle_s}{\sqrt{2N}}
+\sum_{1\leq |k|\leq N}  \rd \left\langle \phi_k \xi_k-\zeta \psi'(x_k) v_k\frac{b_k}{\sqrt{2N}}\right\rangle_s
\label{new6}\\
&+\sum_{1\leq |k|\leq N}  \phi_k^2 \rd \langle \xi_k,  \xi_k\rangle_s
\label{new7}.
\end{align}

Consider 
\begin{align}\label{e:tautime}
\tau=t\wedge\inf\left\{s: \eqref{e:gapb} \text{ fails}\right\}\wedge \inf\left\{s\geq 0: X_s>N^{-2-2\omega}\right\}.
\end{align}

For $(j,k)\in I_\cS$ there are two cases: either $\min\{|j|, |k|\}\leq cN, |j-k|\leq \ell$ or $\min\{|j|, |k|\}\geq cN$.
From $\|\psi'\|_\infty\leq 1$ and the assumption $\zeta=N/\ell$, in the first case we have $\zeta|\psi(x_k)-\psi(x_j)|\leq 
\zeta|x_k-x_j|\leq \tfrac{|k-j|}{c_R\ell}=\OO(1)$, where we have used that for $t\leq \tau$ from \eqref{e:gapb} we have $|x_k(t)-x_j(t)|\leq |k-j|/(c_RN)$ when $|k|,|j|\leq c_R N$, and we choose $c$ small enough compared to $c_R$ (say $c=c_R/2$). This implies that 
$\left|\frac{\phi_k}{\phi_j}+\frac{\phi_j}{\phi_k}-2\right|=\OO(1) \zeta^2|x_k-x_j|^2$. In the second case, $|x_k|, |x_j|\geq c/c_R$, and  $\zeta|\psi(x_k)-\psi(x_j)|=0$. One concludes easily that 
(\ref{new2}) is   $\OO(1)\zeta^2(\ell/N)X_s$. The terms (\ref{new3}) and (\ref{new4}) are of smaller  order by $\|\psi'\|_\infty\leq 1$,  $\|\psi''\|_\infty\leq \zeta$ and $\tfrac{\rd}{\rd t}\langle b_k\rangle\leq 1$.

For $s\leq \tau$, in (\ref{new5}) the sums of $\max\{|j|,|k|\}\geq cN$ (say $|k|\geq cN$) are bounded as
\begin{align*}
2\frac{\zeta}{N}\sum_{|j|\leq N^{1-2\nu}, |k|\geq N^{1-\nu}}\frac{\psi'(x_j)v_j^2}{x_j-x_k}\leq \zeta \sum_{1\leq |j|\leq N}v_j^2=\zeta X_s \leq \frac{\xi}{N^2}=\frac{1}{\ell N}.
\end{align*}
The sums of $|j|, |k|\leq cN$ is of order at most
\[
\frac{\zeta}{N}\sum_{(j,k)\in I_\cL}\frac{v_k^2}{|x_j-x_k|}
+
\frac{\zeta}{N}\sum_{(j,k)\in I_\cS}|\psi'(x_j)|\frac{|v_j^2-v_k^2|}{|x_j-x_k|}
+
\frac{\zeta}{N}\sum_{(j,k)\in I_\cS}\|\psi''\|_\infty v_k^2.
\]
For $s\leq \tau$, the first sum above has order $\zeta(\log N)X_s\leq \zeta(\log N)/N^2=\log N/(N\ell)$. The third sum is at most $(\xi^2\ell/N) X_s\leq 1/(N\ell)$.
Finally, the second sum is bounded using
\[
2\frac{|v_j^2-v_k^2|}{|x_j-x_k|}\leq M^{-1}(v_j+v_k)^2+M\frac{(v_j-v_k)^2}{(x_j-x_k)^2}.
\]
Choosing $M=\e\zeta^{-1}$ for $\e$ small enough, this proves that (\ref{new5}) can be absorbed into the dissipative term (\ref{new1}) plus 
an error of order $(\zeta^2\ell/N) X_s=1/(N\ell)$.

In (\ref{new6}) and (\ref{new7}), using $\|\psi'\|_\infty\leq 1$, we can bound them by (up to a constant)
\begin{align*}
\sum_{1\leq |k|\leq N}  \frac{\zeta \phi_k v_k}{2N}+ \frac{\zeta^2 v_k^2}{2N}+\phi_k^2 \frac{\rd \langle \xi_k, \xi_k\rangle_s}{\rd s}
=
\sum_{1\leq |k|\leq N} \left( \frac{\zeta \phi_k v_k}{2N}+\phi_k^2 \frac{\rd \langle \xi_k, \xi_k\rangle_s }{\rd s}\right)+\frac{\zeta^2 X_s}{2N}.
\end{align*}
Let $K'=\min\{K, \ell^{3/2}\}=N^{\min\{\omega_K, 3\omega_\ell/2\}}\leq K$. For $s\leq \tau$, if $|k|\geq K'$, then $\psi(x_k(s))\geq K'/N=(K'/\ell)\zeta^{-1}$,  so that $\phi_k(s)\leq e^{-K'/\ell}$ is negligible when $\omega_\ell<\omega_K$. For $|k|\leq K'$, we simply bound $|\phi_k|\leq 1$, then
\begin{align*}
\sum_{1\leq |k|\leq N} \left( \frac{\zeta \phi_k v_k}{2N}+\phi_k^2 \frac{\rd \langle \xi_k, \xi_k\rangle_s}{\rd s} \right)
&=\sum_{|k|\leq K'}\left( \frac{\zeta \phi_k v_k}{2N}+\phi_k^2 \frac{\rd \langle \xi_k, \xi_k\rangle_s}{\rd s} \right)+\sum_{|k|\geq K'}\left( \frac{\zeta \phi_k v_k}{2N}+\phi_k^2 \frac{\rd \langle \xi_k, \xi_k\rangle_s}{\rd s} \right)\\
&=\sum_{|k|\leq K'}\left( \frac{\zeta  v_k}{2N}+ \frac{\rd \langle \xi_k, \xi_k\rangle_s}{\rd s} \right)+e^{-K'/\ell}\sum_{1\leq |k|\leq N} \left( \frac{\zeta  v_k}{2N}+ \frac{\rd \langle \xi_k, \xi_k\rangle_s}{\rd s} \right)\\
&\lesssim 
\frac{\zeta K'}{N}\sqrt{\sum_{|k|\leq K'} \frac{v^2_k}{2K'}}+\sum_{|k|\leq K'}\frac{N^{-\wt \omega}}{2N}+e^{-K'/\ell}\left( \zeta \sqrt{\frac{\sum_{1\leq |k|\leq N}  v_k^2}{N}}  +\sum_{1\leq |k|\leq N}  \frac{1}{N} \right)\\
&\lesssim 
\frac{1}{N^{1+\omega_\ell/4}}+\frac{1}{N^{1-\omega_K+\wt\omega}},
\end{align*}
where in the third line we used that for $|k|\leq K$, $\rd \langle \xi_k, \xi_k\rangle_s/\rd s\lesssim N^{-\wt\omega}/N$; in the last line we used that for $s\leq \tau$, $\sum_{1\leq |k|\leq N}  v_k^2(s)=X_s\leq 1/N^2$, and $\sqrt{K'}/\ell\leq N^{-\omega_\ell/4}$.

We now  bound the relevant martingales in \eqref{stoc} in terms of the quadratic variation.  
From \cite[Appendix B.6,  Equation (18)]{ShoWel2009} with  $c=0$ allowed for continuous martingales,  for any continuous martingale $M$ and any $\lambda,\mu>0$, we have 
$
\mathbb{P}\Big(\sup_{0\leq u\leq t}|M_u|\geq \lambda,\ \langle M\rangle_{t}\leq \mu\Big)\leq 2 e^{-\frac{\lambda^2}{2\mu}}
$. This 
implies that there exists a $c$ such that for any $N,\e>0$ we have
\begin{equation}\label{eqn:boundBr}
\mathbb{P}\left(\sup_{0\leq u\leq t}|M_u|\geq \varphi^{\e}\langle M\rangle_t^{1/2}\right)\leq c^{-1}e^{-c\varphi^{2\e}}.
\end{equation}
The stochastic terms in \eqref{stoc} have quadratic variations bounded as follows.
First, for $s\leq \tau$
\begin{align*}
\frac{\rd}{\rd s}\left\langle\sum_{1\leq |k|\leq N} -\xi\psi'(x_k) v_k^2 \frac{ b_k(s)}{\sqrt{2N}}\right\rangle
\lesssim \frac{\xi^2}{N}\left(\sum_{1\leq |k|\leq N}  v_k^2\right)^2=\frac{\xi^2 X^2_s}{N}\lesssim \frac{1}{\ell^2 N^{3+4\omega}},
\end{align*}
Again, for $s\leq \tau$ and $k\geq K$, $\phi_k(s)\leq e^{-K/\ell}$ is negligible. This implies
\begin{align*}
\frac{\rd}{\rd s}\left\langle\sum_{1\leq |k|\leq N}  v_k \phi_k \frac{ \xi_k}{\sqrt{2N}}\right\rangle
&
\leq \sum_{|k|,|j|\leq K}v_kv_j \phi_k \phi_j \frac{N^{-\wt\omega}}{2N}
+\sum_{|k|\geq K \text{  or }|j|\geq K}v_kv_j \phi_k \phi_j \frac{C}{2N}\\
&\leq\frac{1}{N^{1+\wt\omega-\omega_K}} \sum_{1\leq |k|\leq N} v_k^2 + e^{-K/\ell}\sum_{1\leq |k|\leq N} v_k^2
\lesssim \frac{1}{N^{1+\wt\omega-\omega_K}} X_s\leq \frac{1}{N^{3+\wt\omega+2\omega-\omega_K}}.
\end{align*}
Therefore with very high probability, 
\begin{align*}
\sup_{t\leq \tau} \left|\int_0^s \sum_{1\leq |k|\leq N}  -\xi\psi_k'(u)v_k^2(u)\frac{\rd b_k(u)}{\sqrt{2N}}\right|
\leq \frac{\sqrt t}{\ell N^{3/2}},\quad
\sup_{t\leq \tau} \left|\int_0^s \sum_{1\leq |k|\leq N}  v_k \psi_k\frac{\rd \xi_k}{\sqrt{2N}} \right|
\leq \frac{\sqrt t}{N^{3/2} N^{(\wt\omega-\omega_K)/2}}.
\end{align*}

We  have thus  proved that with very high probability, for any $0\leq s\leq \tau$, 
\begin{align*}
X_s
&\leq C\left(\frac{\log N t}{N\ell }+\frac{\sqrt t}{\ell N^{3/2}}+  \frac{\sqrt t}{N^{3/2} N^{(\wt\omega-\omega_K)/2}}+\frac{t}{N^{1+\omega_\ell/4}}+\frac{t}{N^{1-\omega_K+\wt\omega}}\right)\\
&\leq 
\frac{C}{N^2}\left(\frac{1}{N^{\omega_\ell-\omega_t} }+\frac{1}{N^{\omega_\ell-\omega_t/2}}+  \frac{1}{ N^{(\wt\omega-\omega_K-\omega_t)/2}}+\frac{1}{N^{\omega_\ell/4-\omega_t}}+\frac{1}{N^{\wt\omega-\omega_K-\omega_t}}\right)
\ll \frac{1}{N^{2+2\omega}},
\end{align*}
where we used $\omega\ll \omega_t\ll\omega_\ell< \omega_K \ll\wt\omega$. We conclude that with very high probability $\tau=t$, and 
$X_t\ll N^{-2-2\omega}$. Thanks to the choice of the stopping time \eqref{e:tautime}, $|x_i(t,\alpha)-x_j(t,\alpha)|\leq |i-j|\varphi/N$. Thus for any $|i|\leq \ell$, $|x_i(t,\alpha)|\lesssim \ell /N$, and $\phi_i(t)\asymp 1$. We conclude that $N^{-2-2\omega}\geq X_t\gtrsim w_i(t)^2$,  so that $|w_i(t)|\leq N^{-1-\omega}$ with  very high probability.
\end{proof}

\begin{proof}[Proof of Proposition \ref{p:independence}]
Proposition \ref{p:independence} follows from Lemma \ref{l:intpdiff} by integrating from $\al=0$ to $\al=1$:
\begin{align}\label{eqn:integrate}
|s_i(t_f)-r_i(t_f)|=|x_i(t_f, 1)-x_i(t_f, 0)|=\left|\int_{0}^{1}\del_\alpha x_i(t_f, \al)\rd \alpha\right|
\leq \int_0^1 |u_i(t_f, \al)|\rd \alpha\leq N^{-1-\omega}. 
\end{align}
Note that the second inequality holds with  very high probability, and requires intermediate steps based on the Markov and H{\"o}lder inequalities, similarly to (\ref{e:uit2}).
\end{proof}

\subsection{Local Law.} \label{s:locallaw} \, In this section, we study a modified  version of Dyson Brownian motion where the driving martingales can be essentially arbitrary.
\begin{align}\label{generalDBM}
\rd x_i(t)=\frac{\rd b_i}{\sqrt{2N}}+\frac{1}{2N}\sum_{j\neq i}\frac{1}{x_i(t)-x_j(t)}\rd t, \quad 1\leq |i|\leq N.
\end{align}
\begin{assumption}\label{a:Bbound}
We assume the martingales satisfy
\begin{align*}
\tfrac{\rd}{\rd t} \langle b_i\rangle_t \leq 1.
 \end{align*}
 \end{assumption}
Under Assumption \ref{a:Bbound}, the existence and strong uniqueness hold for the above stochastic differential equation, see Proposition \ref{prop:WellPosed}. We denote the empirical eigenvalue density and its Stieltjes transform as,
\begin{align*}
\mu_t=\frac{1}{2N}\sum_{1\leq |i|\leq N}\delta_{x_i(t)},\qquad\quad m_t(z)= \frac{1}{2N} \sum_{1\leq |i|\leq N} \frac{1}{x_i(t)-z}.
\end{align*}
We assume that the initial data $\mu_0$ is $(\nu, R)$-regular around $\widetilde \mu_0$ as in Assumption \ref{d:regular}.

In the following, we first recall some notations and concepts about the free convolution with the semicircle distribution. The \emph{semicircle distribution} is a measure $\mu_{\rm sc} \in \mathscr{P} (\mathbb{R})$ whose density $\varrho_{\rm sc} : \mathbb{R} \rightarrow \mathbb{R}_{\ge 0}$ with respect to the Lebesgue measure is given by
	\begin{flalign*}
		\rho_{\rm sc} (x) = \displaystyle\frac{(4-x^2)^{1/2}}{2\pi} \cdot \textbf{1}_{x \in [-2, 2]}, \qquad \text{for all $x \in \mathbb{R}$}. 
	\end{flalign*}
	
	\noindent For any real number $t > 0$, we  denote the rescaled semicircle probability distribution 
	\begin{flalign*}
\mu_{\rm sc}^{(t)} =  t^{-1/2} \rho_{\rm sc} (t^{-1/2} x) \rd x.
	\end{flalign*}
Let $\widetilde \mu_t = \widetilde \mu_0 \boxplus \mu_{\rm sc}^{(t)}$ denote the free convolution between $\wt\mu_0$ with the rescaled semicircle distribution. By \cite[Corollary 2]{biane1997free}, $\widetilde \mu_t$ has a density $\widetilde \rho_t: \mathbb{R} \rightarrow \mathbb{R}_{\ge 0}$ with respect to Lebesgue measure for $t > 0$.	 Following \cite{biane1997free}, the  Stieltjes transform $\widetilde m_t$ of $\widetilde\mu_t$ is characterized in the following way.
For any $t > 0$, denote the function $\Phi = \Phi^{t}: \mathbb{H} \rightarrow \mathbb{C}$ and the set $\Lambda_t \subseteq \mathbb{H}$ by
	\begin{flalign}
		\label{mtlambdat} 
		\Phi(z) = z - t \wt m_0(z); \quad \Lambda_t = \Big\{ z \in \mathbb{H} : \Im \big( z - t\wt m_0 (z) \big) > 0 \Big\} = \Bigg\{ z \in \mathbb{H} : \displaystyle\int_{-\infty}^{\infty} \displaystyle\frac{\wt\mu_0(d x)}{|z-x|^2} < \displaystyle\frac{1}{t} \Bigg\}.
	\end{flalign}

\begin{wrapfigure}[10]{l}{0.4\textwidth}
\vspace{-1.3cm}
\includegraphics[width=\textwidth]{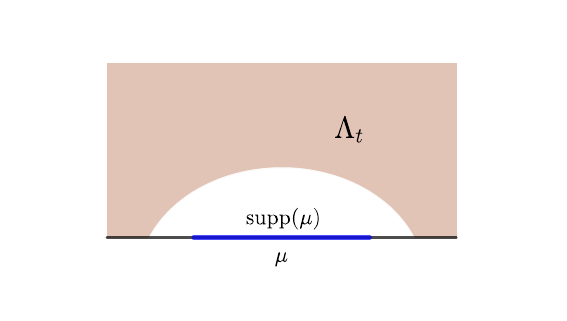}
\vspace{-1cm}
\caption{$\Lambda_t$ as defined in \eqref{mtlambdat} is an open subset of the upper half plane $\mathbb H$. $\Phi(z)$ is a holomorphic map from $\Lambda_t$ to $\mathbb H$. }
\label{f:lambdat}
\end{wrapfigure}

From \cite[Lemma 4]{biane1997free}, the function $\Phi$ is a homeomorphism from $\overline{\Lambda}_t$ to $\overline{\mathbb{H}}$. Moreover, it is a holomorphic map from $\Lambda_t$ to $\mathbb{H}$ and a bijection from $\partial \Lambda_t$ to $\mathbb{R}$ (see Figure \ref{f:lambdat}). For any real number $t \ge 0$,  $\widetilde m_t = : \mathbb{H} \rightarrow \mathbb{H}$ is characterized as
	\begin{flalign*}
		\widetilde m_t \big( u - t \widetilde m_0 (u) \big) = \widetilde m_0 (u), \qquad \text{for any $u\in \Lambda_t$}.
	\end{flalign*}
In the following we denote the characteristic flow 
\begin{align}
\label{e:flow}
z_t(u)=u-t\widetilde m_0(u), \quad \text{for any } u\in \Lambda_t,\quad \widetilde m_t(z_t(u))=\widetilde m_0(u).
\end{align}
If the context is clear, we will simply write $z_t(u)$ as $z_t$. 
Note that $z$ satisfies the characteristics equation
\begin{equation}\label{eqn:char}
\del_t z_t=-\widetilde m_t(z_t).
\end{equation}
Lemma \ref{l:xixj} is an easy consequence of the following proposition. 
\begin{proposition}\label{prop:LocalLaw}
Let $\sfT=\varphi^{-2}$ and recall the definition of $\varphi$ in (\ref{eqn:phi}).
Assume that $\mu_0$ is $(\nu, R)$ regular. Then, 
we have the following local law with high probability for any $t\in[0,\sfT]$
\begin{align*}
\left|m_t(w)-\widetilde m_t(w)\right| \leq \varphi \sqrt{\frac{\Im[\widetilde m_0(w)]}{N\Im[w]}},\quad w\in \{z\in \bH: |z|\leq R/2, \Im[z]\geq \varphi^4 N^{-1+\nu}\}.
\end{align*}
\end{proposition}

\begin{proof}[Proof of Lemma \ref{l:xixj}]
The first statement \eqref{e:moveest} in Lemma \ref{l:xixj} follows from the lower and upper bound of $m_t(w)$, which follows from  Proposition \ref{prop:LocalLaw}. We omit its proof and refer to \cite[Theorem 1.1]{erdHos2009local} for the proof.

To prove \eqref{e:moveest}, we interpolate \eqref{generalDBM} with the standard Dyson Brownian motion,
\begin{align}\label{e:generalDBM}
\rd \wt x_i(t)=\frac{\rd {\wt b_i(t)}}{\sqrt{2N}}+\frac{1}{2N}\sum_{j\neq i}\frac{1}{\wt x_i(t)-\wt x_j(t)}\rd t, \quad 1\leq |i|\leq N.
\end{align}
with initial data $\widetilde{x}_i(0)=x_i(0)$ for $1\leq |i|\leq N$, where $\wt b_1(t), \wt b_2(t), \cdots, b_N(t)$ are independent Brownian motions, and $\wt b_{-i}(t)=-b_i(t)$ for $1\leq i\leq N$. The evolution \eqref{e:generalDBM} is the standard BC-type Dyson Brownian motion (singular-value process of matrix-valued Brownian motions), and can be analyzed by the same way as the standard Dyson Brownian motion. In particular  the optimal rigidity holds (see \cite[Corollary 3.2]{HuaLan2019} and \cite[Theorem 4.5]{chelopatto2019}), i.e. with  very high probability the $i$-th particle is close to the classical locations with error $\varphi/N$:
\begin{align}
|\wt x_i(t)-\gamma_i(t)|\lesssim \frac{\varphi}{N}, \qquad\quad \int_{\gamma_i(t)}^\infty \rd \mu_t(x)=\frac{(N-i+1/2)}{2N}.
\end{align}
We can then interpolate \eqref{generalDBM} with \eqref{e:generalDBM} as in \eqref{e:coupleB}, and by Lemma  \ref{l:wbound}, the interpolation induces an error bounded by $\varphi \sqrt{t/N}$. We also notice that the classical locations $\gamma_i(t)$ move with speed $\OO(t)$.  It thus follows that 
\begin{align*}
|x_i(t)-x_i(0)|\leq |\wt x_i(t)-\gamma_i(t)|+|x_i(t)-\wt x_i(t)|+|\gamma_i(t)-\gamma_i(0)|\lesssim\varphi\left( \frac{1}{N}+\sqrt{\frac{t}{N}}+t\right)\lesssim \varphi\left(\frac{1}{N}+t\right),
\end{align*}
and \eqref{e:moveest} follows.
\end{proof}

Recall we denote $\varphi=N^\nu$. For any real number $t\geq 0$, we define the spectral domain $\cD_t=\cD_t(\widetilde \mu_t)$ by 
\begin{align}\label{e:defDt}
\cD_t\deq\left\{z\in \bH:   |z|\leq R-Ct, \Im[z]\,\Im[\widetilde m_t(z)]\geq \varphi^4/N \right\}
\end{align}
where the constant $C$ is defined in Equation (\ref{eqn:tildeC}).
The following lemma collects some properties of the domain $\cD_t$. 

\begin{lemma}\label{c:Dtproperty}
Recall the domain $\cD_t$ from \eqref{e:defDt}, and  $\sfT=\varphi^{-2}$. For  $N$ large enough and any time $0\leq t\leq \sfT$, the following holds.
\begin{enumerate}[(i)]
\item\label{i:Im}$\cD_t$ is non-empty. For any $z=E_t+\ri\eta\in \cD_t$, we have
\begin{align}\label{e:Immz1}
\Im[\widetilde m_t(z)]\geq \varphi^2 \sqrt{\frac{\Im[\widetilde  m_t(z)]}{N\eta}}.
\end{align}
\item\label{i:zs}If $z_t(u)=E_t+\ri\eta_t\in \cD_t$, then $z_s(u)\in \cD_s$ for any $0\leq s\leq t$.
\end{enumerate}
\end{lemma}
\begin{proof}
From our assumption for $|u|\leq R$, $|\widetilde  m_0(u)|\leq C$, and $|z_t(u)-u|=t|\widetilde  m_0(u)|$. Thus $z_t(u)$ maps $\{u: |u|\leq R\}$ surjectively to $\{z: |z|\leq R-Ct\}$. We also conclude that for $z=z_t(u)$ with $|z|\leq R-Ct$, it holds $\widetilde  m_t(z)=\widetilde  m_0(u)$, and \eqref{e:Immz1} holds from the definition of $\cD_t$. Moreover, for any $z\in \cD_t$, we have $\Im[\widetilde m_t(z)]\geq 1/C$. It is easy to see that $\cD_t$ is not empty, and in fact $\{z\in \bH: |z|\leq R/2, \Im[z]\geq C\varphi^4 N^{-1}\}\subset \cD_t$. The statement \eqref{e:Immz1} follows from the definition of $\cD_t$ from \eqref{e:defDt}.

For $(ii)$  since $\eta_s\geq \eta_t$ from \eqref{e:flow}, if $\eta_t\geq \varphi^4/(N\Im[\widetilde  m_t(z_t)])$, then $\eta_s\geq \eta_t\geq \varphi^4/(N\Im[\widetilde  m_t(z_t)])=\varphi^4/(N\Im[\widetilde  m_s(z_s)])$. Thus,  $z_s (u)$ satisfies the lower bounds required for $z_s (u)$ to be in $\mathcal{D}_s$. 
\end{proof}

\noindent By the second part of Lemma~\ref{c:Dtproperty}, if for any $u = v (u) \in \cD_0$ we define
\begin{align}\label{e:deftu}
\ft(u):=\sfT\wedge \sup\{t\geq 0: z_t(u)\in \cD_t\},
\end{align}

\noindent then $z_s(u)\in \cD_s$ for any $0\leq s\leq \ft(u)$. We also define the lattice on the domain $\cD_0$ given by
\begin{align}\label{def:L}
\mathcal L=\left\{E+\ri \eta\in \dom_0: E\in \bZ/ N^{6}, \eta\in \bZ/N^{6}\right\}.
\end{align}

\begin{lemma}\label{c:Lapproximate}
Adopt the assumptions in Proposition \ref{prop:LocalLaw}.
For any $t\in [0,\sfT]$ and $w\in \dom_t$, there exists some lattice point $u\in \mathcal L\cap z_t^{-1}(\dom_t)$, such that
$
|z_t(u)-w|\leq N^{-5},
$
provided $N$ is large enough. 
\end{lemma}
\begin{proof}
It follows from Lemma \ref{c:Dtproperty} that if $z_t(u)\in \dom_t$, then $u\in \cD_0$, and  $z_t(u)=u-tm_0(u)$ with $u\in \dom_0$. In particular 
$|\del_u z_t(u)|=|1-t\del_u m_0(u)|\leq 1+C\sfT $. Thus
$z_t(u)$ is Lipschitz in $u$ with Lipschitz constant bounded by $(1+C \sfT )$. Thus for any $w\in \dom_t$, there exists some lattice point $u\in \mathcal L\cap z_t^{-1}(\dom_t)$, such that
$
|z_t(u)-w|\leq (1+C \sfT)/N^6\leq N^{-5},
$
provided $N$ is large enough. 
\end{proof}

In the rest of this section, we prove Proposition~\ref{prop:LocalLaw} by studying the stochastic differential equation satisfied by $m_s(z)$.

\begin{proof}[Proof of Proposition~\ref{prop:LocalLaw}]
From Lemma \ref{lem:key} in the next subsection applied to the special case $v_k\equiv \tfrac{1}{N}$, we have
\begin{equation}\label{eq:dm}
\rd m_s(z) = m_s(z)\partial_z m_s\rd s
+\frac{1}{4N^2}\sum_{1\leq |i|\leq N}\frac{\rd\langle b_i\rangle_s-\rd s}{(x_i(s)-z)^3}
-\frac{1}{2\sqrt{2}N^{3/2}}\sum_{1\leq |i|\leq N}\frac{\rd b_k(s)}{(z-x_i(s))^2}.
\end{equation}
By plugging the characteristic flow (\ref{eqn:char}) in \eqref{eq:dm}, we get 
\begin{align}\label{eq:dm2}
\rd m_s(z_s)= -&\frac{1}{2\sqrt{2} N^{3/2}}\sum_{1\leq |i|\leq N} \frac{{\rm d} b_i(s)}{(x_i(s)-z_s)^2}+(m_s(z_s)-\widetilde m_s(z_s))\partial_z m_s(z_s)\rd s +\frac{1}{N^2}\sum_{1\leq |i|\leq N}\frac{\OO(1)\rd s}{(x_i(s)-z_s)^3}.
\end{align}
To analyze \eqref{eq:dm2}, we introduce a stopping time
\begin{align*}
\sigma=\sfT\wedge \inf_{s\geq 0}\left\{s: \exists z\in \cD_s,   |m_s(z)-\widetilde m_s(z)|\geq \varphi\sqrt{\frac{\Im[\widetilde m_s(z_s)]}{N\eta_s}}\right\}.
\end{align*}
Then for $s\leq \sigma$, thanks to \eqref{e:Immz1}, we have that \begin{align}\label{e:mstms}
\Im[m_s(z_s)]\asymp \Im[\widetilde m_s(z_s)].
\end{align}

\begin{proposition}\label{p:thirdbound}
There exists an event $\Omega$, measurable with respect to the paths $\{b_1(s), b_2(s), \cdots, b_N(s)\}_{0\leq s\leq \sfT }$, such that $\mathbb{P} [\Omega] \ge 1-Ce^{-(\log N)^2}$ and the following holds. On $\Omega$, for any $u\in \cL$ (recall this lattice from \eqref{def:L}), denote $z_s(u)=E_s(u)+\ri \eta_s(u)$. Then, for any $0\leq s\leq \ft(u)$ (recall \eqref{e:deftu}), we have
\begin{align}\label{StochasticBnd}
&\int_{0}^{s \wedge \sigma}\frac{1}{ N^{3/2}} \sum_{1\leq |i|\leq N} \frac{\rd b_i(\tau)}{|z_\tau(u) - x_i(\tau)|^2}  \leq  \varphi^{1/20}\sqrt{\frac{\Im[\widetilde m_{s\wedge \sigma}(z_{s\wedge \sigma})]}{N\eta_{s\wedge \sigma}}},\\
\label{e:ssbound1}
& \int_{0}^{s \wedge \sigma}  \frac{1}{N^2} \sum_{1\leq |i|\leq N}\frac{1}{|x_{i}(\tau) - z_{\tau}(u)|^3}\rd \tau\lesssim  \frac{1}{N \eta_{s \wedge \sigma}(u)}.
 \end{align}
\end{proposition}
\begin{proof}
For simplicity of notation, we write $\mathfrak{t} (u), E_{\tau}(u), \eta_{\tau}(u)$ as $\mathfrak{t}, z_{\tau}, E_{\tau}, \eta_{\tau}$, respectively.

To prove \eqref{StochasticBnd}, we notice by Lemma~\ref{c:Dtproperty} that $z_{\tau}(u)\in \cD_{\tau}$ for any $0\leq {\tau}\leq \ft$. We define a series of stopping times $0=t^{(0)}< t^{(1)}<t^{(2)}<\cdots<t^{(m)}=\ft$, 
as follows:
\begin{equation}\label{e:choosetik}
t^{(k)} = \ft \wedge \inf\{{\tau}> t^{(k-1)} : \eta_{{\tau}} <\eta_{t^{(k-1)}}/2\},\quad k=1,2,3,\cdots,m,
\end{equation}

\noindent where $m = m(u)$ might depend on $u$. From \eqref{e:Immz1},  $\eta_t$ cannot be smaller than $1/N^2$, and so any $u \in \mathcal{L}$ must satisfy $m = m(u) \leq10 \log N$.

Recall from Assumption \ref{a:Bbound} that $\langle \rd b_i, \rd b_j\rangle \leq 1$.   To bound the quadratic variation of the left side of \eqref{StochasticBnd}, for any  $s \leq t^{(k)}\wedge \sigma$ we have  
\begin{equation}\label{e:quadraticv}
\frac{1}{N^2}\int_0^{s}\sum_{1\leq |i|,|j|\leq N} \frac{\rd\tau}{|x_i(\tau)-z_\tau|^2|x_j(\tau)-z_\tau|^2}
=\int_0^{s}\frac{\Im[ m_\tau(z_\tau)]^2 }{\eta_\tau^2}\rd \tau
\leq\int_0^{s}\frac{\Im[\widetilde m_\tau(z_\tau)]^2 }{\eta_\tau^2}\rd \tau
\lesssim \frac{\Im[\widetilde m_{s}(z_{s})]}{\eta_{s}},
\end{equation}
where we successively  used \eqref{e:mstms} and $-\Im[\widetilde m_\tau(z_\tau)]=\del_\tau \eta_\tau$.
With (\ref{eqn:boundBr}) and (\ref{e:quadraticv}) have proved that for some $c_1>0$,  with  probability $1-c_1^{-1}e^{-c_1\varphi^{1/10}}$ we have 
\begin{equation}\label{localineq}
\sup_{0\leq {\tau} \leq t^{(k)} } \left|\int_{0}^{{\tau}\wedge \sigma}\frac{1}{ N^{3/2}} \sum_{1\leq |i|\leq N} \frac{\rd b_i({\tau})}{|z_{\tau} - x_i({\tau})|^2}  \right|\leq \varphi^{1/20}\sqrt{\frac{\Im[\widetilde m_{t^{(k)}\wedge \sigma}(z_{t^{(k)}\wedge \sigma})]}{N\eta_{t^{(k)}\wedge \sigma}}}.
\end{equation}
We define $\Omega$ to be the event on which \eqref{localineq} holds for all $0\leq k\leq m$ and all $u\in\cL$. Since $m|\cL|\leq |\mathcal{L}| \cdot 10 \log N \le N^{20}$, it follows from the discussion above that  $\Omega$ holds with probability $1-c_2^{-1}e^{-c_2\varphi^{1/10}}$ for some $c_2>0$.  Therefore, for any $s\in[t^{(k-1)},t^{(k)}]$, the bounds \eqref{localineq} and our choice of $t^{(k)}$ \eqref{e:choosetik} (with the fact that $\eta_s$  is non-increasing in $s$) yield on $\Omega$ that, for any $0\leq s\leq \ft(u)$, we have
\begin{align*}\begin{split}
\left|\int_{0}^{s\wedge \sigma}\frac{1}{N^{3/2}} \sum_{1\leq |i|\leq N} \frac{\rd b_i(\tau)}{|z_\tau - x_i(\tau)|^2} \right|
 \leq  \varphi^{1/20}\sqrt{\frac{\Im[\widetilde m_{t^{(k)}\wedge \sigma}(z_{t^{(k)}\wedge \sigma})]}{N\eta_{t^{(k)}\wedge \sigma}}}
\leq  \varphi^{1/20}\sqrt{\frac{\Im[\widetilde m_{s\wedge \sigma}(z_{s\wedge \sigma})]}{N\eta_{s\wedge \sigma}}}.
\end{split}
\end{align*}
This finishes the proof of  \eqref{StochasticBnd}.

The error terms in \eqref{e:ssbound1} can be bounded as
\begin{align*}
\int_0^{t\wedge \sigma}\frac{\OO(1)}{N^2}\sum_{1\leq |i|\leq N}\frac{\rd s}{|x_i(s)-z_s|^3}
\leq \int_0^{t\wedge \sigma}\frac{\OO(1)}{N}\frac{\Im[m_s(z_s)]\rd s}{\eta_s^2}
\leq \int_0^{t\wedge \sigma}\frac{\OO(1)}{N}\frac{\Im[\widetilde m_s(z_s)] \rd s}{\eta_s^2}
\lesssim \frac{1}{N\eta_{t\wedge \sigma}}.
\end{align*}
The first inequality  relies on $|x_i(s)-z_s|\geq \eta_s$,  and  last inequality uses $-\Im[\widetilde m_s(z_s)]=\del_s \eta_s$,  see (\ref{eqn:char}).
\end{proof}

From the previous proposition,  abbreviating $\Delta_s=m_s(z_s)-\widetilde m_s(z_s)$, for any $s\leq t\wedge\sigma$ we can rewrite Equation \eqref{eq:dm2} as 
\begin{equation}\label{e:spd}
\Delta_{s}
=\Delta_0+\int_0^{s}\Delta_v\partial_z m_v(z_v)\rd v
+ \OO\left( \varphi^{1/20}\sqrt{\frac{\Im[\widetilde m_{s}(z_{s})]}{N\eta_{s}}}\right).
\end{equation}

Thanks to Lemma \ref{l:STproperty} (by taking $\phi= \varphi \sqrt{\Im[\widetilde m_{v}(z_{v})/N\eta_{v}]}$), we have
\begin{equation}\label{e:aterm}
\left|\partial_z m_v(z_v(u))\right|
\leq \left|\del_z \widetilde m_v(z_v(u))\right|+\frac{1}{\eta_v} \sqrt{ \varphi\Im[\widetilde m_v(z_v)]\sqrt{\frac{\Im[\widetilde m_{v}(z_{v})]}{N\eta_{v}}}} 
=\del_z \widetilde m_v(z_v(u)) + \varphi^{\frac{1}{2}}\frac{\Im[\widetilde m_v(z_v)]^{3/4}}{\eta_v(N\eta_v)^{1/4}}.
\end{equation}
With \eqref{e:flow}, we bound the first term on the right-hand side above  with
$
|\del_z \widetilde m_v(z_v(u))|=\left|\frac{\del_u \widetilde m_v(z_v(u))}{\del_u z_v(u)}\right|
=\left|\frac{\del_u \widetilde m_0(u)}{1-s\del_u \widetilde m_0(u)}\right|\leq 2 C,
$
so that  \eqref{e:aterm} gives 
\begin{align}\label{eqn:partialBound}
\left|\partial_z m_v(z_v(u))\right|\leq \varphi^{\frac{1}{2}}\left(1+\frac{\Im[\widetilde m_v(z_v)]}{\eta_v(N\eta_v)^{1/4}}\right).
\end{align}
Denoting 
$
\beta_s\deq \varphi^{\frac{1}{2}}\left(1+\frac{\Im[\widetilde m_s(z_s)]}{\eta_s(N\eta_s)^{1/4}} \right),
$
\eqref{e:spd} and (\ref{eqn:partialBound}) imply, for any $s\leq t\wedge\sigma$
\begin{align*}
\left|\Delta_{s}\right|
\leq \int_0^{s}\beta_v\left|\Delta_{v}\right|\rd v+
 \OO\left(\varphi^{1/20}\sqrt{\frac{\Im[\widetilde m_{s}(z_{s})]}{N\eta_{t\wedge \sigma}}}\right).
\end{align*}
By Gr{\"o}nwall's inequality, this implies the estimate
\begin{equation}\label{e:midgronwall}
\left|\Delta_{t\wedge \sigma}\right|
\leq \OO\left(\varphi^{1/20}\sqrt{\frac{\Im[\widetilde m_{t\wedge \sigma}(z_{t\wedge \sigma})]}{N\eta_{t\wedge \sigma}}}\right)
+\varphi^{1/20} \int_0^{t\wedge\sigma}\beta_s\sqrt{\frac{\Im[\widetilde m_{s}(z_{s})]}{N\eta_{s}}}e^{\int_s^{t\wedge\sigma} \beta_\tau\rd \tau}\rd s.
\end{equation}
For the integral of $\beta_\tau$, we have
\begin{align*}
\int_s^{t\wedge\sigma} \beta_\tau\rd \tau
\leq  \varphi^{\frac{1}{2}}\left((t\wedge \sigma-s)+\frac{1}{(N\eta_{t\wedge \sigma})^{1/4}}\right)
\leq C.
\end{align*}
The last term in  \eqref{e:midgronwall} is therefore bounded with
\begin{align}\begin{split}\label{e:term2}
&\varphi^{\frac{1}{2}+\frac{1}{20}}\int_0^{t\wedge\sigma}\left( 1+\frac{\Im[\widetilde m_s(z_s)]}{\eta_s(N\eta_s)^{1/4}} \right)\sqrt{\frac{\Im[\widetilde m_{s}(z_{s})]}{N\eta_{s}}}\rd s
\lesssim \varphi^{\frac{1}{2}+\frac{1}{20}}\sqrt{\frac{\Im[\widetilde m_{t\wedge \sigma}(z_{t\wedge \sigma})]}{N\eta_{t\wedge \sigma}}}.
\end{split}\end{align}
It follows by combining \eqref{e:midgronwall} and \eqref{e:term2} that
\begin{align*}
\left|\td m_{t\wedge\sigma}(z_{t\wedge\sigma}(u))-m_{t\wedge\sigma}(z_{t\wedge\sigma}(u))\right|
\lesssim  \varphi^{\frac{1}{2}+\frac{1}{20}} \sqrt{\frac{\Im[\widetilde m_{t\wedge \sigma}(z_{t\wedge \sigma})]}{N\eta_{t\wedge \sigma}}}\ll \varphi\sqrt{\frac{\Im[\widetilde m_{t\wedge \sigma}(z_{t\wedge \sigma})]}{N\eta_{t\wedge \sigma}}}.
\end{align*}
We conclude that with very high probability $t\wedge \sigma=\sfT$, and thus Proposition \ref{prop:LocalLaw} follows. 
\end{proof}

\subsection{Hard edge universality.}\label{s:maximum_principle}\ 
Before proving Proposition \ref{p:universality}, we explain why we believe that the obtained result is optimal, up to subpolynomial (in $N$) errors,  in the setting of   dynamics driven by general correlated Brownian motions and particle locations given by a local law. 
We denote $\widetilde \rho_{t}$ (resp.  $\widetilde \rho'_{t}$) the density of
$\widetilde \mu_0\boxplus \mu_{\rm sc}^{(t)}$ (resp.  $\widetilde \mu_0'\boxplus \mu_{\rm sc}^{(t)}$)
(see the discussion in \Cref{s:locallaw}).
Without loss of generality we assume that 
$\widetilde \rho_{t}(0)=\widetilde \rho'_{t}(0)$, and 
apply the local law estimate (\ref{e:m0tm0diff}) to obtain the best possible bound on the initial condition
of Equation (\ref{e:tangential_eq}):
$
|v_i(0)|=|s_i(0)-s'_i(0)|\lesssim \frac{\sqrt{i}}{N}$ and $v_{-i}(0)=-v_i(0)$.
As the expected continuous limit of the propagator $\cU$ of  (\ref{e:tangential_eq}) is $p_{s,t}(x,y)=\frac{t-s}{(t-s)^2+(x-y)^2}$ (see \cite{BouErdYauYin2016}),  a perfect homogenization of  (\ref{e:tangential_eq}) would give
\begin{equation}\label{eqn:Opt}
|v_i(\alpha,t)|\lesssim \frac{1}{N} \int \frac{t}{t^2+(x-\frac{i}{N})^2}\cdot \sgn(x)(Nx)^{\frac{1}{2}}\rd x\lesssim \frac{i}{N}\cdot\frac{1}{(\max(i,Nt))^{\frac{1}{2}}}.
\end{equation}
Integration in $\alpha\in[0,1]$ then gives
$(\ref{eqn:relax})$ (the necessary additional term
$\max(|i|/N,t)$ is irrelevant for small $i$ and is
due to the 
 densities difference $\widetilde \rho_{t}(u)-\widetilde \rho'_{t}(u)$ for large $u$).

\begin{remark}[Improved estimates under stronger assumptions]\label{rem:stronger}
If we assume that the initial condition is rigid instead of the weak local law (\ref{e:m0tm0diff}), i.e.  if 
$
|v_i(0)|=|s_i(0)-s'_i(0)|\lesssim \frac{1}{N}$, then the above heuristic gives $|s_i(t)-s'_i(t)|\lesssim \frac{i}{N}\cdot\frac{1}{\max(i,Nt)}$. 

In fact, for rigid initial conditions and independent driving Brownian motions,  \cite[Theorem 6]{Wang2022} proves that this heuristic is correct by 
adapting the bulk universality proof from \cite{Bou2022}.
The assumptions of Proposition \ref{p:universality} are more general as they only assume a local law and general correlations for the noise.  The methods developed below likely applies to 
give a simpler proof that  $|s_i(t)-s'_i(t)|\lesssim \frac{i}{N}\cdot\frac{1}{\max(i,Nt)}$  under the rigidity and independence assumptions,  but we do not pursue this direction in this article. 
\end{remark}

Thanks to Assumption \ref{d:regular}, the two measures $\widetilde \mu_0, \widetilde \mu'_0$ have bounded densities $\widetilde \rho_0(x), \widetilde \rho'_0(x)$, and their densities have bounded derivatives for $|x|\leq G$. By our assumption $\wt \rho_0(0)=\wt \rho'_0(0)=A$. Since $t\ll1$, the free convolution density also satisfies $\widetilde \rho_{t}(0)= \widetilde \rho'_{t}(0)=A+\OO(t)$. Thus Proposition \ref{p:universality} follows if we can show that
\begin{align*}
|s_i(t)-s_i'(t)|\leq N^{\varepsilon}\cdot\frac{|i|}{N}\cdot\left(\frac{1}{\sqrt{Nt}}+\max\left(\frac{|i|}{N},t\right)\right), \quad 1\leq |i|\leq N.
\end{align*}

We now start the proof of Proposition \ref{p:universality}.
Assumption \ref{d:regular} implies that $\widetilde \rho_0(x),\widetilde \rho'_0(x)$ have bounded derivatives for $|x|\leq R$, and it follows
\begin{align}\label{e:density}
\widetilde \rho_0(x)=A+\OO(|x|), \quad \widetilde \rho'_0(x)=A+\OO(|x|). 
\end{align}
If we denote the classical locations of $\widetilde \rho_0(x)$ as $\gamma_i$
\begin{align*}
\frac{i}{2N}=\int_0^{\gamma_i}\widetilde \rho_0(x)\rd x, \quad 1\leq |i|\leq N.
\end{align*}
Then for $|i|\leq c_R N$, \eqref{e:density} implies that 
\begin{align*}
\gamma_i=\frac{i}{2AN}+\OO\left(\frac{i^2}{N^2}\right). 
\end{align*}
By a standard argument  (see \cite[Corollary 4.2]{Lazlo2009}),  $(\nu, R)$-regularity in the sense of Assumption \ref{d:regular} implies that 
\begin{align}\label{eqn:numberEigenvalues}
\mu_0(I)=\widetilde \mu_0(I)+\OO\left((N|I|)^{1/2}\right)
\end{align}
for any interval  $I\subset[-c_R,c_R]$ with  $|I|\geq N^{-1+\nu}$. Since $\mu_0$ and $\widetilde \mu_0$ are both symmetric around origin, we conclude that for any $|i|\leq cN$
\begin{align*}
|s_i(0)-\gamma_i(0)|\lesssim \frac{|i|^{1/2}+N^{\nu}}{N} +\frac{i^2}{N^2}. 
\end{align*}
By the same argument  we also have that 
\begin{align}\label{e:initial_diff}
|s'_i(0)-\gamma_i(0)|\lesssim \frac{|i|^{1/2}+N^{\nu}}{N} +\frac{i^2}{N^2}, \qquad\quad |s_i(0)-s'_i(0)|\lesssim \frac{|i|^{1/2}+N^{\nu}}{N} +\frac{i^2}{N^2}.
\end{align}

In the rest of this section, we analyze the coupling from \eqref{e:coupleB}:
\begin{align}\label{generalDBMII}
\rd x_i(t,\alpha)=\frac{\rd b_i}{\sqrt{2 N}}+\frac{1}{2N}\sum_{j\neq i}\frac{1}{x_i(t,\alpha)-x_j(t,\alpha)}\rd t, \quad 1\leq |i|\leq N,
\end{align}
where 
$
x_i(0,\al)=\alpha s_i+(1-\alpha)s'_i,\quad 0\leq \al\leq 1,
$
and the martingales satisfies $\rd \langle b_i\rangle/\rd t \leq 1$.
We consider the corresponding tangential equation (\ref{e:tangential_eq}) satisfied by $v_i(t)=v_i(\alpha,t)=\frac{\rd}{\rd\alpha}x_i(\alpha,t)$

We first state qualitative properties of  (\ref{e:tangential_eq}). The elementary proof is left to the reader.

\begin{lemma} \label{lem:qualitative} Consider (\ref{e:tangential_eq}) with initial condition $\bv$.
\begin{enumerate}[(i)]
\item If $\bv$ is antisymmetric($v_i=-v_{-i}$)  and  $v_{i}\geq 0$ for all $i\geq 1$, then $v_i(t)\geq 0$ for all $i\geq 1$, and $t\geq 0$.
\item If $v_{-N}\leq \dots\leq v_{-1}\leq v_1\leq \dots\leq v_N$, then for any $t\geq 0$ we have  
$v_{-N}(t)\leq \dots\leq v_{-1}(t)\leq v_1(t)\leq \dots\leq v_N(t)$.
\end{enumerate}
\end{lemma}

We now define the key observable from \cite{Bou2022},
\begin{equation}\label{eqn:ft}
f_t(z)=\sum_{1\leq |i|\leq N}\frac{v_i(t)}{x_i(t)-z},
\end{equation}
which satisfies the following stochastic advection equation.

\begin{lemma}\label{lem:key}
For any $\im z\neq 0$, we have
\[
\rd f_t=m_t(z)\partial_z f_t\rd t
+\frac{1}{2N}\sum_{1\leq |k|\leq N}\frac{v_k(t)(\rd\langle b_k\rangle_t-\rd t)}{(x_k(t)-z)^3}
-\frac{1}{\sqrt{2N}}\sum_{1\leq |k|\leq N}\frac{v_k(t)}{(z-x_k(t))^2}\rd b_k(t).
\]
\end{lemma}

\begin{proof} It is a simple application of It{\^ o}'s formula. We omit the time index. First,
\begin{equation}\label{eqn:Ito1}
\rd f = \sum_{1\leq |k|\leq N}\frac{\rd v_k}{x_k-z}+\sum_{1\leq |k|\leq N}v_k\rd\frac{1}{x_k-z}.
\end{equation}
Applying again the It{\^ o} formula $\rd (x_k-z)^{-1}=-(x_k-z)^{-2}\rd x_k+\frac{\rd\langle {b_k}\rangle}{2N}(x_k-z)^{-3}\rd t$,  with (\ref{generalDBMII}) we naturally decompose the second sum above as (I)+[(II)+(III)]$\rd t$ where
\begin{align*}
{\rm (I)}&=-\frac{1}{\sqrt{2N}}\sum_{1\leq |k|\leq N}\frac{v_k}{(z-x_k)^2}\rd b_k,\\
{\rm (II)}&=\frac{1}{2N}\sum_{\ell\neq k}\frac{v_k}{x_\ell-x_k}\frac{1}{(x_k-z)^2},
\\
{\rm (III)}&=\frac{1}{2N}\sum_{1\leq |k|\leq N} \frac{v_k\rd\langle b_k\rangle}{(x_k-z)^3}.
\end{align*}
Concerning the first sum in (\ref{eqn:Ito1}), by (\ref{e:tangential_eq}) we have
\begin{align*}
\sum_{1\leq |k|\leq N}\frac{\partial_t v_k}{x_k-z}
&=\sum_{\ell\neq k}\frac{v_\ell-v_k}{2N(x_\ell-x_k)^2(x_k-z)}
=\frac{1}{2}\sum_{\ell\neq k}\frac{v_\ell-v_k}{2N(x_\ell-x_k)^2}\left(\frac{1}{x_k-z}-\frac{1}{x_\ell-z}\right)
\\
&=\frac{1}{4N}\sum_{\ell\neq k}\frac{v_\ell-v_k}{x_\ell-x_k}\frac{1}{(x_k-z)(x_\ell-z)}
=-\frac{1}{2N}\sum_{\ell\neq k}\frac{v_k}{x_\ell-x_k}\frac{1}{(x_k-z)(x_\ell-z)}.
\end{align*}
Combining with (II), we obtain 
$$
{\rm (II)}+\sum_{1\leq |k|\leq N}\frac{\partial_t v_k}{x_k-z}=
\frac{1}{2N}\sum_{\ell\neq k}\frac{v_k}{x_\ell-x_k}\frac{1}{x_k-z}\left(\frac{1}{x_k-z}-\frac{1}{x_\ell-z}\right)
=
\frac{1}{2N}\sum_{\ell\neq k}\frac{v_k}{(x_k-z)^2}\frac{1}{x_\ell-z}.
$$
All singularities have disappeared. We obtained
$
{\rm (II)}+{\rm (III)}+\sum_{1\leq |k|\leq N}\frac{\partial_t v_k}{x_k-z}
=m(z)\partial_z f+\frac{1}{2N}\sum_{1\leq |k|\leq N}\frac{v_k(t)(\rd\langle b_k\rangle_t-\rd t)}{(x_k(t)-z)^3}.
$
Summation of the remaining term (I) concludes the proof.
\end{proof}

Instead of considering the stochastic advection equation directly for (\ref{eqn:ft})
we will exploit positivity of the $v_i$'s for $i\geq 1$ by defining
\[
f_s(z)=f^{(1)}_s(z)+f_s^{(2)}(z)
\]
where 
\[
f_s^{(1)}(z)=\sum_{1\leq |i|\leq N}\frac{v^{(1)}_i(s)}{x_i(s)-z},\qquad\quad v^{(1)}_i(0)=\sgn(i)\cdot \frac{i^{1/2}}{N},
\]
and
\[
f_s^{(2)}(z)=\sum_{1\leq |i|\leq N}\frac{v^{(2)}_i(s)}{x_i(s)-z},\qquad\quad  v^{(2)}_i(0)=\sgn(i)\cdot \frac{i^2}{N^2}.
\]

\begin{lemma}\label{lem:init}

For any $|z|\leq R/4$ with $E\geq 0$,  $\eta:=\im z\geq N^{-1+\nu}$,  we have
$$ 
|\im f^{(1)}_0(z)|+|\im f^{(2)}_0(z)|\leq  E\sqrt{\frac{N}{\max(E,\eta)}}+ N\log N\ E\max(\eta,E).
$$
\end{lemma}

\begin{proof}
We start with the contribution from $f_0^{(2)}$.
By symmetry of the $x_k$ and $v_k$'s,  so we can write
\begin{equation}
\label{eq:estf02}
\begin{split}
|\im f^{(2)}_0(z)|&=\left|\im\sum_{1\leq k\leq N}\left(\frac{v^{(2)}_k(0)}{x_k(0)-z}-\frac{v^{(2)}_k(0)}{-x_k(0)-z}\right)\right|
\leq \eta\sum_{1\leq k\leq N} |v^{(2)}_k(0)|\left|\frac{1}{|x_k-E|^2+\eta^2}-\frac{1}{|x_k(0)+E|^2+\eta^2}\right|\\
&=4\eta E\sum_{1\leq k\leq N} |v^{(2)}_k(0)|\frac{x_k(0)}{|x_k(0)-z|^2|x_k(0)+z|^2}
\leq4 \eta E\sum_{1\leq k\leq N}\frac{k^2}{N^2}\frac{x_k(0)}{|x_k(0)-z|^2|x_k(0)+z|^2}.
\end{split}
\end{equation}
From (\ref{eqn:numberEigenvalues}) and the assumption $\Im z\geq N^{-1+\nu}$ for eigenvalues $x_k\leq R/2$,  and the assumption $x_k\leq C$ for the complementary regime,
we have 
\[
\mbox{right-hand side of}\ \eqref{eq:estf02}\lesssim 
 \eta E N\int_0^c\frac{x^3}{|z-x|^2|z+x|^2}\rd x+\eta EN.
\]
If $\eta>E$ we have $|z-x|\asymp|z+x|$, and we obtain
\[
\int_0^c\frac{x^3}{|z-x|^2|z+x|^2}\rd x\lesssim \int_0^c\frac{x^3}{|z-x|^4}\rd x\lesssim \log N.
\]
If $\eta<E$, we have $|z+x|\asymp \max(x,E)$ so that
\begin{multline*}
\int_0^c\frac{x^3}{|z-x|^2|z+x|^2}\rd x\lesssim \int_0^c\frac{x^3}{|z-x|^2\max(x,E)^2}\rd x\lesssim 
E^{-2}\int_0^E\frac{x^3}{|z-x|^2}\rd x+\int_E^c\frac{x}{|z-x|^2}\rd x\\
\lesssim \frac{E}{\eta}+
\int_0^{c-E}\frac{(E+x)\rd x}{\eta^2+x^2}\lesssim \frac{E}{\eta}+ \log\frac{c-E}{\eta}\lesssim \frac{E}{\eta}+\log N.
\end{multline*}
We have therefore obtained
$
|\im f^{(2)}_0(z)|\lesssim N\log N\, E\max(\eta,E).
$

One can treat $ f_0^{(1)}$ similarly, obtaining
\[
|\Im  f_0^{(1)}(z)|\leq
4\eta E\sum_{1\leq k\leq N}\frac{\sqrt{k}}{N}\frac{x_k(0)}{|x_k(0)-z|^2|x_k(0)+z|^2}
\lesssim
\sqrt{N}E\int_0^c\frac{x^{3/2}\eta}{|z-x|^2|z+x|^2}\rd x\lesssim  E\sqrt{\frac{N}{\max(E,\eta)}},
\]
concluding the proof.
\end{proof}

We define $y=\frac{\varphi^5}{N}$ and
\begin{equation*}
\mathscr{S}=
\left\{
z= E+\ii y:
y<E<\frac{R}{10}
\right\}.
\end{equation*}
%

In the following lemma and its proof,  for a fixed $t>0$ we use the convention
\begin{equation}
\label{eqn:NewCharacteristics}
\del_s z_s=\widetilde m_{t-s}(z_s),\ 0\leq s\leq t
\end{equation}
for the characteristics (note the sign change compared to (\ref{e:flow}): the characteristics now move upwards),  for coherence with the  notation from \cite{Bou2022}.

\begin{lemma}\label{lem:AdvEq} Let $\sfT=N^{-2\nu}$.
With  very high probability,  for any $z=E+\ri y\in\mathscr{S}$ and $0<t<\sfT$ we have
$$
|\im f^{(1)}_0(z_{t})-\im f^{(1)}_t(z)|\leq  E\sqrt{\frac{N}{\max(E,t)}}+NE\max(E,t),
$$
and the same estimate holds for $\im f^{(2)}_0(z_{t})-\im f^{(2)}_t(z)$.
\end{lemma}

\begin{proof}
In the following,  we abbreviate $f_s$ for either $f_s^{(1)}$ or $f_s^{(2)}$, as the proof is the same in both cases (and accordingly we denote $v_k=v_k^{(1)}$ or $v_k^{(2)}$).

For any $1\leq \ell,m\leq N^{12}\sfT$, we define $t_\ell=\ell N^{-12}$ and $z^{(m)}=E_m+\ii y$ where $E_m=m N^{-12}$.
We also define the stopping times (with respect to $\mathcal{F}_t=\sigma({b_k}(s),0\leq s\leq t,1\leq k\leq N)$)
\begin{align*}
\tau_{\ell,m}&=
\inf\left\{0\leq s\leq t_\ell: 
|\im  f_s(z^{(m)}_{t_\ell-s})-\im  f_0(z^{(m)}_{t_\ell})|
>
\frac{1}{2}\left(E_m\sqrt{\frac{N}{\max(E_m,t_\ell)}}+NE_m\max(E_m,t_\ell)\right)
\right\},\\
\tau_0&=
\inf\left\{0\leq t\leq 1:\mid\exists z\in[0, R/2]\times[\frac{\varphi^5}{N},1]\ {\rm s.t.}\ \left|m_t(w)-\widetilde m_t(w)\right| \geq\varphi \sqrt{\frac{\Im[\widetilde m_0(w)]}{N\Im[w]}}\right\},\\
\tau&=\min\{\tau_0,\tau_{\ell,m}:0\leq \ell,m\leq N^{12}\sfT, \varphi^5 N^{-1}<E_m< R/10\},
\end{align*}
with the convention $\inf\varnothing=\sfT$. 
We will prove that for any $D>0$ there exists $\wt N_0(D)$ such that for any $N\geq \wt N_0(D)$, we have 
\begin{equation}\label{eqn:inter1}
\mathbb{P}(\tau=\sfT)>1-N^{-D}.
\end{equation}
We first explain why the above equation implies the expected result by a grid argument in $t$ and $z$.

On the one hand, we have the sets inclusion
\begin{equation}\label{eqn:inter2}
\{\tau=\sfT\}\bigcap_{1\leq \ell\leq N^{12}\sfT,1\leq m\leq N^{12},1\leq k\leq N}A_{\ell,m,k}
\subset
\bigcap_{z\in \mathscr{S},0<t<\sfT}\left\{
|\im  f_s(z_{t-s})-\im  f_0(z_t)|
<
 E\sqrt{\frac{N}{\max\{E,t\}}}+NE\max(E,t)
\right\},
\end{equation}
where
$$
A_{\ell,m,k}=\left\{
\sup_{t_\ell\leq u\leq t_{\ell+1}}\left|\int_{t_\ell}^u
\frac{{ v}_k(s)\rd b_k(s)}{(z^{(m)}-x_k(s))^2}
\right|<N^{-4}
\right\}.
$$
Indeed, for any given $z$ and $t$, chose $t_\ell,z^{(m)}$ such that $t_\ell\leq t<t_{\ell+1}$ and $|z-z_m|<N^{-5}$. Then
$
| f_t(z)- f_t(z^{(m)})|<N^{-2},
$
say, as follows directly from the definition of $ f_t$ and the crude estimate $|{ v}_k(t)|<1$ (obtained by maximum principle). 
Moreover, we can bound the time increments using Lemma \ref{lem:key}:
\begin{equation}\label{eqn:dynamicsPDE}
\rd f_s=m_s(z)\partial_z f_s\rd s
+\frac{1}{2N}\sum_{1\leq |k|\leq N}\frac{v_k(s)(\rd\langle b_k\rangle_s-\rd s)}{(x_k(s)-z)^3}
-\frac{1}{\sqrt{2N}}\sum_{1\leq |k|\leq N}\frac{v_k(s)}{(z-x_k(s))^2}\rd b_k(s).
\end{equation}
Thanks to 
the trivial estimates 
$|m_t(E+\ii\eta)|\leq \eta^{-1}$, 
$|\partial_z f_t(E+\ii\eta)|
\leq N\|{v}(0)\|_\infty\eta^{-2}\leq N\eta^{-2}$ and  $|\partial_{zz} f_t(E+\ii\eta)|\leq N\eta^{-3}$, under the event $\cap_kA_{\ell,m,k}$ (to bound the martingale term)
we have
$
| f_t(z)- f_{t_\ell}(z^{(m)})|<N^{-2}.
$

For $M_u=\int_{t_\ell}^u
\frac{{ v}_k(s)\rd b_k(s)}{(z^{(m)}-x_k(s))^2}$, we have the deterministic estimate
$\langle M\rangle_{t_{\ell+1}}\leq N^{-12}(\varphi^2/N)^{-4}\|{v}(0)\|_\infty^2\leq \varphi^{-8}N^{-8}$, 
so that (\ref{eqn:boundBr}) with $\mu=\varphi^{-8}N^{-8}$ gives
$
\mathbb{P}(A_{\ell,m,k})\geq 1-e^{-c\varphi^{1/5}}
$
and therefore, for any $D>0$, for large enough $N$ we have
\begin{equation}\label{eqn:inter3}
\mathbb{P}\Big(
\bigcap_{1\leq \ell\leq N^{12}\sfT,1\leq m\leq N^{12},1\leq k\leq N}A_{\ell,m,k}
\Big)
\geq 
1-N^{-D}.
\end{equation}
Equations (\ref{eqn:inter1}), (\ref{eqn:inter2}), (\ref{eqn:inter3}) conclude the proof of the proposition.

We now prove (\ref{eqn:inter1}). We abbreviate $t=t_\ell$, $z=z^{(m)}$ for some $1\leq \ell,m\leq N^{12}\sfT$. Let $g_u(z)= f_{u}(z_{t-u})$. 
Composing (\ref{eqn:dynamicsPDE}) with the characteristics (\ref{eqn:NewCharacteristics})  we obtain \footnote{In this paper, we abbreviate $u\wedge t=\min(u,t)$ when $u$ and $t$ are time variables.}
\begin{align}\begin{split}\label{eqn:gev}
\rd g_{u\wedge \tau}(z)=(m_u(z_{t-u})-\widetilde m_u(z_{t-u}))\partial_z  f_u(z_{t-u})\rd({u\wedge \tau})
+\frac{1}{2N}\sum_{1\leq |i|\leq N} \frac{(\rd\langle b_i\rangle_u/\rd u-1)v_i(u)}{(x_i(u)-z_{t-u})^3}\rd(u\wedge\tau)
\\-\frac{\OO(1)}{\sqrt{N}}\sum_{1\leq |i|\leq N}\frac{{ v}_i({u})}{(z_{t-u}-x_i(u))^2}\rd b_i({u\wedge \tau}).
\end{split}\end{align}
We  first bound
\begin{equation*}
\int_0^{t}\left| (m_u(z_{t-u})-\widetilde m_u(z_{t-u}))\partial_z  f_{u}(z_{t-u})\right|\rd (u\wedge \tau)
\leq
\int_0^{t}\frac{\varphi}{\sqrt{N\eta_{t-u}}}\left|\partial_z  f_{u}(z_{t-u})\right|\rd (u\wedge \tau).
\end{equation*}
To bound $m_u-\wt m_u$ above, we have used $0\leq u\leq \tau$.   Moreover,
\[
 f_{s}(z)=\sum_{i=1}^N\frac{x_i(t) v_i(t)}{x_i(t)^2-z^2}
\]
so that (from Lemma \ref{lem:qualitative} we know that $v_i(u)\geq 0$ for all $u\geq 0, i\geq 1$)
\begin{equation}\label{eqn:firstTerm}
\left|\partial_z  f_{u}(z_{t-u})\right|\leq |z_{t-u}|\sum_{i=1}^N\frac{x_i(u) v_i(u)}{|x_i(u)-z_{t-u}|^2\cdot |x_i(u)+z_{t-u}|^2}.
\end{equation}
Let $I_j=I_j(u)=\{i\in\llbracket 1,N\rrbracket:x_i(u)\in\frac{\varphi^5}{N}[j,j+1]\}$, $0\leq j\leq c N/\varphi^5$.
Then
\begin{equation}\label{eqn:EstDer}
\left|\partial_z  f_{u}(z_{t-u})\right|\leq |z_{t-u}|
\sum_{0\leq j\leq cN/\varphi^5}\left(\max_{i\in I_j}\frac{x_i(u)}{|x_i(u)-z_{t-u}|^2\cdot |x_i(u)+z_{t-u}|^2}\right)\cdot \sum_{i\in I_j} v_i(u)+N|z_{t-u}|\int_{G/10}^C{|x-z_{t-u}|^{-2}}\rd x,
\end{equation}
where the above integral is due to points $x_i(u)\geq R/10$, and use the local law in the domain $R/10<x_i<R/2$, the upper bound $x_i<C$, and the maximum principle bound $|v_i|\leq 1$. The above integral contributes to
at most
\[
N|z_{t-u}|\int_{G/10}^C{|x-z_{t-u}|^{-2}}\rd x\lesssim \frac{N |z_{t-u}|}{|z_{t-u}-\frac{G}{10}|}.
\]
The above error term is $\lesssim N|z_{t-u}|\frac{\max(E,\eta_{t-u})}{\eta_{t-u}}$ (here $E=\Im z$),  which 
we will now show to be the contribution from the sum in (\ref{eqn:EstDer}): 
the integral term in (\ref{eqn:EstDer}) can be ignored from now.

For each $1\leq j\leq cN/\varphi^5$ (note that we omit the case $j=0$ which will be treated separately), pick a $n=n_j$ such that 
$|\Re z^{(n)}-\varphi^5\frac{j}{N}|<N^{-9}$ and denote $e_{j}=e_{j,N}={\rm Re}z^{(n)}$.  
First, as ${ v}_k(u)\geq 0$ for any $k\geq 1$ and $u$,  we have  
\begin{equation}\label{eqn:deterministic}
\im  f_u(z^{(n)})=ye_j \sum_{i=1}^N\frac{x_i(u) v_i(u)}{|x_i(u)-z^{(n)}|^2|x_i(u)+z^{(n)}|^2}
\gtrsim
y \sum_{i\in I_j}\frac{ v_i(u)}{|x_i(u)-z^{(n)}|^2}
\gtrsim
y^{-1} \sum_{i\in I_j} v_i(u),
\end{equation}
where in the first inequality we used that for $i\in I_j$ ($j\geq 1$) we have $x_i\asymp e_j\asymp|x_i+z^{(n)}|$. 
To estimate $\im  f_u(z^{(n)})$, introduce $\ell$ such that $t_\ell\leq u<t_{\ell+1}$.
On the event $\cap_kA_{\ell,m,k}$ and $u\leq \tau$, we have
$
| f_{u}(z^{(n)})- f_{t_\ell}(z^{(n)})|<N^{-2}
$
as seen easily  from (\ref{eqn:dynamicsPDE}).
We therefore proved 
\begin{equation}\label{eqn:localAverage}
\sum_{k\in I_j}{ v}_k(u)\leq y\im  f_{ t_\ell}(z^{(n)})+N^{-2}\leq  \frac{\varphi^{5}}{N} \left(e_j\sqrt{\frac{N}{\max(e_j,u)}}+Ne_j\max(e_j,u)\right).
\end{equation}
We used $t_\ell\leq u\leq \tau$ for the second inequality.  For $j=0$ we simply bound $\sum_{k\in I_0}{ v}_k(u)\leq \sum_{k\in I_1}{ v}_k(u)$ because the evolution preserves  monotonicity, as proved in Lemma \ref{lem:qualitative} (ii).

We  have therefore obtained (using the local law on scale $\varphi^5/N$,  i.e.  $\tau\leq\tau_0$)
\begin{equation*}
\left|\partial_z  f_{u}(z_{t-u})\right|\leq\frac{\varphi^{5}}{N} |z_{t-u}|
\sum_{0\leq j\leq N/\varphi^5}\frac{e_j}{|e_j-z_{t-u}|^2\cdot |e_j+z_{t-u}|^2}\cdot  \left(e_j\sqrt{\frac{N}{\max(e_j,u)}}+Ne_j\max(e_j,u)\right).
\end{equation*}
The contribution from the first term ($e_j\sqrt{\frac{N}{\max(e_j,u)}}$) is
\[
N^{1/2}|z_{t-u}|\int_0^1\frac{x^2}{|x-z_{t-u}|^2\cdot|x+z_{t-u}|^2}\frac{\rd x}{\sqrt{\max(x,u)}}.
\]
If $t-u\leq E=\Re z$, the above integral is 
\[
\lesssim\frac{1}{\eta_{t-u}} \int_0^1\frac{\eta_{t-u}}{|x-z_{t-u}|^2}\cdot\frac{x^2}{x^2+E^2}\frac{\rd x}{\sqrt{\max(x,u)}}\lesssim \frac{1}{\eta_{t-u}\sqrt{\max(E,u)}}.
\]
If $t-u\geq E$, it is
\[
\lesssim  \int_0^{t-u}\frac{x^2}{|t-u|^4}\frac{\rd x}{\sqrt{\max(x,u)}}+ \int_{t-u}^{1}\frac{x^2}{x^4}\frac{\rd x}{\sqrt{\max(x,u)}}\lesssim \frac{1}{\eta_{t-u}\sqrt{\max(u,\eta_{t-u})}}.
\]
Similarly, the  contribution from the second term ($Ne_j\max(e_j,u)$) is 
\[
N|z_{t-u}|\int_0^1\frac{x^2}{|x-z_{t-u}|^2\cdot|x+z_{t-u}|^2}\max(x,u)\rd x
\lesssim N |z_{t-u}|\frac{\max(E,u,\eta_{t-u})}{\eta_{t-u}}.
\]
All together, we have proved
\[
\left|\partial_z  f_{u}(z_{t-u})\right|
\lesssim  N^{1/2}\frac{|z_{t-u}|}{\eta_{t-u}\sqrt{\max(E,u,\eta_{t-u})}}+ N |z_{t-u}|\frac{\max(E,u,\eta_{t-u})}{\eta_{t-u}},
\]
so that
\begin{multline*}
\int_0^{t}\frac{\varphi}{\sqrt{N\eta_{t-u}}}\left|\partial_z  f_{u}(z_{t-u})\right|\rd (u\wedge \tau)\lesssim 
\varphi\int_0^t\frac{\max(E,\eta_{t-u})}{\eta_{t-u}^{3/2}\sqrt{\max(E,u,\eta_{t-u})}}\rd u
+\varphi\sqrt{N}\int_0^t\frac{\max(E,\eta_{t-u})\max(E,u,\eta_{t-u})}{\eta_{t-u}^{3/2}}\rd u\\
\lesssim\frac{\varphi}{\sqrt{y}}\frac{E}{\sqrt{\max(E,t)}}
+\frac{\varphi^2\sqrt{N}}{\sqrt{y}}E\max(E,t)
\lesssim\varphi^{-1/2}E\sqrt{\frac{N}{\max(E,t)}}+\varphi^{-1/2}NE\max(E,t).
\end{multline*}

With the same reasoning,  we now bound the contribution in  (\ref{eqn:gev}) from
\begin{equation}\label{eqn:secondTerm}
\frac{1}{N}\sum_{1\leq |i|\leq N} \frac{(\rd\langle b_i\rangle_u/\rd u-1)v_i(u)}{(x_i(u)-z_{t-u})^3}=
\partial_{zz}\frac{1}{2N}\sum_{1\leq |i|\leq N} \frac{(\rd\langle b_i\rangle_u/\rd u-1)v_i(u)}{x_i(u)-z_{t-u}}
=
\frac{1}{N}\sum_{1\leq |i|\leq N}
\partial_{zz} \frac{\OO(1)x_iv_i(u)}{x_i(u)^2-z_{t-u}^2}.
\end{equation}
Note that
\[
|\frac{1}{N}
\partial_{zz} \frac{xv}{x^2-z^2}|\lesssim \frac{xv}{N}(\frac{1}{|x-z|^2|x+z|^2}+\frac{|z|^2}{|x-z|^3|x+z|^3})
\lesssim \frac{xv|z|}{|x-z|^2|x+z|^2}
\]
where we used $1/N\leq |z|$,  $|z|\leq |x+z|$ and $1/N\leq  |x-z|$.
This proves that the contribution from (\ref{eqn:secondTerm}) in (\ref{eqn:gev}) can be bounded by the right-hand side in 
(\ref{eqn:firstTerm}) and therefore has a smaller contribution.

We finally want to bound $\sup_{0\leq s\leq t}|M_s|$, let $E=\Re[z_{t-u}]$, then
\begin{align*}
M_s
&:=\im\int_0^{s}\frac{1}{\sqrt{N}}\sum_{1\leq |k|\leq N}\frac{{ v}_k({u})}{(z_{t-u}-x_k(u))^2}\rd b_k(u\wedge\tau)\\
&=
\int_0^{s}\frac{1}{\sqrt{N}}\sum_{k=1}^N\im\left(\frac{1}{(z_{t-u}-x_k(u))^2}+\frac{1}{(z_{t-u}+x_k(u))^2}\right){ v}_k({u})\rd b_k(u\wedge\tau)\\
&=
\int_0^{s}\eta_{t-u}\frac{1}{\sqrt{N}}\sum_{k=1}^N\left(\frac{E-x_k(u)}{|z_{t-u}-x_k(u)|^4}+\frac{E+x_k(u)}{|z_{t-u}+x_k(u)|^4}\right){ v}_k({u})\rd b_k(u\wedge\tau).
\end{align*}
With (\ref{eqn:boundBr}) with  very high probability  we have
\begin{equation}\label{eqn:error2}
\sup_{0\leq s\leq t}|M_{s}|^2\leq \varphi^{1/10}
\int_0^{t}\frac{\eta_{t-u}^2}{N}\left(\sum_{k=1}^N
\left|\frac{E-x_k(u)}{|z_{t-u}-x_k(u)|^4}+\frac{E+x_k(u)}{|z_{t-u}+x_k(u)|^4}\right|{ v}_k({u})
\right)^2\rd (u\wedge\tau).
\end{equation}
With (\ref{eqn:localAverage}) and the local law ($\tau\leq \tau_0$) the above sum is bounded with (we omit the terms corresponding to $x_i> R/10$ as they are shown to be negligible thanks to the local law and the simple bound $|v_i|\leq 1$)
\begin{align*}
 &\phantom{{}={}}N^{1/2}\int_0^c\left|\frac{E-x}{|z_{t-u}-x|^4}+\frac{E+x}{|z_{t-u}+x|^4}\right|\frac{x}{\sqrt{\max(x,u)}}\rd x
 +\varphi N\int_0^c\left|\frac{E-x}{|z_{t-u}-x|^4}+\frac{E+x}{|z_{t-u}+x|^4}\right|x\max(x,u)\rd x
 \\
&\lesssim
 N^{1/2}\left(\frac{E}{\eta_{t-u}^2\sqrt{\max(E,u)}}\mathds{1}_{t-u\leq E}+\frac{1}{\eta_{t-u}\sqrt{\max(\eta_{t-u},u)}}\mathds{1}_{t-u\geq E}\right)\\
 &+N\left(\frac{E}{\eta_{t-u}^2}\max(E,u)\mathds{1}_{t-u\leq E}+\frac{\max(\eta_{t-u},u)}{\eta_{t-u}}\mathds{1}_{t-u\geq E}\right).
\end{align*}
Hence the bracket in (\ref{eqn:error2}) is upper bounded with the sum of these two terms:
\begin{align*}
&\varphi^{1/10}\int_0^t\left(\frac{E}{\eta_{t-u}\sqrt{\max(E,u)}}\mathds{1}_{t-u\leq E}+\frac{1}{\sqrt{\max(\eta_{t-u},u)}}\mathds{1}_{t-u\geq E}\right)^2\rd u\lesssim \frac{\varphi^{1/10}}{y}\frac{E^2}{\max(E,t)},\\
&\varphi^{1/10}N\int_0^t\left(\frac{E\max(E,u)}{\eta_{t-u}}\mathds{1}_{t-u\leq E}+\max(\eta_{t-u},u)\mathds{1}_{t-u\geq E}\right)^2\rd u\lesssim \frac{\varphi^{1/10}}{y}N(E\max(E,t))^2.
\end{align*}
This concludes the proof.
\end{proof}

\begin{corollary} \label{thm:Holder}
For any $t\in[\varphi^{100}/N,1]$ and $1\leq i<N$, uniformly in $\al$ we have, with overwhelming probability,
$$
|v_i(t, \al)|\leq \varphi^8\frac{i}{N}\left(\frac{1}{\sqrt{Nt}}+\max(\frac{i}{N},t)\right).
$$
\end{corollary}

\begin{proof}
We have $|v_i(0)|\leq \varphi (v_i^{(1)}(0)+v_i^{(2)}(0))$ from (\ref{e:initial_diff});
from Lemma \ref{lem:qualitative} (i),  this implies that for any $t\geq 0, 1\leq i\leq N$ we have $|v_i(t)|\leq \varphi(v_i^{(1)}(t)+v_i^{(2)}(t))$.
From Equation (\ref{eqn:deterministic}) and Lemmas \ref{lem:init} and  \ref{lem:AdvEq}, we have
\[
 v_i^{(1)}(t)+v_i^{(2)}(t)\lesssim\frac{\varphi^6}{N}\left(
  \gamma_i\sqrt{\frac{N}{\max(\gamma_i,t)}}+N\gamma_i\max(\gamma_i,t)\right)
 \lesssim
 \varphi^6\frac{i}{N}\left(\frac{1}{\sqrt{Nt}}+\max(\frac{i}{N},t)\right)
\]
with  very high probability. This concludes the proof.
\end{proof}

\begin{proof}[Proof of Proposition \ref{p:universality}]
Based on Corollary \ref{thm:Holder}, the proof of Proposition \ref{p:universality} proceeds similarly to 
(\ref{eqn:integrate}), details are left to the reader.
\end{proof}

\section{Resolvent fluctuations from the dynamics noise}
\label{sec:resolvents}

It is well known \cite{AEK18} that for any $t\ge 0$ the resolvents $G_t^z$ becomes approximately deterministic as $n\to \infty$. Its deterministic approximation is given by
\begin{equation*}
M^z=M^z(\ii\eta):=\left(\begin{matrix}
m^z & -zu^z \\
-\overline{z}u^z & m^z
\end{matrix}\right), \qquad u^z=u^z(\ii\eta):=\frac{m^z(\ii\eta)}{\ii\eta+m^z(\ii\eta)},
\end{equation*}
with $m^z=m^z(\ii\eta)$ being the unique solution of the cubic equation
\begin{equation*}
-\frac{1}{m^z}=\ii\eta+m^z-\frac{|z|^2}{\ii\eta+m^z}, \qquad \eta\Im[m^z]>0.
\end{equation*}
Note that on the imaginary axis $m^z$ is purely imaginary and that $u^z$ is real. Furthermore, note that the deterministic approximation of $G_t$ does not depend on time since the first two moments of $X_t$ are preserved along the flow \eqref{eq:OU} and $M^z$ is determined only by those moments.
\subsection{Stochastic advection equation and consequences.}

\, In this section we analyze the evolution of the resolvent $G_t^z(\ii\eta):=(W_t-Z-\ii\eta)^{-1}$ along the flow \eqref{eq:OU}. In particular, we will study the evolution of the resolvent along the characteristics\footnote{Our convention here is that the characteristics move upwards as in (\ref{e:flow}), contrary to (\ref{eqn:NewCharacteristics}).} (cf. \cite[Eq. (5.4)]{CipErdSch2022bis} and \cite[Eq. (5.4)]{CipErdXu2023bis}):
\begin{equation}
\label{eq:char1}
\partial_s\eta_s=-\Im m^{z_s}(\ii\eta_s)-\frac{\eta_s}{2}, \qquad\quad \partial_s z_s=-\frac{z_s}{2}.
\end{equation}
In the following, for $\eta_*>0$ and $z_*\in \C$, by $\eta_s(\eta_*,t), z_s(z_*)$, for $s\in [0,t]$, we denote the solutions of \eqref{eq:char1} such that at time $t$ satisfy $\eta_t(\eta_*,t)=\eta_*$ and $z_t(z_*)=z_*$.

Here we perform the analysis in the regime $|z|\le 1- N^{-\epsilon} $, for an arbitrary small $\epsilon>0$, since this is the regime of interest of Proposition~\ref{prop:multitCLTres}; the regime $|z|> 1-N^{-\epsilon}$ is negligible (see e.g.  above \cite[Eq. (4.25)]{CipErdSch2023}).
The main result of this section is a decomposition theorem for $I_{\eta_c}^T(f,t)$ into the sum of two terms: one that depends only on the initial condition and one that is a martingale (see Proposition~\ref{pro:decomp}). Before stating this result we introduce the notation
 \begin{equation}
 \label{eq:e1e2}
 E_1:=\left(\begin{matrix}
 1 & 0 \\
 0 & 0 
 \end{matrix}\right), \qquad  E_2:=\left(\begin{matrix}
 0 & 0 \\
 0 & 1 
 \end{matrix}\right).
 \end{equation}
\begin{proposition}
\label{pro:decomp}
For any $s< t$, we have
\begin{equation}
\begin{split}
\label{eq:imprew2}
I_{\eta_c}^T(f,t)&=-\frac{1}{4\pi\ii}\int_\mathbb{C}\Delta f(z)\int_{\eta_s(\eta_c,t)}^{\eta_s(T,t)} \mathrm{Tr}\big[G_s^{z_s(z)}(\eta')-\E G_s^{z_s(z)}(\eta')\big]\, \dif\eta'\dif  z \\
&\quad+\frac{1}{4\pi \ii\sqrt{N}}\sum_{i\ne j}\int_\mathbb{C}\Delta f(z) \int_s^t \int_{\eta_\tau(\eta_c,t)}^{\eta_\tau(T,t)} \mathrm{Tr} \big[G_\tau^{z_\tau(z)}(\ii\eta')^2E_i\dif B_\tau E_j\big]\Bigg]\, \dif\eta\dif z +\OO\left(\frac{N^\xi}{N\eta_c}\right),
\end{split}
\end{equation}
with very high probability for any $\xi>0$. Here $\sum_{i\ne j}$ denotes a summation over the indices $(i,j)\in\{(1,2),(2,1)\}$.
\end{proposition}
This immediately shows the decomposition in \eqref{eq:defint12} with the integrals $I_1, I_2$ being defined as
\begin{equation*}
\begin{split}
I_1=I_1(f,s,t):&=-\frac{1}{4\pi\ii}\int_\mathbb{C}\Delta f(z)\int_{\eta_s(\eta_c,t)}^{\eta_s(T,t)} \mathrm{Tr}\big[G_s^{z_s(z)}(\eta')-\E G_s^{z_s(z)}(\eta')\big]\, \dif\eta'\dif  z,\\
I_2=I_2(f,s,t):&=\frac{1}{4\pi \ii\sqrt{N}}\sum_{i\ne j}\int_\mathbb{C}\Delta f(z) \int_s^t \int_{\eta_\tau(\eta_c,t)}^{\eta_\tau(T,t)} \mathrm{Tr} \big[G_\tau^{z_\tau(z)}(\ii\eta')^2E_i\dif B_\tau E_j\big]\Bigg]\, \dif\eta\dif \tau \dif z.
\end{split}
\end{equation*}

\begin{proof}[Proof of Proposition~\ref{pro:decomp}]

By It\^{o}'s formula (using the short--hand notations $G=G_t^z(\ii\eta)$) we obtain
\begin{equation}
 \label{eq:nochar1}
 \begin{split}
 \dif \langle G(\ii\eta)\rangle&=-\frac{1}{\sqrt{N}}\sum_{i\ne j}\langle G^2E_i\dif B_t E_j\rangle+\frac{1}{2}\langle W G^2\rangle\dif t+2\sum_{i\ne j}\langle GE_i\rangle\langle G^2E_j\rangle \dif t \\
 &=-\frac{1}{\sqrt{N}}\sum_{i\ne j}\langle G^2E_i\dif B_t E_j\rangle+\frac{1}{2}\langle G\rangle\dif t+\frac{1}{2}\langle (Z+\ii\eta)G^2\rangle \dif t+2\sum_{i\ne j}\langle GE_i\rangle\langle G^2E_j\rangle \dif t.
 \end{split}
 \end{equation}

In order to control the evolution of the resolvent along the flow \eqref{eq:nochar1} we consider the characteristics \eqref{eq:char1}. Note that along the characteristics we have 
\begin{equation}
\label{eq:usinf}
\eta_0=e^{t/2}\eta_t+(e^{t/2}-e^{-t/2})\Im m^{z_t}(\ii \eta_t), \quad z_0=e^{t/2}z_t,\quad m^{z_0}(\ii\eta_0)=e^{-t/2} m^{z_t}(\ii\eta_t), \quad u^{z_0}(\ii\eta_0)=e^{-t}u^{z_t}(\ii\eta_t).
\end{equation}

Then, by plugging \eqref{eq:char1} into \eqref{eq:nochar1}, we get (using the notations $G_t:=G_t^{z_t}(\ii\eta_t)$ and $M_t:=M^{z_t}(\ii\eta_t)$)
  \begin{equation}
 \label{eq:charG}
\dif \langle G_t\rangle=-\frac{1}{\sqrt{N}}\sum_{i\ne j}\langle G_t^2E_i\dif B_t E_j\rangle+\frac{1}{2}\langle G_t\rangle\dif t+\langle G_t-M_t\rangle \langle G^2\rangle\dif t,
\end{equation}
 where we used that by the chiral symmetry of $W_t-Z$ we have $2\langle G_t E_i\rangle=\langle G_t\rangle$. Subtracting the expectation in \eqref{eq:charG}, we obtain
   \begin{equation*}
\dif \langle G_t-\E G_t\rangle=-\frac{1}{\sqrt{N}}\sum_{i\ne j}\langle G_t^2E_i\dif B_t E_j\rangle+\frac{1}{2}\langle G_t-\E G_t \rangle\dif t+\langle G_t-M_t\rangle \langle G_t^2\rangle\dif t-\E\langle G_t-M_t\rangle \langle G_t^2\rangle\dif t.
\end{equation*}
Then, writing $G_t^2=M_t'+(G_t^2-M_t')$ (here prime denotes the derivative $\partial_{\ii\eta}$), we get
\[
\dif \langle G_t-\E G_t\rangle=-\frac{1}{\sqrt{N}}\sum_{i\ne j}\langle G_t^2E_i\dif B_t E_j\rangle+\left(\frac{1}{2}+\langle M_t'\rangle\right)\langle G_t-\E G_t \rangle \dif t+\langle G_t-M_t\rangle \langle G_t^2-M_t'\rangle\dif t-\E\langle G_t-M_t\rangle \langle G_t^2-M_t'\rangle\dif t,
\]
so that 
\begin{equation}
\begin{split}
\label{eq:maybeus}
\langle G_t-\E G_t\rangle&=\exp\left(\int_0^t\left[\frac{1}{2}+\partial_{\eta_s}m^{z_s}(\ii\eta_s)\right]\,\dif s\right)\langle G_0-\E G_0\rangle \\
&\quad-\frac{1}{\sqrt{N}}\sum_{i\ne j}\int_0^t \exp\left(\int_0^s\left[\frac{1}{2}+\partial_{\eta_\tau}m^{z_\tau}(\ii\eta_\tau)\right]\,\dif \tau\right) \langle G_s^2E_i\dif B_s E_j\rangle\,\dif s+\OO\left(\frac{N^\xi}{(N\eta_t)^2}\right)\end{split}
\end{equation}
with very high probability, where we used that $|\langle G_t^2-M_t'\rangle|\lesssim N^\xi (N\eta_t^2)^{-1}$ for any arbitrary small $\xi>0$ by \cite[Eq. (3.26)]{CipErdXu2025}
for $z_1=z_2$ and $\eta_1=\eta_2$, and to estimate the error term we used that for $s_1\le s_2\le t$ we have
\[
\exp\left(\int_{s_1}^{s_2}\left[\frac{1}{2}+\partial_{\eta_\tau}m^{z_\tau}(\ii\eta_\tau)\right]\,\dif \tau\right)\lesssim \frac{\eta_{s_1}}{\eta_{s_2}},
\]
as a consequence of $|\partial_{\eta_\tau}m^{z_\tau}(\ii\eta_\tau)|\le \rho_\tau/\eta_\tau$.

We now plug the expression \eqref{eq:maybeus} into Girko's formula \eqref{eq:girko1}; for simplicity of notation within this proof we choose $s=0$. Choose $\eta_0=\eta_0(\eta)=\eta_0(\eta,t)$ and $z_0=z_0(z)=z_0(z,t)$ such that $z_t=z$, $\eta_t=\eta$, with $\eta_t,z_t$ being the solutions of \eqref{eq:char1} with initial conditions $\eta_0$, $z_0$. Then plugging \eqref{eq:maybeus} into Girko's formula \eqref{eq:girko1} we find that
\begin{align}
I_{\eta_c}^T(f,t)&=-\frac{1}{4\pi\ii}\int_\mathbb{C}\Delta f(z)\int_{\eta_c}^T \exp\left(\int_0^t\left[\frac{1}{2}+\partial_{\eta_s}m^{z_s}(\ii\eta_s)\right]\,\dif s\right)\mathrm{Tr}\big[G_0^{z_0(z)}(\ii\eta_0(\eta))-\E G_0^{z_0(z)}(\ii\eta_0(\eta))\big] \notag\\
&\quad+\frac{1}{4\pi\ii \sqrt{N}}\sum_{i\ne j}\int_\mathbb{C}\Delta f(z)\int_{\eta_c}^T\int_0^t \exp\left(\int_0^s\left[\frac{1}{2}+\partial_{\eta_\tau}m^{z_\tau}(\ii\eta_\tau)\right]\,\dif \tau\right) \mathrm{Tr} \big[G_s^{z_s(z)}(\ii\eta_s(\eta))^2E_i\dif B_sE_j\big]\, \dif\eta\dif z^2 \notag\\
&\quad+\OO\left(\frac{N^\xi}{N\eta_c}\right),\label{eq:imprew}
\end{align}
with $z_s(z)=e^{(t-s)/2}z$. Performing the change of variables $\eta_0(\eta,t)\to \eta'$ and $\eta_s(\eta,t)\to \eta'$, in the first and second line of \eqref{eq:imprew} respectively, and using
\begin{equation*}
\partial_\eta\eta_0(\eta)=\partial_{\eta_t}\eta_0(\eta_t)=\exp\left(\int_0^t\left[\frac{1}{2}+\partial_{\eta_s}m^{z_s}(\ii\eta_s)\right]\,\dif s\right),
\end{equation*}
we obtain \eqref{eq:imprew2}.
\end{proof}

\subsection{Proof of Proposition~\ref{prop:multitCLTres}.}\ 
Proposition~\ref{prop:multitCLTres} consists of two main statements: i) the integrals $I_{\eta_c}^T(f_{v,a}^{(i)},t_i)$ obey a Wick theorem and we have an explicit formula for their covariance, ii) each one of those integrals can be decomposed as the sum of two terms (see \eqref{eq:covsplit}). The main input for both these results is the decomposition of $I_{\eta_c}^T(f,t)$ from Proposition~\ref{pro:decomp}.

The Wick theorem factorization in \eqref{eq:CLTtimecor} immediately follows from \eqref{eq:imprew2} together with \cite[Proposition 3.3]{CipErdSch2023}. We present its proof in Section~\ref{sec:wick}.

We now turn to the proof of ii). We recall that, by Proposition~\ref{pro:decomp}, we immediately obtain \eqref{eq:defint12} and that $I_1, I_2$ are (approximately) uncorrelated. Furthermore, the first relation in \eqref{eq:covsplit} immediately follows by \cite[Proposition 3.3]{CipErdSch2023} as it consists of the variance of the product of two resolvents evaluated at the same time; hence to conclude the proof of ii) we only need to compute the variance of the martingale term $I_2$.

In the following, to keep the presentation simple, we present the proof of Proposition~\ref{prop:multitCLTres} only in the macroscopic case $a=0$. The proof in the mesoscopic case is completely analogous after replacing any reference to \cite{CipErdSch2023} with the corresponding one in \cite{CipErdSch2022bis}. Additionally, for simplicity of notation we also assume that $s=0$; the general case can be achieved by a simple time--shift.

We divide this section into three subsections: in Section~\ref{sec:timecor} we compute  $\E I_{\eta_c}^T(f,t), I_{\eta_c}^T(g,0)$, in Section~\ref{sec:mart} we compute the variance of the martingale term in \eqref{eq:imprew2}, concluding the proof of ii), in Section~\ref{sec:wick} we prove a Wick theorem, concluding the proof of i).

Before presenting these proofs we comment on their relation with the proof of  \cite[Proposition 3.3]{CipErdSch2023}. In the reminder of this section the local law \cite[Theorem 5.2]{CipErdSch2022bis} is a fundamental input to compute the deterministic approximation of product of resolvents of the form $G_s^{z_1}G_s^{z_2}$, i.e. this local law is used for any fixed time $0\le s\le t$, but only for resolvents evaluated at the same time. On the other hand, the CLT for resolvents \cite[Proposition 3.3]{CipErdSch2023} is used only at time $s=0$ (see \eqref{eq:incondcorr} below) to compute the correlation of the initial condition.

\subsubsection{Correlation between $I_{\eta_c}^T(f,t), I_{\eta_c}^T(g,0)$.}
\label{sec:timecor}
From now on $c>0$ is a small fixed constant that may change from line to line. By \cite[Proposition 3.3]{CipErdSch2023}, for $\eta_i\ge n^{-1+\delta_1}$, it follows that\footnote{The term $|z_1-z_2|^{-4}$ in \cite[Eq. (3.13)]{CipErdSch2023} is just a consequence of the sub-optimal local law in \cite[Eq. (5.9)]{CipErdSch2023}, and it can readily be removed using the local law \cite[Eq. (3.26)]{CipErdXu2025} (see e.g. \cite[Proposition 3.4]{CipErdSch2022bis} for the regime $|z|\le 1-c$). Additionally, if  $z_0(z_1)> 1-N^{-C\epsilon}$, for a fixed $C>0$, then inspecting the proof of \cite[Proposition 3.3]{CipErdSch2023} it is easy to see that the term $1-|z_i|$ in \cite[Eq. (3.13)]{CipErdSch2023} can be replaced by $N^{C\epsilon}$. In the complementary regime we use \cite[Proposition 3.3]{CipErdSch2023} verbatim.}
\begin{equation}
\label{eq:incondcorr}
\begin{split}
&\mathbb{E}\mathrm{Tr}\big[G_0^{z_0(z_1)}(\ii \eta_1)-\E G_0^{z_0(z_1)}(\ii\eta_1)\big]\mathrm{Tr}\big[G_0^{z_2}(\ii\eta_2)-\E G_0^{z_2}(\ii\eta_2)\big] \\
&=\partial_{\eta_1}\partial_{\eta_2}\log\Big[1+e^{-t}(|z_0(z_1)z_2|^2u_1^2u_2^2-m_1^2m_2^2)-2e^{-t/2}u_1u_2\Re[z_0(z_1)\overline{z_2}]\Big]-\kappa_4\partial_{\eta_1}\partial_{\eta_2} m_1^2m_2^2+\OO\left(\frac{N^{-c}}{\eta_1\eta_2}\right),
\end{split}
\end{equation}
where we used the notations $m_1:=m^{z_0(z_1)}(\ii\eta_1)$, $u_1:=u^{z_0(z_1)}(\ii\eta_1)$, and similar notations for $m_2,u_2$.

Adding back the small $\eta_1$ and $\eta_2$ regimes (see the proof of \cite[Lemma 4.7]{CipErdSch2023}), and performing the $(\eta_1,\eta_2)$-integrals in Girko's formula we conclude that (note that the expectation of the martingale term in \eqref{eq:imprew} vanishes)
 \begin{equation}
 \label{eq:usefform}
 \begin{split}
 \E I_{\eta_c}^T(f,t) I_{\eta_c}^T(g,0)&=\frac{1}{8\pi^2}\int\int \Delta f(z_1) \Delta g(z_2)\, \dif z_1\dif z_2\\
 &\quad\times \bigg[-\frac{1}{2}\log \bigg[1+|z_2z_0(z_1)|^2(u^{z_0(z_1)}(\ii\eta_0(\eta_1)))^2(u^{z_2}(\ii\eta_2))^2 \\
 &\qquad\quad-(m^{z_0(z_1)}(\ii\eta_0(\eta_1)))^2(m^{z_2}(\ii\eta_2))^2-2u^{z_0(z_1)}(\ii\eta_0(\eta_1))u^{z_2}(\ii\eta_2)\Re[z_0(z_1)\overline{z_2}] \bigg]  \\
 &\qquad\quad +\frac{\kappa_4}{2} m^{z_0(z_1)}(\ii\eta_0(\eta_1))^2m^{z_2}(\ii\eta_2)^2\bigg]\bigg|_{\eta_1=0, \atop \eta_2=0}+\OO(N^{-c}).
 \end{split}
 \end{equation}

Then, using \eqref{eq:usinf},  for the $\log$--term in \eqref{eq:usefform}, we conclude
 \begin{equation}
 \label{eq:usident}
 \begin{split}
& \log \big[1+|z_2z_0(z_1)|^2(u^{z_0(z_1)}(\ii\eta_0(\eta_1)))^2(u^{z_2}(\ii\eta_2))^2-(m^{z_0(z_1)}(\ii\eta_0(\eta_1)))^2(m^{z_2}(\ii\eta_2))^2-2u^{z_0(z_1)}(\ii\eta_0(\eta_1))u^{z_2}(\ii\eta_2)\Re[z_0(z_1)\overline{z_2}] \big] \\
&\quad  = \log \big[1+e^{-t}(|z_2z_1|^2u^{z_1}(i\eta_1)^2(u^{z_2}(\ii\eta_2))^2- m^{z_1}(i\eta_1)^2
 (m^{z_2}(\ii\eta_2))^2)-2e^{-t/2}u^{z_1}(i\eta_1)u^{z_2}(\ii\eta_2)\Re[z_1\overline{z_2}] \big].
 \end{split}
\end{equation}
Evaluating \eqref{eq:usident} at $\eta_1=\eta_2=0$ we obtain
\begin{equation}
\label{eq:finanswer}
\mathrm{rhs.} \eqref{eq:usident}
=\Theta(z_1,z_2,t):=-\frac{1}{2}
\begin{cases}
\log\big[|(1-e^{-t})(1-|z_1|^2)+|z_1-e^{-t/2}z_2|^2 \big] & |z_1|,|z_2|\le 1 \\
\log|z_l-e^{-t/2}z_m|^2-\log|z_l|^2 &|z_m|\le 1,\,|z_l|>1 \\
\log|e^{-t/2}-z_1\overline{z_2}|^2-\log|z_1z_2|^2 & |z_1|,|z_2|>1.
\end{cases}
\end{equation}
Furthermore, for $\eta_1=\eta_2=0$ we also compute the term in the last line of  \eqref{eq:usefform}:
\begin{equation}
\label{eq:finanswerkappa4}
m^{z_0(z_1)}(\ii\eta_0(\eta_1))^2=e^{-t}m^{z_1}(\ii\eta_1)=e^{-t}(1-|z_1|^2), \qquad m^{z_2}(\ii\eta_2)^2=(1-|z_2|^2).
\end{equation}

Plugging \eqref{eq:finanswer}--\eqref{eq:finanswerkappa4} into \eqref{eq:usefform} we thus obtain
\begin{equation}
\label{eq:concleq}
 \E I_{\eta_c}^T(f,t) I_{\eta_c}^T(g,0)\approx \frac{1}{8\pi^2}\int\int \Delta f(z_1) \Delta g(z_2)\, \dif z_1\dif z_2 \Big[\Theta(z_1,z_2,t)+\frac{\kappa_4}{2} e^{-t} (1-|z_1|^2)(1-|z_2|^2)\Big].
\end{equation}

We now conclude the proof of \eqref{eq:CLTtimecor} performing integration by parts in $z_1$ and $\overline{z_2}$ in \eqref{eq:concleq}. Here we omit the details of the tedious explicit computations to compute the second line of \eqref{eq:CLTtimecor} given \eqref{eq:finanswer}, since they are analogous to \cite{CipErdSch2023}; we only point out the very minor differences. The computations of the term with the fourth cumulant $\kappa_4$ are exactly the same as in \cite[Lemma 4.10]{CipErdSch2023}, in fact the only difference compared to \cite{CipErdSch2023} is the scaling factor $e^{-t}$ in \eqref{eq:concleq}. For the term with $\Theta(z_1,z_2,t)$, using that
\[
\partial_{z_1}\partial_{\overline{z_2}}\log\big[|z_l-e^{-t/2}z_m|^2]=0,
\]
when $|z_m|\le 1$, $|z_l|>1$ and proceeding exactly as in \cite[Appendix A]{CipErdSch2023}, we conclude (cf. \cite[Lemma 4.9]{CipErdSch2023})
\begin{equation}
\label{eq:almostcomp}
\frac{1}{8\pi^2}\int_{\C^2}\Delta f(z_1) \Delta g(z_2)\Theta(z_1,z_2,t)\, \dif z_1\dif z_2=-\int_{\D^2}\partial_{\overline{z_1}}f\partial_{z_2} g\partial_{z_1}\partial_{\overline{z_2}}K(z_1,z_2,t)\, \dif z_1\dif z_2+\lim_{\epsilon\to 0}\mathcal{I}_\epsilon,
\end{equation}
with $K(z_1,z_2,t)$ from \eqref{eqn:kernel}, and
\begin{equation*}
\begin{split}
\mathcal{I}_\epsilon:&=\frac{1}{2\pi^2} \int_{\abs{z_1}\ge 1} \dif z_1 \int_{\substack{\abs{e^{-t/2}-z_1\overline{z}_2}\ge \epsilon,                                                                              \\ \abs{z_2}\ge 1}} \dif z_2 \, \partial_{z_1} f(z_1) \partial_{\overline{z_2}} g(z_2)  \frac{1}{(e^{-t/2}-\overline{z}_1z_2)^2} \\
             & \quad+ \frac{1}{2\pi^2} \int_{\abs{z_1}\ge 1} \dif z_1 \int_{\substack{\abs{e^{-t/2}-z_1\overline{z}_2}\ge \epsilon,                                                                                                   \\ \abs{z_2}\ge 1}} \dif z_2\,  \partial_{\overline{z_1}} f(z_1) \partial_{z_2} \overline{g(z_2)} \frac{1}{(e^{-t/2}-z_1\overline{z}_2)^2}.
\end{split}
\end{equation*}
To compute the second term in the rhs. of \eqref{eq:almostcomp} we proceed similarly to \cite[Eqs. (4.33)--(4.35)]{CipErdSch2023}. Using the change of variables $\overline{z_1}\to 1/\overline{z_1}$ and $z_2\to1/z_2$ the integrals in $\mathcal{I}_\epsilon$ are equal to the integral of $1/(e^{-t/2}\overline{z_1}z_2-1)^2$ over the domain $|z_i|\le 1$ and $|e^{-t/2}\overline{z_1}z_2-1|\ge \epsilon$. By a density argument it is enough to compute the second term in the rhs. of \eqref{eq:almostcomp} for polynomials
\[
f(z_1)=\sum_{k,l\ge 0} a_{kl} z_1^k\overline{z_1}^l, \qquad g(z_2)=\sum_{k,l\ge 0} b_{kl}z_2^k\overline{z_2}^l.
\]
Using that
\[
\lim_{\epsilon\to 0}\int_{|z_1|\le 1}\int_{|e^{-t/2}\overline{z_1}z_2-1|\ge\epsilon, \atop |z_2|\le 1} z_1^\alpha\overline{z_1}^\beta z_2^{\alpha'}\overline{z_2}^{\beta'}\,\dif z_1\dif z_2=\frac{\pi^2}{(\alpha+1)(\alpha'+1)}\delta_{\alpha\beta}\delta_{\alpha'\beta'}
\]
and proceeding exactly as in \cite[Eqs. (4.34)]{CipErdSch2023} we conclude
\begin{equation}
\label{eq:poissonkernel}
\lim_{\epsilon\to 0}\mathcal{I}_\epsilon=\frac{1}{2}\sum_{k,k',l,l'\ge0 \atop m\in\mathbf{Z}}|m|e^{-\frac{t}{2}|m|}a_{k,l}\overline{b_{k',l'}}\delta_{k,l+m}\delta_{k',l'+m}=\frac{1}{2}\langle f,P_t g\rangle_{H^{1/2}(\partial\D)},
\end{equation}
with $P_t$ being the Poisson kernel defined as $P_tf(z):=\sum_\Z e^{-|m|t/2}\widehat{f}_mz^m$. Combining \eqref{eq:concleq}--\eqref{eq:almostcomp} and \eqref{eq:poissonkernel} we conclude the proof of \eqref{eq:correxplis}, and so of part i).

\subsubsection{Second moment of the martingale term.}
\label{sec:mart}

We now compute the second moment of the martingale term in the second line of \eqref{eq:imprew2}. For fix $(i,j)\in\{(1,2),(2,1)\}$, we compute
\[
\E\int_0^t \mathrm{Tr}\big[ [G_s^{z_{1,s}}(\ii\eta_1)]^2E_i\dif B_s E_j\big]\int_0^t \mathrm{Tr}\big[ [G_s^{z_{2,s}}(\ii\eta_2)]^2E_i\dif B_s E_j\big] 
=\E\int_0^t \mathrm{Tr}\big[ G_s^{z_{1,s}}(\ii\eta_1)^2E_iG_s^{z_{2,s}}(\ii\eta_2)^2 E_j\big] \,\dif s.
\]
Furthermore, we note that
\[
\mathrm{Tr}\big[ G_s^{z_{1,s}}(\ii\eta_1)^2E_iG_s^{z_{2,s}}(\ii\eta_2)^2 E_j\big] =-\partial_{\eta_1}\partial_{\eta_2}\mathrm{Tr}\big[ G_s^{z_{1,s}}(\ii\eta_1)E_iG_s^{z_{2,s}}(\ii\eta_2) E_j\big].
\]
Plugging this into the computation of the second moment of the martingale term $I_2$ in \eqref{eq:imprew2}, we obtain
\begin{equation}
\label{eq:imprew2a}
\frac{1}{16 \pi^2 n}\iint_\mathbb{C}\Delta f(z_1)\Delta f(z_2) \int_0^t \iint_{\eta_s(\eta_c,t)}^{\eta_s(T,t)}\partial_{\eta_1}\partial_{\eta_2}\mathrm{Tr}\big[ G_s^{z_{1,s}}(\ii\eta_1)E_iG_s^{z_{2,s}}(\ii\eta_2) E_j\big]\, \dif\eta_1\dif\eta_2\dif z_1^2\dif z_2^2.
\end{equation}
We can thus perform the $(\eta_1,\eta_2)$--integrations in  \eqref{eq:imprew2a} and obtain
\begin{equation}
\label{eq:imprew2aa}
\frac{1}{16 \pi^2 n}\iint_\mathbb{C}\Delta f(z_1)\Delta f(z_2) \int_0^t \mathrm{Tr}\big[ G_s^{z_{1,s}}(\ii\eta_{1,s}(\eta_c))E_iG_s^{z_{2,s}}(\ii\eta_{2,s}(\eta_c)) E_j\big]\, \dif z_1^2\dif z_2^2.
\end{equation}
Here we omitted a negligible error smaller than $n^{-10}$ coming from the upper extreme of integration in the $(\eta_1,\eta_2)$--integrals. 

We now compute the leading deterministic term in \eqref{eq:imprew2aa}. For this purpose we recall the multi--resolvent local law from
\cite[Eq. (3.26)]{CipErdXu2025}.
Define $G_i:=(W-Z_i-\ii\eta_i)^{-1}$, with $Z_i$ as in \eqref{eq:herm}, and denote by $M_i$ its deterministic approximation. Then, for a deterministic matrix $A\in\C^{2N\times 2N}$, the deterministic approximation of $G_1A G_2$ is given by 
\begin{equation*}
M_A^{z_1,z_2}:=\big(1-M_1\mathcal{S}[\cdot]M_2\big)^{-1}[M_1AM_2].
\end{equation*}
Here $\mathcal{S}[\cdot]$ denotes the \emph{covariance operator}, which is defined by (recall the definition of $E_1, E_2$ from \eqref{eq:e1e2})
\[
\mathcal{S}[\cdot]:=2\langle \cdot E_1\rangle E_2+2\langle \cdot E_2\rangle E_1.
\]
By the local law from
\cite[Eq. (3.26)]{CipErdXu2025}, we then have
\begin{equation}
\label{eq:G1G2llaw}
\left|\left\langle \big[G_s^{z_{1,s}}(\ii\eta_{1,s}(\eta_c))E_iG_s^{z_{2,s}}(\ii\eta_{2,s}(\eta_c))-M_{E_i}^{z_{1,s},z_{2,s}}(\ii \eta_{1,s}(\eta_c),\ii\eta_{2,s}(\eta_c))\big] E_j\right\rangle\right|\lesssim \frac{N^\xi}{N\eta_{*,s}^2},
\end{equation}
with very high probability for $\xi>0$ arbitrary small and $\eta_{*,s}:=\eta_{1,s}(\eta_c)\wedge \eta_{2,s}(\eta_c)$. We are thus left to compute
\begin{equation*}
\int_0^t \langle M_{E_i}^{z_{1,s},z_{2,s}}(\ii \eta_{1,s}(\eta_c),\ii\eta_{2,s}(\eta_c))E_j\rangle\, \dif s,
\end{equation*}
and then plug the answer into \eqref{eq:imprew2aa}. In fact the error term in \eqref{eq:G1G2llaw}, after the time integration, can by estimated by $N^{-c}$ for some small fixed $c>0$ and so shown to be negligible.

Using the notations $m_{i,s}:=m^{z_{i,s}}(\ii\eta_{i,s}(\eta_c))$, $u_{i,s}:=u^{z_{i,s}}(\ii\eta_{i,s}(\eta_c))$, we then compute
\begin{equation}
\begin{split}
\label{eq:Mformula}
\sum_{(i,j)\in \{(1,2),(2,1)\}}&\langle M_{E_i}^{z_{1,s},z_{2,s}}(\ii \eta_{1,s}(\eta_c),\ii\eta_{2,s}(\eta_c))E_j\rangle\\
&=\frac{u_{1,s}u_{2,s}\Re[z_{1,s}\overline{z_{2,s}}]-|z_{1,s}z_{2,s}|^2u_{1,s}^2u_{2,s}^2+m_{1,s}^2m_{2,s}^2}{1+|z_{1,s}z_{2,s}|^2u_{1,s}^2u_{2,s}^2-m_{1,s}^2m_{2,s}^2-2u_{1,s}u_{2,s}\Re[z_{1,s}\overline{z_{2,s}}]} \\
&=-\frac{1}{2}\partial_s\log \big[1+e^{-2(t-s)}|z_1z_2|^2u_1^2u_2^2-e^{-2(t-s)}m_1^2m_2^2-2e^{-(t-s)}u_1u_2\Re[z_{1,s}\overline{z_{2,s}}] \big],
\end{split}
\end{equation}
where
we used that $z_{i,s}=e^{(t-s)/2}z_i$, $u_{i,s}=e^{-(t-s)}u_i$, and $m_{i,s}=e^{-(t-s)/2}m_i$. Plugging \eqref{eq:Mformula} evaluated at $\eta_1=\eta_2=0$ into \eqref{eq:imprew2aa}, we thus conclude (neglecting negligible errors of size $N^{-c}$)
\begin{equation}
\label{eq:finmart}
\begin{split}
&\E|I_2|^2=\E\left|\frac{1}{4\pi \ii\sqrt{N}}\int_\mathbb{C}\Delta f(z) \int_0^t \int_{\eta_s(\eta_c,t)}^{\eta_s(T,t)} \mathrm{Tr} \big[G_s^{z_s(z)}(\ii\eta')^2E_i\dif B_sE_j\big]\Bigg]\, \dif\eta'\dif z^2\right|^2 \\
&\quad=-\frac{1}{16\pi^2}\int_\C\int_\C \Delta f(z_1)\Delta f(z_2) \times \int_0^t \partial s \begin{cases}
\log\big[|(1-e^{-2(t-s)})(1-|z_1|^2)+|z_1-e^{-(t-s)}z_2|^2 \big] & |z_1|,|z_2|\le 1 \\
\log|z_l-e^{-(t-s)}z_m|^2-\log|z_l|^2 &|z_m|\le 1,\,|z_l|>1 \\
\log|e^{-(t-s)}-z_1\overline{z_2}|^2-\log|z_1z_2|^2 & |z_1|,|z_2|>1
\end{cases} \\
&\quad=-\frac{1}{16\pi^2}\int_\C\int_\C \Delta f(z_1)\Delta f(z_2) \times \begin{cases}
\log|z_1-z_2|^2-\log\big[|(1-e^{-2t)})(1-|z_1|^2)+|z_1-e^{-t}z_2|^2 \big] & |z_1|,|z_2|\le 1 \\
\log|z_1-z_2|^2-\log|z_l-e^{-t}z_m|^2 &|z_m|\le 1,\,|z_l|>1 \\
\log|1-z_1\overline{z_2}|^2-\log|e^{-t}-z_1\overline{z_2}|^2 & |z_1|,|z_2|>1 \\
\end{cases}
\end{split}
\end{equation}

Finally, performing integration by parts in \eqref{eq:finmart} as explained at the end of Section~\ref{sec:timecor}, but for $t=0$ (e.g. we are in the same setting of \cite{CipErdSch2023}), we conclude the second equality \eqref{eq:covsplit}, and so the proof of part ii) as well.

\subsubsection{Wick Theorem}
\label{sec:wick}
Consider a product of the form (recall that for notational simplicity we only consider the case $a=0$)
\begin{equation}
\label{eq:prodwick}
\E\prod_{i\in [p]} I_{\eta_c}^T(f^{(i)},t_i).
\end{equation}
By Proposition~\ref{pro:decomp}, each of these integrals can be decomposed as
\[
I_{\eta_c}^T(f^{(i)},t_i)=I_1(f^{(i)},0,t_i)+I_2(f^{(i)},0,t_i)+\mathcal{O}(N^{-c}).
\]
By the martingale representation theorem (see e.g. \cite[Theorem 18.12]{Kal02}), we can write (recall the definition of $I_2$ from below \eqref{eq:imprew2})
\begin{equation}
\label{eq:martrepthmnew}
I_2(f^{(i)},0,t_i)=\int_0^{t_i} \big(\widehat{\mathcal{C}_s}^{1/2} \dif \mathfrak{b}_s\big)_i, \qquad\quad i\in [p],
\end{equation}
where $\mathfrak{b}_s\in\R^p$ is a standard real $p$-dimensional Brownian motion, independent of the $I_1(f^{(i)},0,t_i)$'s, and $\widehat{\mathcal{C}}_s$ is the $p\times p$ covariance matrix of $\{I_2(f^{(i)},0,t_i), i\in [p]\}$, i.e. $\widehat{\mathcal{C}}_{ij}$ is defined analogously to the first line of \eqref{eq:finmart}. Note that $\widehat{\mathcal{C}}_s$ is random. We now show that we can approximate $I_2(f^{(i)},0,t_i)$ with a Gaussian process by replacing $\widehat{\mathcal{C}}_s$ with its deterministic approximation. Define
\[
(\mathcal{C}_s)_{ij}=\frac{V_s\big(f^{(i)}+f^{(j)}\big)-V_s\big(f^{(i)}-f^{(j)}\big)}{4}, \qquad\qquad\quad V_s(f):=\mathrm{integrand \,\, in\,\, the \,\, second\,\, line\,\, of\,\, rhs. \, \eqref{eq:finmart},}
\]
and the $p$-dimensional Gaussian process
\begin{equation}
\label{eq:gaussprocess}
\big(\textbf{Z}_s\big)_i:=\int_0^{t_i} \big(\mathcal{C}_s^{1/2}\, \dif \mathfrak{b}_s\big)_i, \qquad\quad i\in [p].
\end{equation}
By analogous computations to \eqref{eq:finmart} it follows that $\widehat{\mathcal{C}}_{ij}=\mathcal{C}_{ij}+\mathcal{O}(N^{-c})$, with very high probability. To estimate the difference between $(\textbf{Z}_s)_i$ and $I_2(f^{(i)},0,t_i)$ we thus compute the quadratic covariation
\begin{equation}
\begin{split}
\label{eq:approxnew}
\left|\langle (\textbf{Z}_s)_i-I_2(f^{(i)},0,t_i), (\textbf{Z}_s)_j-I_2(f^{(j)},0,t_j)\rangle_t\right|&=\left|\int_0^{t_i\wedge t_j} \big(\big[\widehat{\mathcal{C}_s}^{1/2}-\mathcal{C}_s^{1/2}\big]^2\big)_{ij}\, \dif s\right| \\
&\le \int_0^{t_i\wedge t_j}  \big(\big[\widehat{\mathcal{C}_s}^{1/2}-\mathcal{C}_s^{1/2}\big]^2\big)_{ii}^{1/2}\big(\big[\widehat{\mathcal{C}_s}^{1/2}-\mathcal{C}_s^{1/2}\big]^2\big)_{jj}^{1/2}\, \dif s \\
&\le \int_0^{t_i\wedge t_j}  \mathrm{Tr}\big|\widehat{\mathcal{C}_s}-\mathcal{C}_s\big|\, \dif s \\
&\le \int_0^{t_i\wedge t_j}  \sum_{i,j=1}^p \big(\widehat{\mathcal{C}_s}-\mathcal{C}_s\big)_{ij}\, \dif s \lesssim N^{-c},
\end{split}
\end{equation}
for a small $c>0$, where in the second inequality we used that $\mathrm{Tr}\big[(\widehat{\mathcal{C}_s}^{1/2}-\mathcal{C}_s^{1/2})^2\big]\le \mathrm{Tr}\big|\widehat{\mathcal{C}_s}-\mathcal{C}_s\big|$ (see e.g. \cite[Section 3]{Pos70}).
Then, from the BDG inequality it follows that $I_2(f^{(i)},0,t_i)=(\mathbf{Z}_s)_i+\mathcal{O}(N^{-c})$ with very high probability for all $i\in [p]$. In particular,
this shows that, modulo a negligible error $N^{-c}$, the $I_2(f^{(i)},0,t_i)$ are jointly a multivariate Gaussian random variable.

Now a Wick theorem as in \eqref{eq:CLTtimecor}  for the product \eqref{eq:prodwick} easily follows. In fact, \eqref{eq:gaussprocess}--\eqref{eq:approxnew} readily imply a Wick theorem for $I_2(f^{(i)},0,t_i)$; we are thus left only with products of the $I_1(f^{(i)},0,t_i)$'s . Then, by \cite[Proposition 3.3]{CipErdSch2023}, we also conclude that the $I_1(f^{(i)},0,t_i)$'s satisfy a Wick theorem as they consists of product of resolvents evaluated at the same time. This concludes the proof of of Proposition~\ref{prop:multitCLTres}.

\subsection{Proof of Corollary \ref{cor:overlaps}.}\  \label{subsec:revers}
From Theorem \ref{theo:mainresmeso}, for any compactly supported $F,G$ we expect
\begin{equation}\label{eqn:covAsymptotics}
\E\big[L_N(F_{v,a},s_a) L_N(G_{v,a},t_a)\big]=\Gamma_{v}(F,G,t-s)\cdot (1+{\rm o}(1))
\end{equation}
as $N\to\infty$.  Strictly speaking this convergence of the covariance does not follow from Theorem \ref{theo:mainresmeso}, which states convergence in distribution.
However,  an inspection of the proof (which proceeds by asymptotics of the moments) in the special equilibrium case immediately  shows that (\ref{eqn:covAsymptotics}) holds,  notably because (\ref{eq:smallbneed}) is correct at equilibrium, the matrix entries having a smooth density and Wegner estimates applying easily.

Second, we will need a slight generalization of (\ref{eqn:covAsymptotics}) to $F,G$ non-necessarily with compact support,  as it will be applied to $F=\tfrac{1}{2\pi}f_{v,a}*\log, G=\tfrac{1}{2\pi}g_{v,a}*\log$, with $f,g$ with compact support.  Again,  this generalization poses no problem:  The starting point of the proof,  Equation (\ref{eq:girko1}),  can be directly replaced with
\begin{equation*}
\sum_i F(\sigma_i)=\frac{\ii}{4\pi}\int_\C f_{v,a}(z)\int_0^\infty  \mathrm{Tr} \big[G^z(\ii\eta)\big]\, \dif\eta\dif z
\end{equation*}
and the remainder of the proof is then unchanged.

In sum, our starting point will be that for any $a>0$, $f,g\in H_0^2(\C)$, there exists a $c>0$ such that uniformly in $s,t$ in compact sets we have (remember $c_v=1-|v|^2$)
\begin{equation}\label{eqn:covarStart}
\E\big[L_N(F,s_a) L_N(G,t_a)\big]=-\frac{1}{16\pi^2}\int_{\C^2}f_{v,a}(z)g_{v,a}(w)\log\big(c_v|t_a-s_a|+ |z-w|^2\big)\, \dif z\dif w\cdot (1+{\rm O}(N^{-c})),
\end{equation}
where $F(z)=\tfrac{1}{2\pi}\int f_{v,a}(z-w)\log|w|\rd w$ and $G(z)$ is defined similarly.

On the other hand,  from  \cite[Proposition A.1]{BouDub2020} we have $\langle M_k,M_k\rangle=0$ so that $\langle{\rm Re} M_k,{\rm Im} M_k\rangle=0$ and  $\langle{\rm Re} M_k\rangle=\langle{\rm Im}M_k\rangle$.  The It{\^o} formula therefore gives
\begin{equation}\label{eqn:Ito}
\rd F(\sigma_k(t))=\nabla F\cdot \rd M_k-\frac{1}{2}\nabla F\cdot  \sigma_k\rd t+\frac{1}{4}\Delta F\,\rd\langle M_k\rangle=\nabla F\cdot \rd M_k-\frac{1}{2}\nabla F\cdot  \sigma_k\rd t+\frac{f}{4}\,\frac{\mathscr{O}_{kk}}{N}\rd t.
\end{equation}
Let $t'=t_1+s_2-s_1$. The difference $\E\big[L_N(F,s_2) L_N(G,t')\big]-\E\big[L_N(F,s_2) L_N(G,t_1)\big]$ can be written in two manners from (\ref{eqn:covarStart}) and (\ref{eqn:Ito}),  which gives 
\begin{align}\begin{split}
&\phantom{{}={}}\int_{t_1}^{t'}\E\left[L_N(F,s_2) \sum_{k}(-\frac{1}{2}\nabla G(\sigma_k(t))\cdot  \sigma_k(t)+\frac{g(\sigma_k(t))}{4}\,\frac{\mathscr{O}_{kk}(t)}{N})\right]\rd t\\
&=
-\int_{t_1}^{t'}
\frac{1}{16\pi^2}\int_{\C^2}f_{v,a}(z)g_{v,a}(w)\frac{c_v}{c_v|t-s_2|+ |z-w|^2}\, \dif z\dif w \rd t+\OO(N^{-4a-c+\e})\label{eqn:inter11}
\end{split}\end{align}
where we have used that $\int_{\C^2}f_{v,a}(z)g_{v,a}(w)\log\big(c_v|t_a-s_a|+ |z-w|^2\big)\, \dif z\dif w=\OO(N^{-4a+\e})$ for any $\e>0$, and that the expectation of the term containing $\rd M_k$ vanishes. 

We now wish to remove the contribution from the gradient term in (\ref{eqn:inter11}). 
For this we rely on the following rigidity estimates:
With  very high probability we have
\begin{equation}\label{eqn:rigid}
|L_N(F,s_2)|\leq N^{-2a+\e},\ \ \ \ \ \ \ \ \ \ 
\left|\sum_{k}\nabla G(\sigma_k)\cdot  \sigma_k-\E\sum_{k}\nabla G(\sigma_k)\cdot  \sigma_k\right|\leq N^{-a+\e}.
\end{equation}
Indeed, 
consider a partition of the the constant on $\mathbb{C}$: $1=\sum_{n\geq 1} \chi_n$, where $\chi_n$ is supported on $\{|z|\leq 2^i\}$ and  $\|\chi_n^{(k)}\|_{\infty}\leq C_k$ for all $n$.
We now define $F_n(z)=F(z)\chi_n(N^a(z-v))$.
On the support of $f$ we have  $F=\OO(N^{-2a+\e})$, and  more generally $\|F_n^{(k)}\|_{\infty}\leq {\widetilde C}_k N^{-2a+\e}N^{ka}$.  Together with \cite[Theorem 1.2]{BouYauYin2014},  this implies
$\mathbb{P}(|L_N(F_n,s_2)|\geq N^{-2a+\e})=\OO(N^{-D})$ for any fixed $\e,D>0$,  and $n=\OO(\log N)$.    With a union bound we conclude $\mathbb{P}(|L_N(F,s_2)|\geq N^{-2a+\e})=\OO(N^{-D})$, and finaly by taking large moments and H{\"o}lder's inequality this implies the bound on the left-hand side of (\ref{eqn:rigid}). The right-hand side of  (\ref{eqn:rigid}) follows from a similar dyadic decomposition,  as  we have 
$\nabla G=\OO(N^{-a+\e})$ on the support of $g$,  which is the most singular part of the test function.

In sum, by rigidity of the spectrum we have proved
\[
\int_{t_1}^{t'}\E\left[L_N(F,s_2) \sum_{k}\nabla G(\sigma_k(t))\cdot  \sigma_k(t)\right]\rd t=\OO(N^{-5a+\e}),
\]
which together with \ref{eqn:inter11} gives
\[
\int_{t_1}^{t'}\E\left[L_N(F,s_2) \sum_{k}\frac{g(\sigma_k(t))}{4}\,\frac{\mathscr{O}_{kk}(t)}{N}\right]\rd t\\
=
-\int_{t_1}^{t'}
\frac{1}{16\pi^2}\int_{\C^2}f_{v,a}(z)g_{v,a}(w)\frac{c_v}{c_v|t-s_2|+ |z-w|^2}\, \dif z\dif w \rd t+\OO(N^{-4a-c+\e}).
\]
By reversibility at equilibrium, we have
\[
\E\left[L_N(F,s_2) \sum_{k}\frac{g(\sigma_k(t))}{4}\,\frac{\mathscr{O}_{kk}(t)}{N}\right]=\E\left[L_N(F,t_1) \sum_{k}\frac{g(\sigma_k(s_2+t_1-t))}{4}\,\frac{\mathscr{O}_{kk}(s_2+t_1-t)}{N}\right],
\]
so that we have proved
\[
\int_{s_1}^{s_2}\E\left[L_N(F,t_1) \sum_{k}\frac{g(\sigma_k(s))}{4}\,\frac{\mathscr{O}_{kk}(s)}{N}\right]\rd s\\
=
-\int_{s_1}^{s_2}
\frac{1}{16\pi^2}\int_{\C^2}f_{v,a}(z)g_{v,a}(w)\frac{c_v}{c_v|t_1-s|+ |z-w|^2}\, \dif z\dif w \rd s+\OO(N^{-4a-c+\e}).
\]
By subtracting the analogous formula with $t_1$ replaced with $t_2$ we have
\begin{multline*}
\int_{s_1}^{s_2}\E\left[(L_N(F,t_2)-L_N(F,t_1) \sum_{k}\frac{g(\sigma_k(s))}{4}\,\frac{\mathscr{O}_{kk}(s)}{N}\right]\rd s\\
=
-\int_{s_1}^{s_2}
\frac{1}{16\pi^2}\int_{\C^2}f_{v,a}(z)g_{v,a}(w)\int_{t_1}^{t_2}\partial_t\frac{c_v}{c_v|t-s|+ |z-w|^2}\, \dif z\dif w \rd s\rd t+\OO(N^{-4a-c+\e}).
\end{multline*}
With (\ref{eqn:Ito}) to evaluate $L_N(F, t_2)-L_N(F, t_1)$ we obtain (\ref{eqn:DynOve}).
\qed

\subsection{Proof of Proposition \ref{prop:Markov}.}\ 
Note first that from Remark \ref{rem:stereo} we have (remember $s<t$ here)
$
\E[L(g,s)L(f,t)]=\E[L(g,s)L(Q_{t-s}f,s)]
$
for any $g\in\mathscr{C}^{\infty}$,
so that
\[
\E[L(f,t)\mid \Sigma_s]=L(Q_{t-s}f,s)
\]
almost surely.  In a proof by contradiction, 
assume that $\mathbb{P}\left(\E\left[L(f,t)\mid \Sigma_{s-}\right]=\E\left[L(f,t)\mid \Sigma_{s}\right] \right)=1$ for any $f\in\mathscr{C}^\infty$, and let $u<s$.  Together with $\Sigma_{u}\subset\Sigma_{s-}$ and the above equation thus gives
\begin{multline*}
L(Q_{t-u}f,u)=\E[L(f,t)\mid \Sigma_u]=\E[\E[L(f,t)\mid\Sigma_{s-}]\mid \Sigma_u]=\E[\E[L(f,t)\mid\Sigma_{s}]\mid \Sigma_u]\\
=\E[L(Q_{t-s}f,s)\mid \Sigma_u]=L(Q_{s-u}Q_{t-s}f,u)\  {\rm a.s.}
\end{multline*}
so that $L((Q_{t-u}-Q_{s-u}Q_{t-s})f,u)=0$ almost surely.  If $L(h,u)=0$ a.s. then its variance vanishes, so $\int|\nabla h|^2=0$ and $h$ is constant.  Hence $(Q_{t-u}-Q_{s-u}Q_{t-s})f=0$ ($Q_r$ preserves the mean).  As this holds for any $f\in\mathscr{C}^\infty$,   we have
\begin{equation}\label{eqn:convol}
q_{t-u}(z)=\int q_{s-u}(w)q_{t-s}(z-w)\rd w.
\end{equation}
A calculation gives, for any $\xi=(\xi_1,\xi_2)\in\mathbb{R}^2$ and defining $z\cdot\xi=\xi_1{\rm Re}z+\xi_2{\rm Im}z$,
\begin{align}\begin{split}\label{eqn:Bessel}
\hat q_r(\xi)&=\int q_r(z)e^{-\ii z\cdot\xi}\rd z=\int_0^\infty\int_{0}^{2\pi} e^{\ii x|\xi|\cos\theta}\frac{r}{\pi(r+x^2)^2}\rd\theta x\rd x\\
&= \int_0^\infty J_0(x|\xi|) \frac{2 r}{(r+x^2)^2}\rd\theta x\rd x=|\xi|\sqrt{r} K_1(|\xi|\sqrt{r}),
\end{split}\end{align}
where $J_\alpha$ (resp. $K_\alpha$) is the Bessel function of the first kind (resp. modified Bessel function of the second kind) with index $\alpha$.
Equation (\ref{eqn:convol}) therefore gives, for any $r_1,r_2,|\xi|>0$,
\[
|\xi|\sqrt{r_1+r_2}K_1(|\xi|\sqrt{r_1+r_2})=|\xi|^2\sqrt{r_1r_2}K_1(|\xi|\sqrt{r_1})K_1(|\xi|\sqrt{r_2}),
\]
which can be easily proved to be wrong, e.g. by choosing $r_2=|\xi|=1$ and comparing the Taylor series for small $r_1$.
\qed

\subsection{Existence of the limiting Gaussian field.} \label{sec:Fourier}\  
As explained in Remark \ref{rem:existence}, the existence of the limiting Gaussian field on mesoscopic scales is equivalent to $
\sum_{1\leq i,\,j,\leq m}\Gamma (f_i,f_j,|t_i-t_j|)\geq 0
$, i.e. 
\[
\sum_{1\leq i,\,j,\leq m} \int\partial_{\overline z} f_i(z)(q_{|t_i-t_j|}*\partial_{z}f_j)(z)\rd z\geq 0.
\]
In Fourier space we obtain the condition
\begin{equation*}
\sum_{1\leq i,\,j,\leq m} \int |\xi|^2 \hat f_i(\xi)\overline{\hat f_j}(\xi) \hat q_{|t_i-t_j|}(\xi)\rd\xi\geq 0.
\end{equation*}
For the above we clearly only need to prove $\sum_{1\leq i,\,j,\leq m}  v_i \overline{v_j} \hat q_{|t_i-t_j|}(\xi)\rd\xi\geq 0$ for any complex $\xi$,  $v_i$'s and real $t_i$'s.
By \eqref{eqn:Bessel} we have $\hat q_{|t_i-t_j|}(\xi)=|\xi|\cdot|t_i-t_j|^{1/2}K_1(|\xi|\cdot |t_i-t_j|^{1/2})$, so  it is sufficient to show
\begin{equation}\label{eqn:pos}
\sum_{1\leq i,\,j,\leq m} z_i\overline{z_j}|s_i-s_j|^{1/2}K_1(|s_i-s_j|^{1/2})\geq 0
\end{equation} 
for any complex $z_i$'s and real $s_i$'s. The function $g(t)=|t|^{1/2}K_1(|t|^{1/2})$ has non-negative Fourier transform:
\[
\hat g(u)=\int_{\mathbb{R}}|t|^{1/2}K_1(|t|^{1/2})e^{-\ii u t}\rd t=\frac{1}{4 u^2}\int_0^{\infty}\frac{v}{v^2+1}e^{-\frac{v}{4u}}\rd v\geq 0,
\]
which implies (\ref{eqn:pos}) and concludes the proof.
\qed

\appendix
\section{Stieltjes transform}
	In this section we recall various results concerning Stieltjes transforms. To that end, fix a probability measure $\mu$. We define the \emph{Stieltjes transform} of $\mu$ to be the function $m = m_{\mu} : \mathbb{H} \rightarrow \mathbb{H}$ for any complex number $z \in \mathbb{H}$ setting
	\begin{flalign*}
		m(z) = \displaystyle\int_{-\infty}^{\infty} \displaystyle\frac{\mu (\rd x)}{x-z}.
	\end{flalign*}
	Here $\mathbb{H}$ denotes the upper half complex plane.
We have the following estimates on the Stieltjes transform and its derivatives.
\begin{lemma}\label{l:STproperty}
Let $m(z)=m_\mu(z)$ be the Stieltjes transform of a probability measure $\mu$. For any integer $p\geq 1$, we denote its $p$-th derivative by $m^{(p)}(z)$. Then
\[
|m(z)|\leq \frac{1}{\dist(z,\supp(\mu))},\quad
\quad |\del_z m(z)|\leq \frac{\Im[m(z)]}{\Im[z]}, 
\quad |\del_z^p m(z)|\leq \frac{p!\Im[m(z)]}{\dist(z,\supp(\mu))^{p-1}\Im[z]}.
\]
\end{lemma}
\begin{proof}
The Stieltjes transform is given by 
\begin{align*}
|m(z)|=\left|\int_\bR \frac{\rd \mu(x)}{x-z}\right|\leq \int_\bR \frac{\rd \mu(x)}{|z-x|}\leq \frac{1}{\dist(z,\supp(\mu))}.
\end{align*}
For the second statement of \eqref{e:stbb}, we have
\begin{align*}
|\del_z m(z)|= \left|\int_\bR \frac{\rd \mu(x)}{(x-z)^{2}}\right|\leq \int_\bR \frac{\rd \mu(x)}{|z-x|^{2}}=\frac{1}{\Im[z]}\int_\bR \frac{\Im[z]\rd \mu(x)}{|z-x|^{2}}=\frac{\Im[m(z)]}{\Im[z]}.
\end{align*} 
For the last statement of \eqref{e:stbb}, we have
\begin{align*}
|\del_z^p m(z)|=\left|\int_\bR \frac{p!\rd \mu(x)}{(x-z)^{p+1}}\right|\leq \int_\bR \frac{p!\rd \mu(x)}{|z-x|^{p+1}}\leq \int_\bR \frac{p!\rd \mu(x)}{|z-x|^{2}\dist(z,\supp(\mu))^{p-1}}= \frac{p!\Im[m(z)]}{\dist(z,\supp(\mu))^{p-1}\Im[z]},
\end{align*}
concluding the proof.
\end{proof}

\begin{lemma}\label{l:STproperty}
Let $m(z)=m_\mu(z)$ be the Stieltjes transform of a probability measure $\mu$. For any small $\eta>0$, and two complex numbers $w, z$ with $\eta/C\leq \Im[w],\Im[z]\leq C\eta$, and $|w-z|\leq C\eta$, we have 
\begin{align}\label{e:Imwz}
\Im[m(w)]\asymp \Im[m(z)].
\end{align}
If we further assume that there exists a
control parameter $0<\phi\leq \Im[m(z)]$, such that 
\begin{align*}
|m(w)-\widetilde m(w)|, |m(z)-\widetilde m(z)|\leq \phi,
\end{align*}
and $|w-z|\asymp \eta\sqrt{\phi/\Im[m(z)]}$, then 
\begin{align}\label{e:stbb}
|\del_z m(z)|\lesssim \max|\del_z \widetilde m(u)|+\frac{\sqrt{\phi \Im[m(z)]}}{\eta}.
\end{align}
\end{lemma}
\begin{proof}
For the first statement \eqref{e:Imwz}, \eqref{e:stbb} gives that 
\begin{align*}
|\del_u \Im[m(u)]|\leq |\del_u m(u)|\leq |\frac{\Im[m(u)]}{\Im[u]},
\end{align*}
which gives that $|\del_u \log \Im[m(u)]|\leq 1/\Im[u]$. By integrating it from $z$ to $w$ and from $w$ to $z$, we conclude that 
\begin{align*}
\log \frac{\Im[m(w)]}{\Im[m(z)]}\lesssim 1,\quad \log \frac{\Im[m(z)]}{\Im[m(w)]}\lesssim 1.
\end{align*}
Then \eqref{e:Imwz} follows.

For the second statement \eqref{e:stbb}, by Taylor expansion 
\begin{align*}
|m(w)-m(z)-(w-z)\del_z m(z)|\leq \frac{1}{2}|w-z|^2\sup_u |\del_u^2 m(u)|
\end{align*}
where the supremum is on $[z,w]$.
Then we get 
\begin{align*}
|\del_z m(z)|
&\leq \frac{|m(w)-m(z)|}{|w-z|}+\frac{1}{2}|w-z|\sup_u |\del_u^2 m(u)|
\lesssim \max|\del_u \widetilde m(u)|+\frac{\phi}{|w-z|}+|w-z| \frac{\Im[m(z)]}{\Im[z]^2}\\
&\lesssim \max|\del_u \widetilde m(u)|+\frac{\sqrt{\phi \Im[m(z)]}}{\Im[z]},
\end{align*}
where we used that $|w-z|\asymp\Im[z]\sqrt{\phi/ \Im[m(z)]}$.
\end{proof}

\section{Well-posedness}\label{appendix:WellPosedness}

The purpose of this short appendix is to prove that the Equation (\ref{e:coupleB}) is well-posed. More generally, we consider the 
stochastic differential equation
\begin{align}\label{e:generalForm}
\rd x_i(t)=\frac{\rd b_i(t)}{\sqrt{2N}}+\frac{1}{2N}\sum_{j\neq i}\frac{1}{x_i(t)-x_j(t)}\rd t, \quad 1\leq |i|\leq N,
\end{align}
where $0<x_1(0)<\dots<x_N(0)$,  $x_{-i}(0)=-x_i(0)$,  $(i\geq 1)$,    $(b_i)_{1\leq i\leq N}$ is a collection of continuous  martingales,  and $b_{-i}(t)=-b_i(t)$,  $(i\geq 1)$.
Note that the drift in (\ref{e:generalForm}) includes a repulsive term between $x_i$ and $x_{-i}$,  equal to $\tfrac{1}{2N}\tfrac{1}{2 x_i}$.

\begin{proposition}\label{prop:WellPosed} If 
$
\frac{\rd \langle b_i\rangle_t}{\rd t}\leq 1
$
for any $i\geq 1,t\geq 0$, then  existence and strong uniqueness hold for the stochastic differential equation (\ref{e:generalForm}).
\end{proposition}

\noindent Note that this proposition does not require indepedence of the $b_i$'s,  or bounds on the off-diagonal brackets.

\begin{proof}
For any $\e>0$, let $\tau_\e=\inf\{t\geq 0:|x_i(t)-x_j(t)|=\e\ {\mbox{for some }i\neq j}\}$ and $\mu_C=\inf\{t\geq 0:|x_i(t)|=C\ {\mbox{for some }i}\}$.  It is well-known that  existence and strong uniqueness hold if one proves $\tau_\e\to\infty$ a.s. as $\e\to 0$,  and $\mu_C\to\infty$ as $C\to\infty$, see e.g.  \cite{RogShi1993}.

Still following \cite{RogShi1993}, consider $\phi(\bx)=-\sum^*_{i,j}\log|x_i-x_j|$, where here and in the following $\sum_{i,j}^*$ (resp.  $\sum^*_i$) means summation over couples $-N\leq i\neq j\leq N$ such that $i,j\neq 0$ (resp.  $-N\leq i\leq N$ such that $i\neq 0$). Let $\by_t:=\bx_{t\wedge\tau_\e}$, then It{\^o}'s formula gives
\begin{align*}
\rd\phi(\by_t)&=\sum^*_i\partial_i\phi(\by_t)\rd y_i(t)+\frac{1}{2}\sum_{1\leq |i|,|j|\leq N}\partial_{i,j}\phi(\by_t)\rd\langle y_i,y_j\rangle_t\\
&=-\sum_i^*\left(\sum^*_{j:j\neq i}\frac{1}{y_i-y_j}\right)\cdot\left(\frac{\rd b_i}{\sqrt{2N}}+\frac{1}{2N}\sum^*_{j:j\neq i}\frac{\rd t}{y_i-y_j}\right)
-\frac{1}{2}\frac{1}{2N}\sum^*_{i,j}\frac{\rd\langle b_i,b_j\rangle}{(y_i-y_j)^2}
+\frac{1}{2}\frac{1}{2N}\sum^*_{i}\sum^*_{j:j\neq i}\frac{\rd\langle b_i\rangle}{(y_i-y_j)^2}\\
&=\rd M_t+\frac{1}{2N}\left(-\sum_i^*\left(\sum^*_{j:j\neq i}\frac{1}{y_i-y_j}\right)^2\rd t
-\frac{1}{2}\sum^*_{i,j}\frac{\rd\langle b_i,b_j\rangle}{(y_i-y_j)^2}
+\frac{1}{2}\sum^*_{i,j}\frac{\rd\langle b_i\rangle}{(y_i-y_j)^2}\right),
\end{align*}
where $M$ is a local martingale.  Note that by Cauchy-Schwarz we have
\[
|\rd\langle b_i,b_j\rangle/\rd t|\leq (\rd\langle b_i\rangle/\rd t)^{1/2} (\rd\langle b_j\rangle/\rd t)^{1/2}\leq 1,
\]
so that the above parenthesis is smaller than ($\sum^*_{i,j,k}$ below means summation over all $-N\leq i,j,k\leq N$ such that $i,j,k$ are all distinct and non-zero)
\[
-\sum_i^*\left(\sum^*_{j:j\neq i}\frac{1}{y_i-y_j}\right)^2
+\sum^*_{i,j}\frac{1}{(y_i-y_j)^2}=-\sum^*_{i,j,k}\frac{1}{(y_i-y_j)(y_i-y_k)}=0,
\]
where the last equality relies on $\tfrac{1}{(a-b)(a-c)}+\tfrac{1}{(b-a)(b-c)}+\tfrac{1}{(c-a)(c-b)}=0$. This shows that $(\phi(\by_t))_{t\geq 0}$ is a supermartingale. 

Let $A=\{\tau_\e\leq \mu_C,\tau_\e\leq T\}$ where $\e,C$ are chosen so that $\min_{i\neq j}|x_i(0)-x_j(0)|\geq \varepsilon$,  and $\max_{i}|x_i(0)|\leq C$. Then, using the supermartingale property for the first inequality below, we have
\begin{multline*}
\E[\phi(\by_0)]\geq\E[\phi(\by_{T\wedge\mu_C})]=\E[\phi(\by_{T\wedge\mu_C})\mathds{1}_{A}]+\E[\phi(\by_{T\wedge\mu_C})\mathds{1}_{A^{\rm c}}]\\
\geq (2(-\log \e)+(2N(2N-1)-2)(-\log (2C)))\mathbb{P}(A)+2N(2N-1)(-\log (2C))(1-\mathbb{P}(A)).
\end{multline*}
This implies $\mathbb{P}(A)\to 0$ as $\e\to0$, so $\mathbb{P}(\tau_0\leq \mu_C,\tau_0\leq T)=0$. The proof will therefore be complete if $\mu_C\to \infty$ a.s.  as $C\to\infty$.  Let $Z_t=\sum_i z_i(t)^2$ where $z_i(t)=x_i(t\wedge\mu_C)$, then It{\^o}'s formula gives
\[
\rd Z_t= 2\sum _iz_i \rd z_i+\sum_i \rd\langle z_i\rangle=2\sum_i z_i\frac{\rd b_i}{\sqrt{2N}}+N(2N-1)+\frac{1}{2N}\sum_i \rd\langle b_i\rangle,
\]
where we have used $\sum^*_{i,j}\frac{z_i}{z_i-z_j}=N(2N-1)$.  With our assumption on $\frac{\rd \langle b_i\rangle_t}{\rd t}$, this implies $\E(Z_t)\leq Z_0+a t$ for any $t$ and some constant $a$ depending on $N$. We now define the martingale $M_t=\int_0^t2\sum_i z_i\frac{\rd b_i}{\sqrt{2N}}$.  With the Burkholder-Davis-Gundy inequality, we have 
\[
\E[\sup_{0\leq u\leq t}|M_u|^2]\leq C_1\E[\langle M\rangle_t]\leq C_2\int_0^t\sum_{i,j}\E[|z_iz_j\rd\langle b_i,b_j\rangle|]\leq C_3 \int_0^t\sum_{i,j}\E[|z_i|^2+|z_j|^2]\rd t\leq C_4\int_0^t\E[Z_t]\rd t\leq C_5+C_6 t
\]
where $C_1$  is universal and the other constants $C_i$ above and below depend only on $n$.  Note that $|Z_u|\leq |Z_0|+|M_u|+C_7 u$, so from the above we obtain
$
\E[\sup_{0\leq u\leq t}|Z_u|^2]\leq C_8+C_9 t,
$
and in particular 
\[
\mathbb{P}(\mu_C\leq t)=\mathbb{P}(\max_{0\leq u\leq t} |x_i(u)|\geq C)=\mathbb{P}(\max_{0\leq u\leq t} |z_i(u)|\geq C)\leq \frac{C_8+C_9 t}{C^2}.
\]
Taking $C\to\infty$ we obtain $\mathbb{P}(\mu_{\infty}\leq t)=0$ for any fixed $t\geq 0$, so that $\mathbb{P}(\mu_{\infty}=\infty)=1$. This concludes the proof.
\end{proof}

\tableofcontents

\end{document}